\documentclass[10pt,a4paper]{article}

\usepackage{geometry}
\usepackage{amsmath,amssymb,amsthm,mathtools}
\usepackage{enumitem}
\usepackage{tikz-cd}
\usetikzlibrary{arrows.meta,calc,positioning}
\colorlet{wqteal}{teal!70!black}
\colorlet{wqorange}{orange!82!black}
\colorlet{wqblue}{blue!68!black}

\tikzset{
 wqstage/.style={rounded corners=2pt,draw=black!24,line width=0.48pt},
 wqqbox/.style={rounded corners=2pt,draw=teal!45!black,line width=0.58pt,
   fill=teal!2},
 wqqboxtitle/.style={font=\scriptsize,anchor=west,fill=white,
   inner sep=0.9pt,text=teal!60!black},
 wqweaveone/.style={draw=black!36,line width=0.58pt,
   line cap=round,line join=round},
 wqweavetwo/.style={wqweaveone,densely dashed},
 wqweavedisplay/.style={line width=0.82pt},
 wqweaveplus/.style={draw=wqteal},
 wqweaveminus/.style={draw=wqorange},
 wqcut/.style={draw=black!22,line width=0.45pt,densely dotted},
 wqguide/.style={draw=black!28,line width=0.45pt},
 wqcycleedge/.style={draw=wqblue,line width=1.05pt,line cap=round},
 wqstring/.style={draw=black!42,line width=0.82pt,line cap=round},
 wqletter/.style={circle,draw=wqorange,fill=orange!15,
   minimum size=3.1mm,inner sep=0pt,font=\tiny,text=black,
   line width=0.42pt},
 wqstringlabel/.style={font=\scriptsize,fill=white,inner sep=0.65pt,
   text=black!82},
 wqlevel/.style={font=\scriptsize,text=black!68},
 wqwordlabel/.style={font=\scriptsize,text=black!64},
 wqboundarylabel/.style={font=\scriptsize,inner sep=0.55pt,text=black!88},
 wqpanellabel/.style={font=\small,text=black!88},
 wqdiagramlabel/.style={font=\scriptsize,inner sep=0.7pt},
 wqbraid/.style={font=\scriptsize,fill=white,inner sep=0.6pt},
 wqhalfedge/.style={dashed,dash pattern=on 2.3pt off 1.6pt},
 wqstringedge/.style={draw=wqteal,line width=0.68pt},
 wqstringarrow/.style={wqstringedge,
   -{Stealth[length=1.65mm,width=1.18mm]}},
 wqstringhalf/.style={wqstringarrow,wqhalfedge},
 wqmixededge/.style={draw=wqorange,line width=0.76pt,
   line cap=round,line join=round},
 wqmixedarrow/.style={wqmixededge,
   -{Stealth[length=1.65mm,width=1.18mm]}},
 wqlocaledge/.style={draw=wqblue,line width=0.82pt,
   line cap=round,line join=round},
 wqlocalarrow/.style={wqlocaledge,
   -{Stealth[length=1.65mm,width=1.18mm]}},
 wqsharedarrow/.style={wqhalfedge,draw=wqorange,line width=1.58pt,
   line cap=round,line join=round,
   -{Stealth[length=1.95mm,width=1.42mm]},
   postaction={draw=wqblue,line width=0.64pt,wqhalfedge,
     -{Stealth[length=1.52mm,width=1.04mm]}}},
 wqcancelorange/.style={wqhalfedge,draw=wqorange,line width=0.72pt,
   line cap=round,line join=round,
   -{Stealth[length=1.48mm,width=1.04mm]}},
 wqcancelblue/.style={wqhalfedge,draw=wqblue,line width=0.72pt,
   line cap=round,line join=round,
   -{Stealth[length=1.48mm,width=1.04mm]}},
 wqtimearrow/.style={-{Stealth[length=2.4mm,width=1.7mm]},
   draw=black!72,line width=0.92pt},
 wqmaparrow/.style={-{Stealth[length=2mm,width=1.4mm]},
   draw=black!70,line width=0.75pt},
 wqqvertex/.style={draw=wqteal,fill=white,line width=0.64pt},
 wqcompactqvertex/.style={minimum width=8.0mm,minimum height=4.4mm,
   inner xsep=0.8pt,inner ysep=0.15pt,font=\tiny},
 wqlargeqvertex/.style={minimum height=5.2mm,inner xsep=2.7pt,
   inner ysep=1.0pt,font=\scriptsize},
 wqfrozen/.style={wqqvertex,rectangle,fill=teal!6},
 wqmutable/.style={wqqvertex,rounded corners=2.1mm,fill=teal!6},
 wqstringfrozen/.style={wqfrozen,wqcompactqvertex},
 wqstringmutable/.style={wqmutable,wqcompactqvertex},
 wqpivot/.style={wqmutable,wqcompactqvertex,fill=yellow!32,
   line width=0.86pt},
 wqlargefrozen/.style={wqfrozen,wqlargeqvertex},
 wqlargemutable/.style={wqmutable,wqlargeqvertex},
 wqlargepivot/.style={wqlargemutable,fill=yellow!32,
   line width=0.86pt},
 wqamalgamatedmutable/.style={wqmutable,wqlargeqvertex,
   draw=wqorange,fill=orange!15,minimum width=6.8mm},
 wqcontraction/.style={circle,draw=black!86,fill=blue!10,
   minimum size=2.0mm,inner sep=0pt,line width=0.58pt},
 wqweavefrozen/.style={rectangle,draw=black!86,fill=blue!10,
   minimum size=2.2mm,inner sep=0pt,line width=0.58pt},
 wqvlabel/.style={font=\scriptsize,fill=white,inner sep=0.45pt,
   text=black!82},
 wqhalf/.style={fill=teal!2,inner sep=0.35pt,font=\tiny,text=wqteal},
 wqhalfblue/.style={fill=white,inner sep=0.32pt,font=\tiny,text=wqblue},
 wqhalforange/.style={fill=white,inner sep=0.32pt,font=\tiny,text=wqorange},
 wqhalfsum/.style={fill=white,inner sep=0.42pt,font=\tiny,text=black!82},
 wqweightlabel/.style={fill=white,inner sep=0.32pt,font=\tiny,
   text=wqorange,outer sep=0pt,above=0.15mm},
 wqmovelabel/.style={fill=white,inner sep=1.3pt,align=left,font=\scriptsize}
}

\NewDocumentCommand{\wqextensiontriangle}{O{} O{} m m m}{%
  \draw[wqmixedarrow] (#3) to[#1] (#4);
  \draw[wqmixedarrow] (#4) to[#2] (#5);
}

\usepackage[unicode,hidelinks]{hyperref}

\usepackage[
  backend=biber,
  style=alphabetic,
  sorting=anyt,
  maxalphanames=6,
  minalphanames=6,
  maxbibnames=99,
  giveninits=true,
  doi=false,
  isbn=false,
  url=false,
]{biblatex}
\DeclareNameAlias{sortname}{given-family}

\renewcommand*{\intitlepunct}{\addspace}
\DefineBibliographyStrings{english}{in = {in}}

\DeclareFieldFormat*{title}{\mkbibemph{#1}}
\DeclareFieldFormat*{booktitle}{\mkbibemph{#1}}
\DeclareFieldFormat*{maintitle}{\mkbibemph{#1}}
\DeclareFieldFormat{journaltitle}{\mkbibemph{#1}}
\DeclareFieldFormat[article]{volume}{\textbf{#1}}
\DeclareFieldFormat[article]{number}{no.~#1}
\DeclareFieldFormat[incollection,inbook]{number}{\textbf{#1}}
\DeclareFieldFormat{pages}{#1}
\DeclareFieldFormat{eprint:arxiv}{arXiv:#1}

\renewbibmacro*{in:}{%
  \ifentrytype{article}{}{\printtext{\bibstring{in}\intitlepunct}}}

\renewbibmacro*{journal+issuetitle}{%
  \usebibmacro{journal}%
  \setunit*{\addspace}%
  \printfield{volume}%
  \setunit*{\addspace}%
  \printtext[parens]{\printfield{year}}%
  \setunit{\addcomma\space}%
  \printfield{number}%
  \setunit{\addcomma\space}%
  \printfield{eid}%
  \newunit}
\renewbibmacro*{issue+date}{}

\DeclareBibliographyDriver{misc}{%
  \usebibmacro{bibindex}%
  \usebibmacro{begentry}%
  \usebibmacro{author/editor+others/translator+others}%
  \setunit{\labelnamepunct}\newblock
  \usebibmacro{title}%
  \newunit\newblock
  \printfield{howpublished}%
  \newunit\newblock
  \usebibmacro{eprint}%
  \newunit\newblock
  \printfield{year}%
  \newunit\newblock
  \printfield{note}%
  \usebibmacro{finentry}}

\AtEveryBibitem{%
  \clearfield{eprintclass}%
  \clearfield{pagetotal}%
  \ifentrytype{misc}{}{\clearfield{eprint}\clearfield{eprinttype}}%
  \ifentrytype{incollection}{\clearname{editor}}{}%
}

\hypersetup{
  pdftitle={Mutation Sequences along Weaves and Amalgamation of Braid Varieties},
  pdfauthor={Yuma Mizuno}
}

\newtheorem{theorem}{Theorem}[section]
\newtheorem{proposition}[theorem]{Proposition}
\newtheorem{lemma}[theorem]{Lemma}
\newtheorem{corollary}[theorem]{Corollary}
\theoremstyle{definition}
\newtheorem{definition}[theorem]{Definition}
\newtheorem{example}[theorem]{Example}
\theoremstyle{remark}
\newtheorem{remark}[theorem]{Remark}

\newcommand{\C}{\mathbb C}
\newcommand{\Gm}{\mathbb G_m}
\newcommand{\Aplus}{\mathcal A_+}
\newcommand{\Aminus}{\mathcal A_-}
\newcommand{\Bplus}{\mathcal B_+}
\newcommand{\Bminus}{\mathcal B_-}
\newcommand{\flagA}{\mathsf{A}}
\newcommand{\flagB}{\mathsf{B}}
\newcommand{\flagD}{\mathsf{D}}
\newcommand{\groupG}{\mathsf{G}}
\newcommand{\borelB}{\mathsf{B}}
\newcommand{\unipotentU}{\mathsf{U}}
\newcommand{\torusH}{\mathsf{H}}
\newcommand{\Conf}{\operatorname{Conf}}
\newcommand{\Confhalf}{\operatorname{Conf}_{\mathrm{half}}}
\newcommand{\Tcal}{\mathcal T}
\newcommand{\Dcal}{\mathcal D}
\newcommand{\Ucal}{\mathcal U}
\newcommand{\Vcal}{\mathcal V}
\newcommand{\eps}{\varepsilon}
\newcommand{\fW}{\mathfrak W}
\newcommand{\rind}[1]{\overrightarrow{\mathfrak w}(#1)}
\newcommand{\lind}[1]{\overleftarrow{\mathfrak w}(#1)}
\newcommand{\pos}{\operatorname{pos}}
\newcommand{\Split}{\operatorname{split}}
\newcommand{\id}{\operatorname{id}}
\newcommand{\len}{\operatorname{len}}
\newcommand{\Br}{\operatorname{Br}}
\newcommand{\Word}{\operatorname{Word}}
\newcommand{\bigast}{\mathop{\vcenter{\hbox{\Large\(\ast\)}}}\displaylimits}

\title{Mutation Sequences along Weaves and Amalgamation of Braid Varieties}
\author{Yuma Mizuno}
\date{}

\begin{document}

\maketitle

\begin{abstract}
Let \(p,q\) be positive braids with Demazure products \(u,v\).  The endpoint
stratum \(\operatorname{Conf}(p,q)_{u,v}\) of the double Bott--Samelson cell
splits as \(\operatorname{Conf}(u,v)\times X(p)\times X(q^{\mathrm{op}})\), a
double Bruhat cell times two braid varieties.  We prove that the stratum's cluster
structure is given by a cluster localization of the one on \(\operatorname{Conf}(p,q)\), 
and that the splitting map is a quasi-cluster isomorphism.  This comes from a
mutation sequence along a double Demazure weave, one mutation per trivalent
vertex, ending at an amalgamation of an extension of the double word quiver 
with the weave quiver.  In the half-decorated case, this proves the conjecture of
Gorsky--Kim--Scroggin--Simental for the splicing map
\(X(p)\times X(\Delta\Delta)\to X(p\Delta)\), where \(\Delta\) is a reduced
positive braid for \(w_0\) and \(\operatorname{dem}(p)=w_0\).
\end{abstract}

\section{Introduction}

\subsection{Background}

Let \(\groupG\) be a connected simply connected semisimple algebraic group of
finite type over \(\C\).  There are two kinds of varieties determined by positive braids,
both of which admit cluster structures.
One is a double Bott--Samelson cell, and the
other is a braid variety.  Both are moduli spaces parametrizing flag sequences
whose relative positions vary according to a positive braid, but their
conditions are imposed differently.

The first variety, the double Bott--Samelson cell \(\Conf(p,q)\), is a space
defined for a pair \(p,q\) of positive braids and was introduced by
Shen--Weng \cite{ShenWeng2021}.  Double Bott--Samelson cells carry cluster
structures that generalize the cluster structures on double Bruhat cells
\cite{BerensteinFominZelevinsky2005}.
Their cluster coordinates are described by combinatorial objects expressed as triangulations 
(or equivalently, strings), and their cluster variables correspond to generalized minors.  
As a notable application, by constructing finite-order automorphisms of undecorated
double Bott--Samelson cells, Shen--Weng gave a geometric proof of the
Zamolodchikov periodicity conjecture.

The second variety, the braid variety \(X(p)\), is a space defined for a positive
braid \(p\).  Braid varieties form a wide class: open Richardson varieties and
half-decorated double Bott--Samelson cells with one word empty both arise as
special cases, the latter from braids of the form \(p\Delta\), where \(\Delta\)
is a positive braid lift of \(w_0\).  The geometric significance of braid
varieties lies in their connection between the algebraic geometry of flag
configurations and contact--symplectic geometry.  In type \(A\), an appropriate
torus quotient of the braid variety \(X(p\Delta)\), a half-decorated double
Bott--Samelson cell, is isomorphic to the augmentation variety of the
corresponding Legendrian link and carries a holomorphic symplectic structure.
Moreover, the maximal-dimensional toric charts associated with Demazure weaves
are exponential Darboux charts for this holomorphic symplectic form
\cite{CasalsGorskyGorskySimental2020}.  Gao--Shen--Weng identified
\(X(p\Delta)\) itself with an augmentation variety for a different choice of
marked points, and showed that it is a cluster \(K_2\) variety whose seeds are
given by exact Lagrangian fillings \cite{GaoShenWeng2024}.
More recently, Asplund--Capovilla-Searle--Hughes--Leverson--Li--Wu extended 
this contact-geometric interpretation beyond the \(X(p\Delta)\) case: 
in type \(A\), when the Demazure product of \(p\) is \(w_0\), the braid variety 
\(X(p)\) is isomorphic to the augmentation variety of the Legendrian \((-1)\)-closure 
of \(p\Delta\) \cite{AsplundEtAlDecompositions2025}.

The diagrammatic calculus that describes this symplectic structure is the
calculus of weaves.  Interpreting a Demazure weave as a Legendrian weave and
appropriately closing its terminal end yields exact Lagrangian fillings
\cite{CasalsZaslow2022,CasalsGorskyGorskySimental2020}.  In the Casals--Weng
microlocal model for certain Legendrian links, rank-one local systems on a
filling give a seed torus, while the intersection pairing of cycles and
Lagrangian surgery realize the exchange matrix and cluster mutation,
respectively \cite{CasalsWeng2024}.  Algebraically as well, the Lusztig cycles
on a Demazure weave and flag propagation construct cluster coordinates on
braid varieties \cite{CasalsEtAlBraidVarieties}.  Cluster structures on braid
varieties were also constructed by different methods, using 3D plabic graphs in
type \(A\) \cite{GalashinLamShermanBennettSpeyer2022} and Deodhar geometry in
general \cite{GalashinLamShermanBennett2023}, and the two approaches are
compared in \cite{CasalsEtAlComparing2025}.

\subsection{Main results}

In this paper, we relate the cluster structures on double Bott--Samelson cells
and on braid varieties as follows: an open stratum of a double Bott--Samelson
cell carries a cluster structure obtained from that of the whole cell by a
cluster localization, and on this stratum the cell decomposes, quasi-cluster
isomorphically, into the double Bott--Samelson cell of the Demazure products
and a product of braid varieties.

Let \(p,q\) be positive braids, and let \(u:=\operatorname{dem}(p)\) and
\(v:=\operatorname{dem}(q)\) be their respective Demazure products.  Let
\(\Conf(p,q)\) be the double Bott--Samelson cell
(Definition~\ref{def:double-bs-cell}), and let \(\Conf(p,q)_{u,v}\) be its open
subvariety on which the relative positions between the two ends of the plus-side
and minus-side flag sequences are \(u\) and \(v\), respectively
(Definition~\ref{def:endpoint-strata}).  Consider the map that replaces the
flag sequences by reduced flag sequences:
\[
 \operatorname{red}:
 \Conf(p,q)_{u,v}
 \longrightarrow \Conf(u,v).
\]
Although not used below, we note that \(\Conf(u,v)\) is isomorphic to the double Bruhat cell
\(G^{u,v}\).  For each
\(c\in\Conf(u,v)\), the fiber \(\operatorname{red}^{-1}(c)\) parametrizes the
flag sequences that reduce to \(c\), and is naturally isomorphic to a product
of braid varieties:
\[
 \operatorname{red}^{-1}(c)
 \xrightarrow{\sim}
 X(p)\times X(q^{\mathrm{op}}).
\]
See Definition~\ref{def:braid-variety} for the definition of braid varieties.
In fact, as stated in Theorem~\ref{thm:flag-product}, there is an isomorphism
\begin{equation}
 \Conf(p,q)_{u,v}
 \xrightarrow{\sim}
 \Conf(u,v)\times X(p)\times X(q^{\mathrm{op}}).
 \tag{\(\ast\)}
 \label{eq:intro-split}
\end{equation}

Our main result is that this isomorphism preserves the cluster structures.

\begin{theorem}[\(=\)~Corollary~\ref{cor:split-quasi-cluster} and Theorem~\ref{thm:endpoint-cluster-localization}]
\label{thm:main}
There exist a seed \(\Sigma\) for the cluster structure on \(\Conf(p,q)\) and a subset \(S\) of its
mutable vertices satisfying the following properties.
\begin{enumerate}[label=\textup{(\arabic*)}]
  \item Write \(\operatorname{freeze}_S(\Sigma)\) for the seed obtained by
  freezing \(S\) in \(\Sigma\).  Its
  cluster algebra and 
  upper cluster algebra are isomorphic to the coordinate ring of
  \(\Conf(p,q)_{u,v}\), that is,
  \[
  \mathcal A\bigl(\operatorname{freeze}_S(\Sigma)\bigr) = 
  \mathcal U\bigl(\operatorname{freeze}_S(\Sigma)\bigr)
  \cong
  \mathcal O\bigl(\Conf(p,q)_{u,v}\bigr).
  \]
  \item This freezing is a cluster localization, that is,
  \[
  \mathcal A\bigl(\operatorname{freeze}_S(\Sigma)\bigr)
  \cong
  \mathcal A(\Sigma)\left[A_s^{-1}\mid s\in S\right].
  \]
  Here \(A_s\) is the cluster variable corresponding to \(s\).
  \item
  The isomorphism in~\eqref{eq:intro-split} is a quasi-cluster
  isomorphism between the cluster structure on \(\Conf(p,q)_{u,v}\) given by
  \(\operatorname{freeze}_S(\Sigma)\) and 
  the product cluster structure on \(\Conf(u,v)\times X(p)\times X(q^{\mathrm{op}})\).
  Moreover, under this isomorphism, the cluster variables of
  \(X(p)\) and \(X(q^{\mathrm{op}})\) pull back to cluster variables of
  \(\operatorname{freeze}_S(\Sigma)\).
\end{enumerate}
\end{theorem}

We outline the proof.  In the proof of the theorem, we need to determine the
seed \(\Sigma\) used in the cluster localization and the set \(S\) of vertices
to be frozen.  The core of the proof is that these are obtained from a sweep
mutation sequence along a weave.  Let \(\beta\) be a double word with \(p\) as
its plus part and \(q\) as its minus part.  A double word determines a quiver 
\(Q(\beta)\) describing the cluster coordinates on
\(\Conf(p,q)\).  Also, let \(\delta\) be a double word with \(u\) as its plus
part and \(v\) as its minus part.  Since the Demazure product of \(\beta\) is
\(\delta\), one can connect \(\beta\) and \(\delta\) by repeatedly applying
braid moves and contractions \(ii\to i\).  Such a sequence is represented by a
(double) Demazure weave \(\fW:\beta\to\delta\).  For each weave
\(\fW:\beta\to\delta\), we construct the \emph{sweep mutation sequence} along \(\fW\):
\[
  \mu_\fW:
  Q(\beta)
  \longrightarrow Q^{\mathrm{sweep}}(\fW).
\]
The construction is defined inductively by tracing the weave from top to
bottom.  If the weave is extended by a braid move, the corresponding mutation
sequence is the mutation sequence used by Shen--Weng to replace the word.  The
case of a contraction is the new construction in this paper.  In this case, we
mutate the middle vertex of \(Q(ii)\) and ``exclude'' that vertex.  The quiver
obtained is \(Q(i)\) with the excluded vertex attached.  This construction is
illustrated in Figure~\ref{fig:contraction-quiver-mutation}.  Thus, after all
mutations along the weave are performed, one obtains
\(Q^{\mathrm{sweep}}(\fW)\), the quiver obtained from \(Q(\delta)\) by
attaching the quiver formed by the excluded vertices.  It follows from the
construction that the excluded vertices are parametrized by the trivalent
vertices of the weave.  The most important combinatorial fact established in
this paper is that, in fact, after appropriately correcting the entries
between frozen vertices, the quiver formed by these excluded vertices is the
weave quiver \(Q(\fW)\), which describes the cluster coordinates on braid
varieties~\cite{CasalsEtAlBraidVarieties}. See Section~\ref{subsec:terminal-amalgamation}.

The seed to be given in the main theorem is obtained by combining the
braid-variety seed given by \(Q(\fW)\) with the double Bott--Samelson-cell seed
given by \(Q(\delta)\), after the latter is corrected by monomials determined by
the boundary weights of the Lusztig cycles on the weave.  The set \(S\) to be
frozen is \(\Dcal(\fW)\), the set of trivalent vertices of \(\fW\) whose
Lusztig cycles have nonzero terminal weight.  See
Definitions~\ref{def:frozen-quiver-amalgamation}
and~\ref{def:amalgamated-functions}, and Theorem~\ref{thm:quiver}.

\begin{example}[A sweep mutation sequence along a type \(A_2\) weave]
\label{subsec:a2-mutation-term-example}

Let \(\beta=1122112\) and \(\delta=212\), and
consider the following Demazure weave of type \(A_2\):
\begin{equation}
 \fW:
 1122112\xrightarrow{\,11\to1\,}122112
 \xrightarrow{\,22\to2\,}12112
 \xrightarrow{\,11\to1\,}1212
 \xrightarrow{\,121\to212\,}2122
 \xrightarrow{\,22\to2\,}212.
 \label{eq:a2-direct-weave}
\end{equation}
The sweep mutation sequence \(\mu_\fW\) is illustrated in
Figures~\ref{fig:a2-five-step-mutation-early}
and~\ref{fig:a2-five-step-mutation-late}.  
Each quiver is an amalgamation of an extension of the word quiver below with
the weave quiver above.  The number of arrows used in the amalgamation is
determined by the terminal weights of the Lusztig cycles.
\end{example}

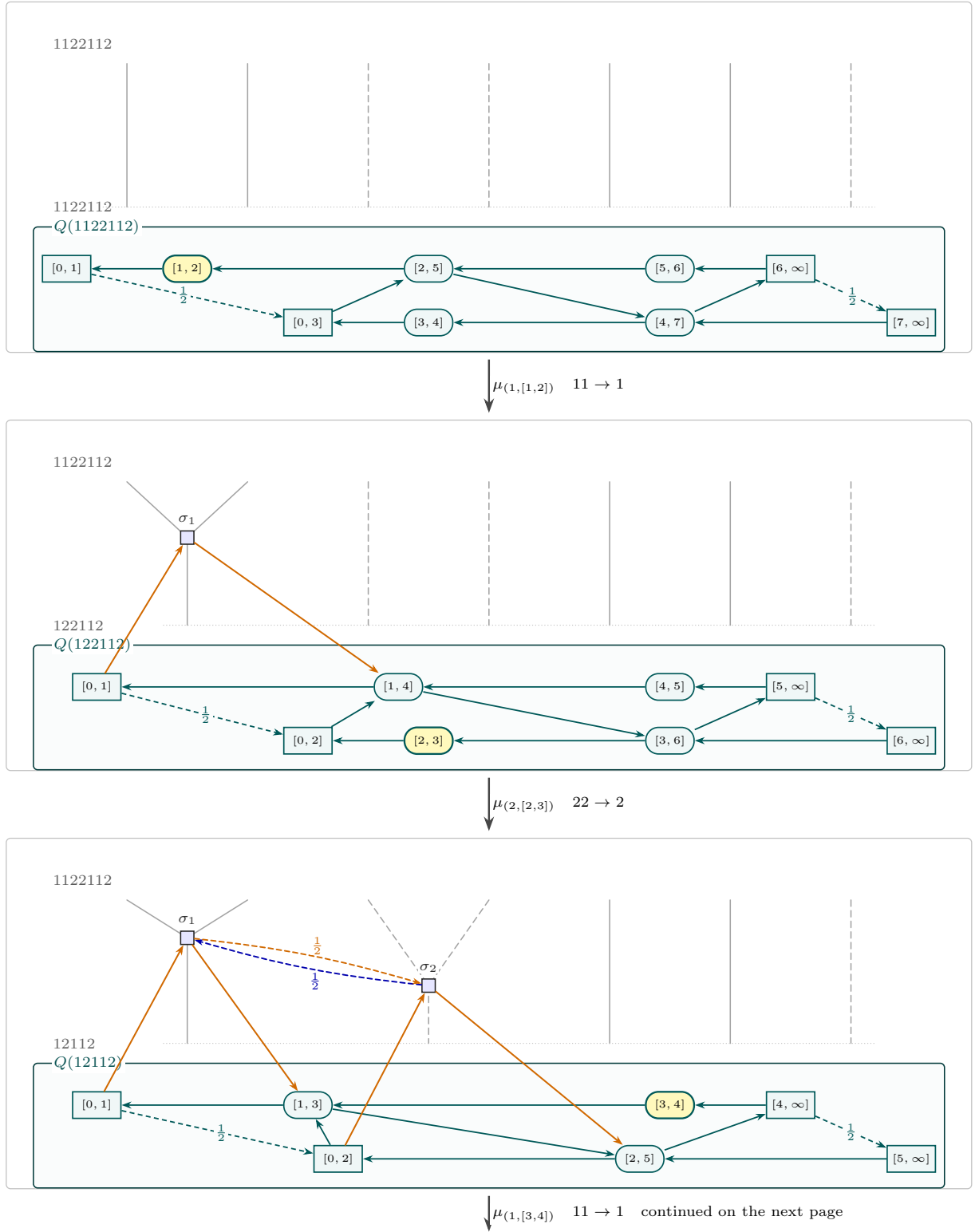
\begin{figure}[p]
\centering
\begin{tikzpicture}[x=1.00cm,y=0.66cm]
\path[use as bounding box] (0,-1.40) rectangle (16,29.95);

\begin{scope}[shift={(0,21.00)}]
\draw[wqstage] (0,0) rectangle (16,8.80);
\node[wqwordlabel,anchor=west] at (0.65,7.75) {$1122112$};

\coordinate (a10pos) at (1.00,2.10);
\coordinate (a11pos) at (3.00,2.10);
\coordinate (a12pos) at (7.00,2.10);
\coordinate (a13pos) at (11.00,2.10);
\coordinate (a14pos) at (13.00,2.10);
\coordinate (a20pos) at (5.00,0.75);
\coordinate (a21pos) at (7.00,0.75);
\coordinate (a22pos) at (11.00,0.75);
\coordinate (a23pos) at (15.00,0.75);
\coordinate (a1endone) at (2.00,3.65);
\coordinate (a1endtwo) at (4.00,3.65);
\coordinate (wqendone) at (6.00,3.65);
\coordinate (wqendtwo) at (8.00,3.65);
\coordinate (a1endthree) at (10.00,3.65);
\coordinate (a1endfour) at (12.00,3.65);
\coordinate (wqendthree) at (14.00,3.65);

\draw[wqweaveone] (2.00,7.25)--(a1endone);
\draw[wqweaveone] (4.00,7.25)--(a1endtwo);
\draw[wqweavetwo] (6.00,7.25)--(wqendone);
\draw[wqweavetwo] (8.00,7.25)--(wqendtwo);
\draw[wqweaveone] (10.00,7.25)--(a1endthree);
\draw[wqweaveone] (12.00,7.25)--(a1endfour);
\draw[wqweavetwo] (14.00,7.25)--(wqendthree);
\draw[wqcut] (1.60,3.65)--(14.40,3.65);
\node[wqwordlabel,anchor=west] at (0.65,3.65) {$1122112$};

\draw[wqqbox] (0.45,0.02) rectangle (15.55,3.15);
\node[wqqboxtitle] at (0.75,3.15)
  {$Q(1122112)$};
\node[wqstringfrozen] (a10) at (a10pos) {$[0,1]$};
\node[wqpivot]        (a11) at (a11pos) {$[1,2]$};
\node[wqstringmutable](a12) at (a12pos) {$[2,5]$};
\node[wqstringfrozen] (a20) at (a20pos) {$[0,3]$};
\node[wqstringmutable](a21) at (a21pos) {$[3,4]$};
\node[wqstringmutable](a22) at (a22pos) {$[4,7]$};
\node[wqstringmutable](a13) at (a13pos) {$[5,6]$};
\node[wqstringfrozen] (a14) at (a14pos) {$[6,\infty]$};
\node[wqstringfrozen] (a23) at (a23pos) {$[7,\infty]$};
\draw[wqstringarrow] (a11)--(a10);
\draw[wqstringhalf] (a10)--(a20);
\draw[wqstringarrow] (a12)--(a11);
\draw[wqstringarrow] (a13)--(a12);
\draw[wqstringarrow] (a20)--(a12);
\draw[wqstringarrow] (a12)--(a22);
\draw[wqstringarrow] (a14)--(a13);
\draw[wqstringarrow] (a22)--(a14);
\draw[wqstringhalf] (a14)--(a23);
\draw[wqstringarrow] (a21)--(a20);
\draw[wqstringarrow] (a22)--(a21);
\draw[wqstringarrow] (a23)--(a22);
\path (a10)--node[wqhalf,pos=0.48] {$\frac12$} (a20);
\path (a14)--node[wqhalf,pos=0.52] {$\frac12$} (a23);
\end{scope}

\draw[wqtimearrow] (8.00,20.82)--node[wqmovelabel,right]
  {$\mu_{(1,[1,2])}\quad 11\to1$} (8.00,19.48);

\begin{scope}[shift={(0,10.50)}]
\draw[wqstage] (0,0) rectangle (16,8.80);
\node[wqwordlabel,anchor=west] at (0.65,7.75) {$1122112$};

\coordinate (b10pos) at (1.50,2.10);
\coordinate (b11pos) at (6.50,2.10);
\coordinate (b12pos) at (11.00,2.10);
\coordinate (b13pos) at (13.00,2.10);
\coordinate (b20pos) at (5.00,0.75);
\coordinate (b21pos) at (7.00,0.75);
\coordinate (b22pos) at (11.00,0.75);
\coordinate (b23pos) at (15.00,0.75);
\coordinate (b1endone) at (3.00,3.65);
\coordinate (b2endone) at (6.00,3.65);
\coordinate (b2endtwo) at (8.00,3.65);
\coordinate (b1endtwo) at (10.00,3.65);
\coordinate (b1endthree) at (12.00,3.65);
\coordinate (b2endthree) at (14.00,3.65);

\coordinate (bvonepos) at (3.00,5.85);
\draw[wqweaveone] (2.00,7.25)--(bvonepos);
\draw[wqweaveone] (4.00,7.25)--(bvonepos);
\draw[wqweaveone] (bvonepos)--(b1endone);
\draw[wqweavetwo] (6.00,7.25)--(b2endone);
\draw[wqweavetwo] (8.00,7.25)--(b2endtwo);
\draw[wqweaveone] (10.00,7.25)--(b1endtwo);
\draw[wqweaveone] (12.00,7.25)--(b1endthree);
\draw[wqweavetwo] (14.00,7.25)--(b2endthree);
  \node[wqweavefrozen] (bvone) at (bvonepos) {};
\draw[wqcut] (2.60,3.65)--(14.40,3.65);
\node[wqwordlabel,anchor=west] at (0.65,3.65) {$122112$};

\draw[wqqbox] (0.45,0.02) rectangle (15.55,3.15);
\node[wqqboxtitle] at (0.75,3.15)
  {$Q(122112)$};
\node[wqstringfrozen] (b10) at (b10pos) {$[0,1]$};
\node[wqstringmutable](b11) at (b11pos) {$[1,4]$};
\node[wqstringfrozen] (b20) at (b20pos) {$[0,2]$};
\node[wqpivot]        (b21) at (b21pos) {$[2,3]$};
\node[wqstringmutable](b22) at (b22pos) {$[3,6]$};
\node[wqstringmutable](b12) at (b12pos) {$[4,5]$};
\node[wqstringfrozen] (b13) at (b13pos) {$[5,\infty]$};
\node[wqstringfrozen] (b23) at (b23pos) {$[6,\infty]$};
\draw[wqstringarrow] (b11)--(b10);
\draw[wqstringhalf] (b10)--(b20);
\draw[wqstringarrow] (b12)--(b11);
\draw[wqstringarrow] (b20)--(b11);
\draw[wqstringarrow] (b11)--(b22);
\draw[wqstringarrow] (b13)--(b12);
\draw[wqstringarrow] (b22)--(b13);
\draw[wqstringhalf] (b13)--(b23);
\draw[wqstringarrow] (b21)--(b20);
\draw[wqstringarrow] (b22)--(b21);
\draw[wqstringarrow] (b23)--(b22);

\wqextensiontriangle{b10}{bvone}{b11}
\path (b10)--node[wqhalf,pos=0.52] {$\frac12$} (b20);
\path (b13)--node[wqhalf,pos=0.50] {$\frac12$} (b23);
\node[wqvlabel,above=0.8mm of bvone] {$\sigma_1$};
\end{scope}

\draw[wqtimearrow] (8.00,10.32)--node[wqmovelabel,right]
  {$\mu_{(2,[2,3])}\quad 22\to2$} (8.00,8.98);

\begin{scope}
\draw[wqstage] (0,0) rectangle (16,8.80);
\node[wqwordlabel,anchor=west] at (0.65,7.75) {$1122112$};

\coordinate (c10pos) at (1.50,2.10);
\coordinate (c11pos) at (5.00,2.10);
\coordinate (c12pos) at (11.00,2.10);
\coordinate (c13pos) at (13.00,2.10);
\coordinate (c20pos) at (5.50,0.75);
\coordinate (c21pos) at (10.50,0.75);
\coordinate (c22pos) at (15.00,0.75);
\coordinate (c1endone) at (3.00,3.65);
\coordinate (c2endone) at (7.00,3.65);
\coordinate (c1endtwo) at (10.00,3.65);
\coordinate (c1endthree) at (12.00,3.65);
\coordinate (c2endtwo) at (14.00,3.65);

\coordinate (cvonepos) at (3.00,6.30);
\coordinate (cvtwopos) at (7.00,5.10);
\draw[wqweaveone] (2.00,7.25)--(cvonepos);
\draw[wqweaveone] (4.00,7.25)--(cvonepos);
\draw[wqweaveone] (cvonepos)--(c1endone);
\draw[wqweavetwo] (6.00,7.25)--(cvtwopos);
\draw[wqweavetwo] (8.00,7.25)--(cvtwopos);
\draw[wqweavetwo] (cvtwopos)--(c2endone);
\draw[wqweaveone] (10.00,7.25)--(c1endtwo);
\draw[wqweaveone] (12.00,7.25)--(c1endthree);
\draw[wqweavetwo] (14.00,7.25)--(c2endtwo);
  \node[wqweavefrozen] (cvone) at (cvonepos) {};
  \node[wqweavefrozen] (cvtwo) at (cvtwopos) {};
\draw[wqcut] (2.60,3.65)--(14.40,3.65);
\node[wqwordlabel,anchor=west] at (0.65,3.65) {$12112$};

\draw[wqqbox] (0.45,0.02) rectangle (15.55,3.15);
\node[wqqboxtitle] at (0.75,3.15)
  {$Q(12112)$};
\node[wqstringfrozen] (c10) at (c10pos) {$[0,1]$};
\node[wqstringmutable](c11) at (c11pos) {$[1,3]$};
\node[wqstringfrozen] (c20) at (c20pos) {$[0,2]$};
\node[wqstringmutable](c21) at (c21pos) {$[2,5]$};
\node[wqpivot]        (c12) at (c12pos) {$[3,4]$};
\node[wqstringfrozen] (c13) at (c13pos) {$[4,\infty]$};
\node[wqstringfrozen] (c22) at (c22pos) {$[5,\infty]$};
\draw[wqstringarrow] (c11)--(c10);
\draw[wqstringhalf] (c10)--(c20);
\draw[wqstringarrow] (c12)--(c11);
\draw[wqstringarrow] (c20)--(c11);
\draw[wqstringarrow] (c11)--(c21);
\draw[wqstringarrow] (c13)--(c12);
\draw[wqstringarrow] (c21)--(c13);
\draw[wqstringhalf] (c13)--(c22);
\draw[wqstringarrow] (c21)--(c20);
\draw[wqstringarrow] (c22)--(c21);

  \wqextensiontriangle{c10}{cvone}{c11}
  \wqextensiontriangle{c20}{cvtwo}{c21}
  \draw[wqcancelblue] (cvtwo) to[bend left=5]
    node[wqhalfblue,pos=0.46,below=0.25mm] {$\frac12$} (cvone);
  \draw[wqcancelorange] (cvone) to[bend left=5]
    node[wqhalforange,pos=0.54,above=0.25mm] {$\frac12$} (cvtwo);
\path (c10)--node[wqhalf,pos=0.52] {$\frac12$} (c20);
\path (c13)--node[wqhalf,pos=0.50] {$\frac12$} (c22);
\node[wqvlabel,above=0.8mm of cvone] {$\sigma_1$};
\node[wqvlabel,above=0.8mm of cvtwo] {$\sigma_2$};
\end{scope}

\draw[wqtimearrow] (8.00,-0.18)--node[wqmovelabel,right]
  {$\mu_{(1,[3,4])}\quad 11\to1$\quad continued on the next page} (8.00,-1.08);

\end{tikzpicture}
\caption{The sweep mutation sequence along the type \(A_2\) weave
\eqref{eq:a2-direct-weave}.  Square vertices in the weave indicate that they
are frozen in the weave quiver.}
\label{fig:a2-five-step-mutation-early}
\end{figure}

\begin{figure}[p]
\centering
\input{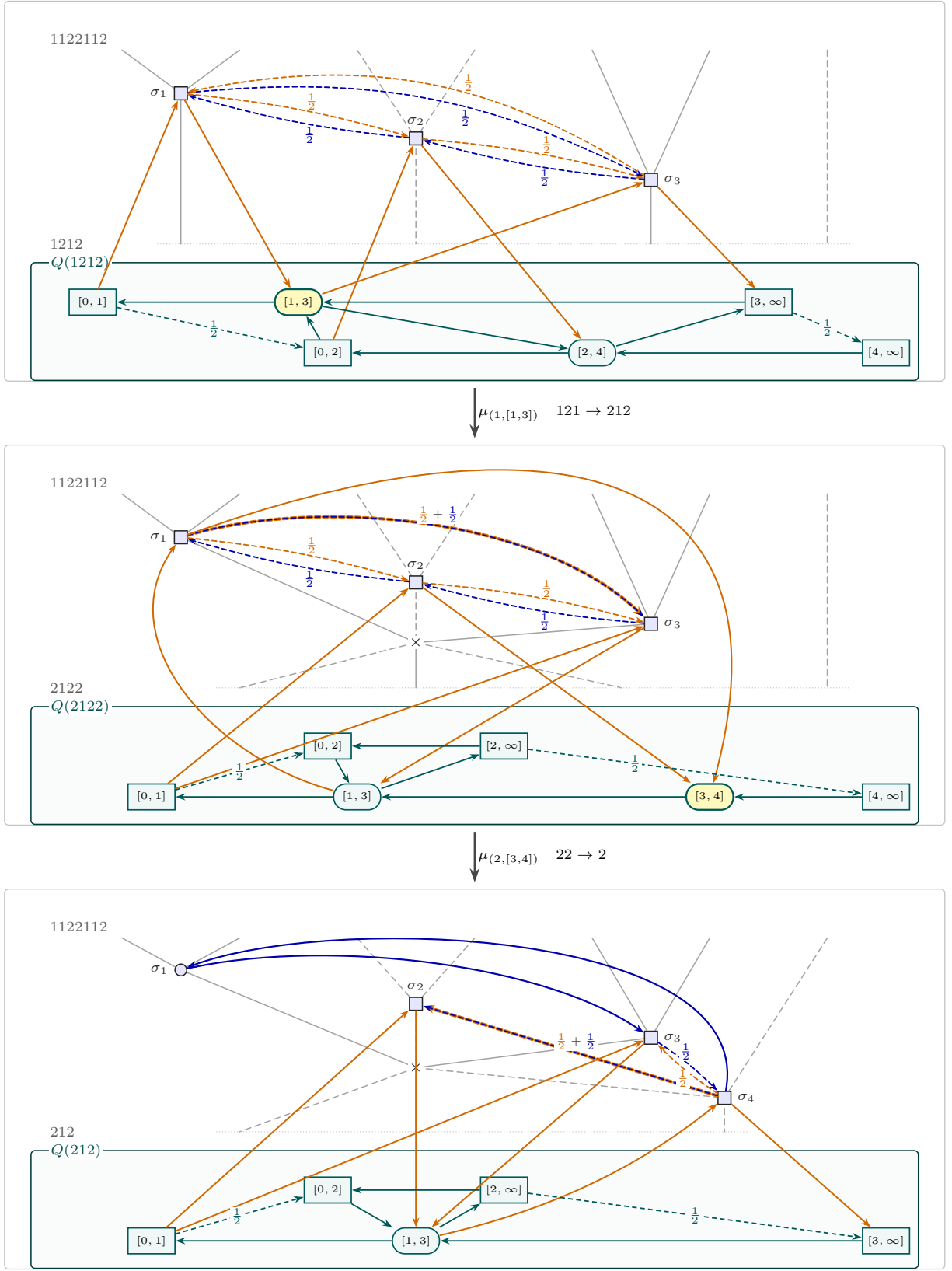}
\caption{Continuation of Figure~\ref{fig:a2-five-step-mutation-early}.}
\label{fig:a2-five-step-mutation-late}
\end{figure}

\subsection{Quasi-cluster property of the splicing map}

\begingroup
\emergencystretch=2em
For positive braids \(p,q\), consider the half-decorated double
Bott--Samelson cells \(\Confhalf^+(p,q)\) and \(\Confhalf^-(p,q)\)
(Definition~\ref{def:half-conf}).  The plus and minus signs indicate which side
retains its decoration.
\par
\endgroup
If \(\Delta\) is a reduced positive braid for \(w_0\), the following natural
isomorphisms are given in Proposition~\ref{prop:delta-extension}:
\[
 \Confhalf^+ (p,\varnothing)\xrightarrow{\sim}X(p \Delta),
 \qquad
 \Confhalf^- (\varnothing,p^{\mathrm{op}})
 \xrightarrow{\sim}X(\Delta p).
\]
When \(\operatorname{dem}(p)=w_0\), transporting the inverses of the
half-decorated flag-configuration decompositions in
Corollary~\ref{cor:half-one-sided-flag-product}
through these identifications gives the splicing maps constructed by
Gorsky--\allowbreak Kim--\allowbreak Scroggin--\allowbreak Simental
\cite[Theorem~1.1]{GorskyKimScrogginSimental2025}.

Gorsky--Kim--Scroggin--Simental conjectured that the splicing map is a
quasi-cluster isomorphism in general
\cite[Conjecture~1.4(c)]{GorskyKimScrogginSimental2025}, and proved the
conjecture for splicing maps arising from double Bott--Samelson varieties
\cite[Theorem~1.8]{GorskyKimScrogginSimental2025}.  Here, as a consequence of
the half-decorated specialization of the main theorem, we prove the conjecture
in a new and different case.

\begin{corollary}[Quasi-cluster property of the splicing map]
\label{cor:splicing-conjecture-four-endpoints}
Let \(p\) be a positive braid of length \(r\) such that
\(\operatorname{dem}(p)=w_0\).  Let \(\Delta\) be a reduced positive braid for
\(w_0\), and let \(N\) be the length of \(\Delta\).  Then both splicing maps
\[
 \Psi_{r,w_0}:
 \bigl(X(p)\times X(\Delta \Delta),
       \Sigma_{\times,w_0}\bigr)
 \xrightarrow{\sim}
 \bigl(\Ucal_{r,w_0}(p\Delta),
       \widehat{\Sigma}_{r,w_0}\bigr),
\]
and
\[
 \Psi_{N,e}:
 \bigl(X(\Delta\Delta)\times X(p),
       \Sigma_{\times,e}\bigr)
 \xrightarrow{\sim}
 \bigl(\Ucal_{N,e}(\Delta p),
       \widehat{\Sigma}_{N,e}\bigr)
\]
are quasi-cluster isomorphisms.
\end{corollary}

See Section~\ref{sec:splicing-specialization} for the meaning of the notations in the corollary. 

\begin{figure}[!t]
\centering
\begin{minipage}[c]{0.49\textwidth}
 \centering
 \resizebox{\linewidth}{!}{
\begin{tikzpicture}[
  x=0.94cm,
  y=0.82cm,
  scale=0.80,
  aThreeStrand/.style={line width=0.98pt,line cap=round,line join=round},
  aThreeOne/.style={aThreeStrand,draw=wqteal},
  aThreeTwo/.style={aThreeStrand,draw=wqorange,
    dash pattern=on 3.0pt off 1.45pt},
  aThreeThree/.style={aThreeStrand,draw=wqblue,
    dash pattern=on 0.75pt off 1.35pt},
  aThreeFrame/.style={draw=black!22,line width=0.48pt},
  aThreeT/.style={circle,draw=black!82,fill=blue!12,
    minimum size=2.35mm,inner sep=0pt,line width=0.56pt},
  aThreeH/.style={circle,draw=black!82,fill=white,
    minimum size=3.35mm,inner sep=0pt,line width=0.68pt},
  aThreeVertexLabel/.style={font=\scriptsize,fill=white,
    inner xsep=0.8pt,inner ysep=0.45pt,text=black!88},
  aThreeBoundaryLabel/.style={font=\footnotesize,fill=white,
    inner xsep=1.2pt,inner ysep=0.7pt,text=black!82}
]
\path[use as bounding box] (-6.55,-0.82) rectangle (6.55,10.57);

\coordinate (aCorner) at (-6,0);
\coordinate (bCorner) at ( 6,0);
\coordinate (cCorner) at ( 0,{6*sqrt(3)});

\coordinate (t002) at
  (barycentric cs:aCorner=1,bCorner=1,cCorner=7);
\coordinate (t101) at
  (barycentric cs:aCorner=4,bCorner=1,cCorner=4);
\coordinate (t011) at
  (barycentric cs:aCorner=1,bCorner=4,cCorner=4);
\coordinate (t200) at
  (barycentric cs:aCorner=7,bCorner=1,cCorner=1);
\coordinate (t110) at
  (barycentric cs:aCorner=4,bCorner=4,cCorner=1);
\coordinate (t020) at
  (barycentric cs:aCorner=1,bCorner=7,cCorner=1);

\coordinate (h001) at
  (barycentric cs:aCorner=2,bCorner=2,cCorner=5);
\coordinate (h100) at
  (barycentric cs:aCorner=5,bCorner=2,cCorner=2);
\coordinate (h010) at
  (barycentric cs:aCorner=2,bCorner=5,cCorner=2);
\coordinate (h000) at
  (barycentric cs:aCorner=1,bCorner=1,cCorner=1);

\coordinate (left1) at
  (barycentric cs:aCorner=1,cCorner=6);
\coordinate (left2) at
  (barycentric cs:aCorner=2,cCorner=5);
\coordinate (left3) at
  (barycentric cs:aCorner=3,cCorner=4);
\coordinate (left4) at
  (barycentric cs:aCorner=4,cCorner=3);
\coordinate (left5) at
  (barycentric cs:aCorner=5,cCorner=2);
\coordinate (left6) at
  (barycentric cs:aCorner=6,cCorner=1);
\coordinate (right1) at
  (barycentric cs:bCorner=1,cCorner=6);
\coordinate (right2) at
  (barycentric cs:bCorner=2,cCorner=5);
\coordinate (right3) at
  (barycentric cs:bCorner=3,cCorner=4);
\coordinate (right4) at
  (barycentric cs:bCorner=4,cCorner=3);
\coordinate (right5) at
  (barycentric cs:bCorner=5,cCorner=2);
\coordinate (right6) at
  (barycentric cs:bCorner=6,cCorner=1);
\coordinate (bottom1) at
  (barycentric cs:aCorner=6,bCorner=1);
\coordinate (bottom2) at
  (barycentric cs:aCorner=5,bCorner=2);
\coordinate (bottom3) at
  (barycentric cs:aCorner=4,bCorner=3);
\coordinate (bottom4) at
  (barycentric cs:aCorner=3,bCorner=4);
\coordinate (bottom5) at
  (barycentric cs:aCorner=2,bCorner=5);
\coordinate (bottom6) at
  (barycentric cs:aCorner=1,bCorner=6);
\draw[aThreeFrame]
  (aCorner)--(cCorner)--(bCorner)--cycle;

\draw[aThreeOne] (t002)--(h001);
\draw[aThreeOne] (t101)--(h001);
\draw[aThreeOne] (t011)--(h001);
\draw[aThreeOne] (t101)--(h100);
\draw[aThreeOne] (t200)--(h100);
\draw[aThreeOne] (t110)--(h100);
\draw[aThreeOne] (t011)--(h010);
\draw[aThreeOne] (t110)--(h010);
\draw[aThreeOne] (t020)--(h010);
\draw[aThreeTwo] (h001)--(h000);
\draw[aThreeTwo] (h100)--(h000);
\draw[aThreeTwo] (h010)--(h000);

\draw[aThreeOne]   (t002)--(left1);
\draw[aThreeTwo]   (h001)--(left2);
\draw[aThreeOne]   (t101)--(left4);
\draw[aThreeThree] (h000)--(left3);
\draw[aThreeTwo]   (h100)--(left5);
\draw[aThreeOne]   (t200)--(left6);

\draw[aThreeOne]   (t002)--(right1);
\draw[aThreeTwo]   (h001)--(right2);
\draw[aThreeOne]   (t011)--(right3);
\draw[aThreeThree] (h000)--(right4);
\draw[aThreeTwo]   (h010)--(right5);
\draw[aThreeOne]   (t020)--(right6);

\draw[aThreeOne]   (t200)--(bottom1);
\draw[aThreeTwo]   (h100)--(bottom2);
\draw[aThreeOne]   (t110)--(bottom3);
\draw[aThreeThree] (h000)--(bottom4);
\draw[aThreeTwo]   (h010)--(bottom5);
\draw[aThreeOne]   (t020)--(bottom6);

\node[aThreeT] at (t002) {};
\node[aThreeT] at (t101) {};
\node[aThreeT] at (t011) {};
\node[aThreeT] at (t200) {};
\node[aThreeT] at (t110) {};
\node[aThreeT] at (t020) {};

\node[aThreeH] at (h001) {};
\node[aThreeH] at (h100) {};
\node[aThreeH] at (h010) {};
\node[aThreeH] at (h000) {};

\node[aThreeVertexLabel,above=1.25mm of t002] {$(0,0,2,0)$};
\node[aThreeVertexLabel,above left=0.80mm and 0.60mm of t101]
  {$(1,0,1,0)$};
\node[aThreeVertexLabel,right=1.20mm of t011] {$(0,1,1,0)$};
\node[aThreeVertexLabel,left=1.20mm of t200]  {$(2,0,0,0)$};
\node[aThreeVertexLabel,below=1.10mm of t110] {$(1,1,0,0)$};
\node[aThreeVertexLabel,right=1.20mm of t020] {$(0,2,0,0)$};
\node[aThreeVertexLabel,right=1.20mm of h001] {$(0,0,1,1)$};
\node[aThreeVertexLabel,left=1.20mm of h100]  {$(1,0,0,1)$};
\node[aThreeVertexLabel,right=1.20mm of h010] {$(0,1,0,1)$};
\node[aThreeVertexLabel,left=1.20mm of h000]  {$(0,0,0,2)$};

\node[aThreeBoundaryLabel,rotate=60] at (-4.70,6.45)
  {input: $\boldsymbol\Delta=121321$};
\node[aThreeBoundaryLabel,rotate=-60] at (4.70,6.45)
  {input: $\boldsymbol\Delta=121321$};
\node[aThreeBoundaryLabel] at (0,-0.45)
  {output: $\boldsymbol\Delta=121321$};
\end{tikzpicture}}
 \par\smallskip
 {\small\textup{(a)} Tetrahedral projection}
\end{minipage}\hfill
\begin{minipage}[c]{0.46\textwidth}
 \centering
 \resizebox{\linewidth}{!}{
\begin{tikzpicture}[
  x=0.53cm,
  y=0.36cm,
  aThreeString/.style={line width=0.92pt,line cap=round,line join=round},
  aThreeStringOne/.style={aThreeString,draw=wqteal},
  aThreeStringTwo/.style={aThreeString,draw=wqorange,
    dash pattern=on 2.8pt off 1.35pt},
  aThreeStringThree/.style={aThreeString,draw=wqblue,
    dash pattern=on 0.72pt off 1.22pt},
  aThreeStringT/.style={circle,draw=black!82,fill=blue!12,
    minimum size=2.15mm,inner sep=0pt,line width=0.54pt},
  aThreeStringH/.style={circle,draw=black!82,fill=white,
    minimum size=3.05mm,inner sep=0pt,line width=0.64pt},
  aThreeStringVertexLabel/.style={font=\tiny,fill=white,
    inner xsep=0.55pt,inner ysep=0.25pt,text=black!82},
  aThreeStringBoundary/.style={font=\tiny,text=black!78},
  aThreeStringWord/.style={font=\scriptsize,text=black!84}
]
\path[use as bounding box] (-6.30,-1.70) rectangle (7.15,17.70);

\foreach \name/\x in {
  b0/-5.5,b1/-4.5,b2/-3.5,b3/-2.5,b4/-1.5,b5/-.5,
  b6/.5,b7/1.5,b8/2.5,b9/3.5,b10/4.5,b11/5.5}{
  \coordinate (\name) at (\x,16);}
\foreach \name/\x in {e0/-2.5,e1/-1.5,e2/-.5,e3/.5,e4/1.5,e5/2.5}{
  \coordinate (\name) at (\x,0);}

\coordinate (v1)  at ( 0.0,15);
\coordinate (v2)  at ( 0.0,14);
\coordinate (v4)  at (-2.0,11.5);
\coordinate (v5)  at ( 1.0,11.5);
\coordinate (v7)  at ( 0.0, 8);
\coordinate (v8)  at (-2.0, 6.5);
\coordinate (v9)  at ( 2.0, 6.5);
\coordinate (v11) at (-3.5, 2);
\coordinate (v12) at (-0.35,2);
\coordinate (v13) at ( 2.5, 2);

\coordinate (v3)  at ($(v2)!60/89!(v4)$);
\coordinate (v6)  at ($(v5)!1/6!(v9)$);
\coordinate (v10) at ($(v9)!305/421!(v12)$);

\foreach \u/\v in {
  b5/v1,b6/v1,v1/v2,b2/v4,v2/v3,v3/v4,v2/v5,b8/v5,
  v4/v8,v5/v6,v6/v9,b0/v11,v8/v11,v8/v12,
  v9/v10,v10/v12,v9/v13,b11/v13,
  v11/e0,v12/e2,v13/e5}{
  \draw[aThreeStringOne] (\u)--(\v);}

\foreach \u/\v in {
  b4/v2,b7/v2,v2/v7,v7/v8,v7/v9,b1/v8,b10/v9,
  v8/e1,v9/e4}{
  \draw[aThreeStringTwo] (\u)--(\v);}

\foreach \u/\v in {
  b3/v3,v3/v7,b9/v6,v6/v7,v7/v10,v10/e3}{
  \draw[aThreeStringThree] (\u)--(\v);}

\foreach \v in {v1,v4,v5,v11,v12,v13}{\node[aThreeStringT] at (\v) {};}
\foreach \v in {v2,v7,v8,v9}{\node[aThreeStringH] at (\v) {};}

\node[aThreeStringVertexLabel,right=0.80mm of v1]  {$(0,0,2,0)$};
\node[aThreeStringVertexLabel,right=0.80mm of v2]  {$(0,0,1,1)$};
\node[aThreeStringVertexLabel,left=0.80mm of v4]   {$(1,0,1,0)$};
\node[aThreeStringVertexLabel,right=0.80mm of v5]  {$(0,1,1,0)$};
\node[aThreeStringVertexLabel,right=0.80mm of v7]  {$(0,0,0,2)$};
\node[aThreeStringVertexLabel,left=0.80mm of v8]   {$(1,0,0,1)$};
\node[aThreeStringVertexLabel,right=0.80mm of v9]  {$(0,1,0,1)$};
\node[aThreeStringVertexLabel,left=0.80mm of v11]  {$(2,0,0,0)$};
\node[aThreeStringVertexLabel,left=0.80mm of v12]  {$(1,1,0,0)$};
\node[aThreeStringVertexLabel,right=0.80mm of v13] {$(0,2,0,0)$};

\node[aThreeStringWord] at (0,17.35)
  {$\boldsymbol\Delta\boldsymbol\Delta$};
\foreach \x/\c in {
  -5.5/1,-4.5/2,-3.5/1,-2.5/3,-1.5/2,-.5/1,
  .5/1,1.5/2,2.5/1,3.5/3,4.5/2,5.5/1}{
  \node[aThreeStringBoundary,above=0.45mm] at (\x,16) {$\c$};}
\foreach \x/\c in {-2.5/1,-1.5/2,-.5/1,.5/3,1.5/2,2.5/1}{
  \node[aThreeStringBoundary,below=0.45mm] at (\x,0) {$\c$};}
\node[aThreeStringWord] at (0,-1.35) {$\boldsymbol\Delta$};
\end{tikzpicture}}
 \par\smallskip
 {\small\textup{(b)} String diagram}
\end{minipage}
\caption{An \(A_3\) weave
\(\fW : \boldsymbol\Delta\boldsymbol\Delta\to\boldsymbol\Delta\), where
\(\boldsymbol\Delta=121321\), drawn as the tetrahedral projection and
the string diagram.}
\label{fig:a3-tetrahedral-weave}

\bigskip

\centering
\begingroup
\newcommand{\AThreeSweepCoordinates}[2]{%
  \coordinate (#1top)    at ($(#2)+(0,4)$);
  \coordinate (#1right)  at ($(#2)+(4,0)$);
  \coordinate (#1bottom) at ($(#2)+(0,-4)$);
  \coordinate (#1left)   at ($(#2)+(-4,0)$);
  \coordinate (#1m13pos)  at ($(#2)+(-1,3)$);
  \coordinate (#1p13pos)  at ($(#2)+(1,3)$);
  \coordinate (#1p22pos)  at ($(#2)+(2,2)$);
  \coordinate (#1p31pos)  at ($(#2)+(3,1)$);
  \coordinate (#1p3m1pos) at ($(#2)+(3,-1)$);
  \coordinate (#1p2m2pos) at ($(#2)+(2,-2)$);
  \coordinate (#1p1m3pos) at ($(#2)+(1,-3)$);
  \coordinate (#1m1m3pos) at ($(#2)+(-1,-3)$);
  \coordinate (#1m2m2pos) at ($(#2)+(-2,-2)$);
  \coordinate (#1m3m1pos) at ($(#2)+(-3,-1)$);
  \coordinate (#1m31pos)   at ($(#2)+(-3,1)$);
  \coordinate (#1m22pos)   at ($(#2)+(-2,2)$);
  \coordinate (#1z2pos)   at ($(#2)+(0,2)$);
  \coordinate (#1m11pos)  at ($(#2)+(-1,1)$);
  \coordinate (#1p11pos)  at ($(#2)+(1,1)$);
  \coordinate (#1m20pos)  at ($(#2)+(-2,0)$);
  \coordinate (#1z0pos)   at ($(#2)+(0,0)$);
  \coordinate (#1p20pos)  at ($(#2)+(2,0)$);
  \coordinate (#1m1m1pos) at ($(#2)+(-1,-1)$);
  \coordinate (#1p1m1pos) at ($(#2)+(1,-1)$);
  \coordinate (#1zm2pos)  at ($(#2)+(0,-2)$);
}

\newcommand{\AThreeSweepTriangleGrid}[3]{%
  \foreach \k in {0,...,4}{%
    \pgfmathsetmacro{\aThreeSweepFraction}{\k/4}
    \draw[aThreeSweepGrid]
      ($(#1)!\aThreeSweepFraction!(#2)$)--
      ($(#1)!\aThreeSweepFraction!(#3)$);
    \draw[aThreeSweepGrid]
      ($(#2)!\aThreeSweepFraction!(#1)$)--
      ($(#2)!\aThreeSweepFraction!(#3)$);
    \draw[aThreeSweepGrid]
      ($(#3)!\aThreeSweepFraction!(#1)$)--
      ($(#3)!\aThreeSweepFraction!(#2)$);
  }
}

\newcommand{\AThreeSweepSourceNodes}{%
  \foreach \v in {m13,p13,p22,p31,m31,m22}{%
    \node[aThreeSweepFrozen] (L\v) at (L\v pos) {};}
  \foreach \v in {z2,m11,p11,m20,z0,p20,m1m1,p1m1,zm2}{%
    \node[aThreeSweepMutable] (L\v) at (L\v pos) {};}
}

\newcommand{\AThreeSweepTargetNodes}{%
  \foreach \v in {m13,p13,p22,p31,m31,m22}{%
    \node[aThreeSweepFrozen] (R\v) at (R\v pos) {};}
  \foreach \v in {z2,m11,p11,m20,z0,p20,m1m1,p1m1,zm2}{%
    \node[aThreeSweepMutable] (R\v) at (R\v pos) {};}
}

\newcommand{\AThreeSweepSourceArrows}{%
  \foreach \u/\v in {
    m20/m31,m1m1/m20,m22/m20,m20/m11,zm2/m1m1,
    m11/m1m1,m1m1/z0,p1m1/zm2,p20/p1m1,z0/p1m1,
    p1m1/p11,p31/p20,p11/p20,p20/p22,m11/m22,
    z0/m11,m13/m11,m11/z2,p11/z0,p22/p11,z2/p11,
    p11/p13,z2/m13,p13/z2}{%
    \draw[aThreeSweepArrow] (L\u)--(L\v);}
  \foreach \u/\v in {m31/m22,p22/p31,m22/m13,p13/p22}{%
    \draw[aThreeSweepHalfArrow] (L\u)--(L\v);}
}

\newcommand{\AThreeSweepTargetArrows}{%
  \foreach \u/\v in {
    zm2/m1m1,p1m1/zm2,p1m1/z0,p20/p1m1,p20/p31,
    m1m1/p1m1,m1m1/m20,z0/m1m1,z0/p11,m20/m11,
    m31/m20,m11/z0,m11/m31,m11/z2,p11/p20,
    p11/m11,p11/p22,p31/p11,m22/m11,z2/p11,
    z2/m22,z2/p13,p22/z2,m13/z2,p13/m13}{%
    \draw[aThreeSweepArrow] (R\u)--(R\v);}
  \foreach \u/\v in {m31/m22,m22/m13,p22/p31,p13/p22}{%
    \draw[aThreeSweepHalfArrow] (R\u)--(R\v);}
}

\begin{tikzpicture}[
  x=0.52cm,
  y=0.52cm,
  aThreeSweepGrid/.style={draw=black!16,line width=0.34pt},
  aThreeSweepBoundary/.style={draw=black!56,line width=0.72pt,
    line cap=round,line join=round},
  aThreeSweepDiagonal/.style={draw=wqorange,line width=0.86pt,
    line cap=round},
  aThreeSweepArrow/.style={wqlocalarrow,line width=0.50pt,
    -{Stealth[length=1.18mm,width=0.84mm]}},
  aThreeSweepHalfArrow/.style={aThreeSweepArrow,wqhalfedge},
  aThreeSweepMutable/.style={circle,wqqvertex,fill=white,
    minimum size=1.90mm,inner sep=0pt,line width=0.52pt},
  aThreeSweepFrozen/.style={rectangle,wqqvertex,fill=teal!6,
    minimum size=1.90mm,inner sep=0pt,line width=0.52pt},
  aThreeSweepPanel/.style={font=\scriptsize,text=black!82},
  aThreeSweepMap/.style={wqmaparrow},
  aThreeSweepSequence/.style={font=\scriptsize,text=black!88,
    align=left,inner sep=0pt}
]
\path[use as bounding box] (-13.8,-5.35) rectangle (13.8,4.8);
\coordinate (leftCenter) at (-9,0);
\coordinate (rightCenter) at (9,0);
\AThreeSweepCoordinates{L}{leftCenter}
\AThreeSweepCoordinates{R}{rightCenter}

\AThreeSweepTriangleGrid{Lleft}{Lbottom}{Ltop}
\AThreeSweepTriangleGrid{Lbottom}{Lright}{Ltop}
\AThreeSweepTriangleGrid{Rleft}{Rright}{Rtop}
\AThreeSweepTriangleGrid{Rleft}{Rbottom}{Rright}
\draw[aThreeSweepBoundary]
  (Ltop)--(Lright)--(Lbottom)--(Lleft)--cycle;
\draw[aThreeSweepBoundary]
  (Rtop)--(Rright)--(Rbottom)--(Rleft)--cycle;
\draw[aThreeSweepDiagonal] (Ltop)--(Lbottom);
\draw[aThreeSweepDiagonal] (Rleft)--(Rright);

\AThreeSweepSourceNodes
\AThreeSweepTargetNodes
\AThreeSweepSourceArrows
\AThreeSweepTargetArrows

\node[aThreeSweepPanel,anchor=north] at ($(leftCenter)+(0,-4.55)$)
  {initial $Q(\boldsymbol\Delta\boldsymbol\Delta)$};
\node[aThreeSweepPanel,anchor=north] at ($(rightCenter)+(0,-4.55)$)
  {final $Q^{\mathrm{sweep}}(\fW)$};

\draw[aThreeSweepMap] (-3.35,0)--(3.35,0);
\node[font=\small,fill=white,inner xsep=1.5pt,inner ysep=0.8pt]
  at (0,0.58) {sweep $\mu_\fW$};
\node[font=\scriptsize,fill=white,inner xsep=1.2pt,inner ysep=0.6pt]
  at (0,-0.57)
  {$\boldsymbol\Delta\boldsymbol\Delta\longrightarrow\boldsymbol\Delta$};

\end{tikzpicture}
\endgroup
\caption{The sweep mutation sequence for \(\fW:121321\,121321\to121321\).}
\label{fig:a3-sweep-quiver}
\end{figure}

\subsection{Relation with local systems on a tetrahedron}
\label{subsec:tetrahedron-local-systems}

The \(N\)-decomposition subdivides each ideal tetrahedron of a triangulated
\(3\)-manifold into \(\binom{N+1}{3}\) octahedra labeled by
\((a,b,c,d)\in\mathbb N^4\) with \(a+b+c+d=N-2\).  It gives coordinates on
framed flat \(\mathrm{PGL}_N(\C)\)-local systems.  Each octahedron carries
shape parameters, which satisfy the higher-rank form of Thurston's gluing
equations
\cite{GaroufalidisGoernerZickert2015,DimofteGabellaGoncharov2016}.
Also, each octahedron contributes one Ptolemy relation
for the Ptolemy coordinates \cite{GaroufalidisThurstonZickert2015}.  
A flip of an ideal triangulation is a sequence of mutations, one at
each octahedron \cite{FockGoncharov2006,DimofteGabellaGoncharov2016}.
A sequence of mutations realizing a flip is used to study the Ptolemy relations
in non-simply-laced rank two cases as well \cite{Zickert2020}, but without a combinatorial object
like the \(N\)-decomposition.

For type \(A_{N-1}\), let
\(\boldsymbol\Delta:=(1)(21)\cdots((N-1)(N-2)\cdots1)\), a reduced word
for \(w_0\), and let
\[
 \fW:\boldsymbol\Delta\boldsymbol\Delta\longrightarrow\boldsymbol\Delta
\]
be the Demazure weave of Figure~\ref{fig:a3-tetrahedral-weave}.
Projecting the \(N\)-decomposition of a tetrahedron from one vertex onto the
opposite face, we see that the octahedra correspond to the vertices of \(\fW\) other than the
commutations. The octahedron \((a,b,c,d)\) becomes the vertex at which the
strands of colors \(d\) and \(d+1\) meet, trivalent when \(d=0\) and six-valent
otherwise, and neighboring octahedra are joined by the strands of \(\fW\).
The exchange relation of the sweep mutation at a vertex of \(\fW\) is the
Ptolemy relation of the corresponding octahedron.  In this way \(\mu_\fW\) realizes the
flip.  See Figure~\ref{fig:a3-sweep-quiver} for \(N=4\).  In particular \(\mu_\fW\) consists of
\(\binom N2+\binom N3=\binom{N+1}{3}\) mutations, one for each octahedron,
since a contraction and a type \(A_2\) braid move each contribute one mutation
and a commutation none.

A weave of the type
\(\boldsymbol\Delta\boldsymbol\Delta\to\boldsymbol\Delta\) thus appears as a
combinatorial object generalizing the \(N\)-decomposition, and it is available
outside type \(A\) as well.  For example, a weave for type \(C_3\) is shown in
Figure~\ref{fig:c3-tetrahedral-weave}.  We will discuss this topic in more detail
elsewhere.

\begin{figure}[htbp]
\centering
\resizebox{0.84\linewidth}{!}{
\begingroup
\colorlet{cThreeTeal}{teal!70!black}
\colorlet{cThreeOrange}{orange!82!black}
\colorlet{cThreeBlue}{blue!68!black}

\begin{tikzpicture}[
  x=1cm,y=1cm,
  cThreeStrand/.style={line width=0.92pt,line cap=round,line join=round},
  cThreeOne/.style={cThreeStrand,draw=cThreeTeal},
  cThreeTwo/.style={cThreeStrand,draw=cThreeOrange,
    dash pattern=on 3.0pt off 1.45pt},
  cThreeThree/.style={cThreeStrand,draw=cThreeBlue,
    dash pattern=on 0.78pt off 1.35pt},
  cThreeFrame/.style={draw=black!24,line width=0.52pt},
  cThreeContraction/.style={circle,draw=black!82,fill=blue!12,
    minimum size=2.45mm,inner sep=0pt,line width=0.56pt},
  cThreeBraidThree/.style={circle,draw=black!82,fill=white,
    minimum size=3.35mm,inner sep=0pt,line width=0.68pt},
  cThreeBraidFour/.style={circle,draw=black!82,fill=white,double,
    double distance=0.48pt,minimum size=3.80mm,inner sep=0pt,
    line width=0.52pt},
  cThreeBoundary/.style={font=\footnotesize,fill=white,
    inner xsep=1.2pt,inner ysep=0.7pt,text=black!84},
  cThreeLegend/.style={font=\scriptsize,text=black!78}
]
\path[use as bounding box] (-9.30,-1.82) rectangle (9.30,12.90);
\coordinate (leftCorner) at (-8.350,0);
\coordinate (topCorner) at (0,12.550);
\coordinate (rightCorner) at (8.350,0);
\draw[cThreeFrame] (leftCorner)--(topCorner)--(rightCorner)--cycle;

\draw[cThreeOne] (-5.010,5.020) -- (-4.153,3.886);
\draw[cThreeOne] (0.835,11.295) -- (-1.212,8.184);
\draw[cThreeOne] (-2.505,8.785) -- (-1.212,8.184);
\draw[cThreeOne] (-7.515,1.255) -- (-5.337,2.853);
\draw[cThreeOne] (-4.153,3.886) -- (-5.337,2.853);
\draw[cThreeOne] (-1.212,8.184) -- (-1.071,4.757);
\draw[cThreeOne] (5.845,3.765) -- (3.104,5.257);
\draw[cThreeOne] (3.340,7.530) -- (3.104,5.257);
\draw[cThreeOne] (-1.071,4.757) -- (-3.271,3.387);
\draw[cThreeOne] (-4.153,3.886) -- (-3.271,3.387);
\draw[cThreeOne] (3.104,5.257) -- (0.864,3.434);
\draw[cThreeOne] (-1.071,4.757) -- (0.864,3.434);
\draw[cThreeOne] (-3.271,3.387) -- (-3.535,2.098);
\draw[cThreeOne] (-5.337,2.853) -- (-5.343,1.553);
\draw[cThreeOne] (-3.535,2.098) -- (-5.343,1.553);
\draw[cThreeOne] (0.864,3.434) -- (0.717,1.292);
\draw[cThreeOne] (-3.535,2.098) -- (-1.151,0.791);
\draw[cThreeOne] (0.717,1.292) -- (-1.151,0.791);
\draw[cThreeOne] (0.717,1.292) -- (3.340,0.000);
\draw[cThreeOne] (-5.343,1.553) -- (-6.680,0.000);
\draw[cThreeOne] (-1.151,0.791) -- (-1.670,0.000);

\draw[cThreeTwo] (-1.670,10.040) -- (-1.212,8.184);
\draw[cThreeTwo] (-4.175,6.275) -- (-2.612,5.924);
\draw[cThreeTwo] (-1.212,8.184) -- (-2.612,5.924);
\draw[cThreeTwo] (-6.680,2.510) -- (-4.153,3.886);
\draw[cThreeTwo] (-2.612,5.924) -- (-4.153,3.886);
\draw[cThreeTwo] (-1.212,8.184) -- (0.629,7.497);
\draw[cThreeTwo] (1.670,10.040) -- (0.629,7.497);
\draw[cThreeTwo] (4.175,6.275) -- (3.104,5.257);
\draw[cThreeTwo] (0.629,7.497) -- (1.397,5.971);
\draw[cThreeTwo] (3.104,5.257) -- (1.397,5.971);
\draw[cThreeTwo] (0.629,7.497) to[bend left=-4.20] (0.523,5.435);
\draw[cThreeTwo] (1.397,5.971) -- (0.523,5.435);
\draw[cThreeTwo] (-2.612,5.924) -- (-1.071,4.757);
\draw[cThreeTwo] (0.523,5.435) -- (-1.071,4.757);
\draw[cThreeTwo] (3.104,5.257) -- (3.796,3.004);
\draw[cThreeTwo] (6.680,2.510) -- (3.796,3.004);
\draw[cThreeTwo] (0.523,5.435) -- (0.864,3.434);
\draw[cThreeTwo] (0.864,3.434) -- (2.607,2.579);
\draw[cThreeTwo] (3.796,3.004) -- (2.607,2.579);
\draw[cThreeTwo] (-1.071,4.757) -- (-0.835,2.833);
\draw[cThreeTwo] (0.864,3.434) -- (-0.835,2.833);
\draw[cThreeTwo] (-4.153,3.886) -- (-3.535,2.098);
\draw[cThreeTwo] (-0.835,2.833) -- (-3.535,2.098);
\draw[cThreeTwo] (2.607,2.579) -- (3.073,1.673);
\draw[cThreeTwo] (3.796,3.004) to[bend left=-4.20] (3.073,1.673);
\draw[cThreeTwo] (-0.835,2.833) -- (0.717,1.292);
\draw[cThreeTwo] (3.073,1.673) -- (0.717,1.292);
\draw[cThreeTwo] (-3.535,2.098) -- (-5.010,0.000);
\draw[cThreeTwo] (0.717,1.292) -- (0.000,0.000);
\draw[cThreeTwo] (3.073,1.673) -- (5.010,0.000);

\draw[cThreeThree] (-5.845,3.765) -- (-2.612,5.924);
\draw[cThreeThree] (-0.835,11.295) -- (0.629,7.497);
\draw[cThreeThree] (-3.340,7.530) -- (-2.612,5.924);
\draw[cThreeThree] (-2.612,5.924) -- (-0.088,6.478);
\draw[cThreeThree] (2.505,8.785) -- (0.629,7.497);
\draw[cThreeThree] (0.629,7.497) -- (-0.088,6.478);
\draw[cThreeThree] (5.010,5.020) -- (3.796,3.004);
\draw[cThreeThree] (-0.088,6.478) -- (0.523,5.435);
\draw[cThreeThree] (0.629,7.497) to[bend left=4.20] (0.523,5.435);
\draw[cThreeThree] (-2.612,5.924) -- (-0.835,2.833);
\draw[cThreeThree] (0.523,5.435) -- (3.796,3.004);
\draw[cThreeThree] (3.796,3.004) -- (5.567,1.833);
\draw[cThreeThree] (7.515,1.255) -- (5.567,1.833);
\draw[cThreeThree] (0.523,5.435) -- (-0.835,2.833);
\draw[cThreeThree] (-0.835,2.833) -- (3.073,1.673);
\draw[cThreeThree] (3.796,3.004) to[bend left=4.20] (3.073,1.673);
\draw[cThreeThree] (3.073,1.673) -- (5.293,1.077);
\draw[cThreeThree] (5.567,1.833) -- (5.293,1.077);
\draw[cThreeThree] (-0.835,2.833) -- (-3.340,0.000);
\draw[cThreeThree] (3.073,1.673) -- (1.670,0.000);
\draw[cThreeThree] (5.293,1.077) -- (6.680,0.000);

\node[cThreeBraidThree] (v03) at (-1.212,8.184) {};
\node[cThreeBraidFour] (v04) at (-2.612,5.924) {};
\node[cThreeBraidThree] (v05) at (-4.153,3.886) {};
\node[cThreeContraction] (v06) at (-5.337,2.853) {};
\node[cThreeBraidFour] (v08) at (0.629,7.497) {};
\node[cThreeContraction] (v09) at (-0.088,6.478) {};
\node[cThreeBraidThree] (v11) at (3.104,5.257) {};
\node[cThreeContraction] (v12) at (1.397,5.971) {};
\node[cThreeBraidFour] (v13) at (0.523,5.435) {};
\node[cThreeBraidThree] (v14) at (-1.071,4.757) {};
\node[cThreeContraction] (v16) at (-3.271,3.387) {};
\node[cThreeBraidFour] (v18) at (3.796,3.004) {};
\node[cThreeContraction] (v19) at (5.567,1.833) {};
\node[cThreeBraidThree] (v21) at (0.864,3.434) {};
\node[cThreeContraction] (v22) at (2.607,2.579) {};
\node[cThreeBraidFour] (v23) at (-0.835,2.833) {};
\node[cThreeBraidThree] (v24) at (-3.535,2.098) {};
\node[cThreeContraction] (v25) at (-5.343,1.553) {};
\node[cThreeBraidFour] (v27) at (3.073,1.673) {};
\node[cThreeContraction] (v28) at (5.293,1.077) {};
\node[cThreeBraidThree] (v30) at (0.717,1.292) {};
\node[cThreeContraction] (v31) at (-1.151,0.791) {};

\node[cThreeBoundary,rotate=56.4] at (-5.93,6.55) {input: $\boldsymbol\Delta=123123123$};
\node[cThreeBoundary,rotate=-56.4] at (5.93,6.55) {input: $\boldsymbol\Delta=123123123$};
\node[cThreeBoundary] at (0,-0.43) {output: $\boldsymbol\Delta=123123123$};

\node[cThreeContraction] at (-5.95,-1.25) {};
\node[cThreeLegend,anchor=west] at (-5.68,-1.25) {$ii\to i$};
\node[cThreeBraidThree] at (-1.45,-1.25) {};
\node[cThreeLegend,anchor=west] at (-1.15,-1.25) {$iji\to jij$};
\node[cThreeBraidFour] at (3.75,-1.25) {};
\node[cThreeLegend,anchor=west] at (4.08,-1.25) {$ijij\to jiji$};
\end{tikzpicture}
\endgroup}
\caption{A \(C_3\) weave
\(\fW : \boldsymbol\Delta\boldsymbol\Delta\to\boldsymbol\Delta\), where
\(\boldsymbol\Delta=123123123\), drawn as the tetrahedral projection.}
\label{fig:c3-tetrahedral-weave}
\end{figure}

\paragraph*{Declaration of AI use.}
The author used GPT 5.6 Sol in preparing this paper: to help refine the
mathematical formulation, to draw the TikZ figures, and to help write the text in
English.  The author takes full responsibility for the
contents of this paper.

\paragraph*{Acknowledgments.}
This project originated in discussions with Tsukasa Ishibashi on the local
systems on a tetrahedron.  The author also thanks him for his many comments on
preliminary versions of this paper.  The author was supported by the Irish
Research Council Advanced Laureate Award IRCLA/2023/1934 held by Robert
Osburn.

\section{Basic setup}
\label{sec:basic-setup}

\subsection{Algebraic groups and flags}
For a finite set \(I\), write \(\Word(I)\) for the set of finite words over
\(I\), and write \(\len(\mathbf a):=m\) for the number of letters in a word
\(\mathbf a=a_1\cdots a_m\).

Let \(\groupG\) be a connected simply connected semisimple algebraic group of
finite type over \(\C\), 
let \(\borelB_+,\borelB_-\subset\groupG\) be an opposite pair of Borel subgroups,
and let \(\torusH:=\borelB_+ \cap \borelB_-\) be a maximal torus.  Write 
\(\unipotentU_\pm\) for the unipotent radicals of \(\borelB_\pm\).  
Let \(I\)
be the index set of simple roots relative to \((\borelB_+,\torusH)\), and call
its elements \emph{colors}.  Write \(\alpha_i\in X^*(\torusH)\),
\(\alpha_i^\vee\in X_*(\torusH)\), and \(s_i\) for the simple root, simple
coroot, and simple reflection, respectively.  For each \(i\in I\), choose
\(\varphi_i:\mathrm{SL}_2\to\groupG\), and set
\[
 x_i(t):=\varphi_i\!\begin{pmatrix}1&t\\0&1\end{pmatrix},\qquad
 y_i(t):=\varphi_i\!\begin{pmatrix}1&0\\t&1\end{pmatrix},\qquad
 \alpha_i^\vee(a):=\varphi_i\!\begin{pmatrix}a&0\\0&a^{-1}\end{pmatrix}.
\]
We identify the quotients \(\borelB_\pm/\unipotentU_\pm\) with \(\torusH\).
Define the Cartan matrix and its symmetrizer by
\[
 c_{ij}:=\langle\alpha_i,\alpha_j^\vee\rangle,\qquad
 d_ic_{ij}=d_jc_{ji},\qquad d_i\in\mathbb Z_{>0}
\]
and normalize \((d_i)\) to be primitive on each simple component.  Let
\(W=N_\groupG(\torusH)/\torusH\) be the Weyl group, and denote by
\(\len:W\to\mathbb N\) the length function with respect to the simple
reflections \((s_i)_{i\in I}\).  Set
\[
 \dot s_i:=x_i(-1)y_i(1)x_i(-1).
\]
By the braid relations, for a reduced expression
\(w=s_{i_1}\cdots s_{i_m}\), the element
\[
 \dot w:=\dot s_{i_1}\cdots\dot s_{i_m}
\]
is independent of the expression.  Also set
\[
 B_i:\mathbb G_a\longrightarrow\groupG,
 \qquad t\longmapsto x_i(t)\dot s_i.
\]
Let \(w_0\) be the longest element of \(W\).  Write the transpose as
\[
 (-)^{\mathsf t}:\groupG\longrightarrow\groupG,
 \qquad
 x_i(t)^{\mathsf t}=y_i(t),\quad
 y_i(t)^{\mathsf t}=x_i(t),\quad
 h^{\mathsf t}=h\quad(h\in\torusH).
\]

Define the flag varieties and their decorated versions by
\begin{gather*}
 \Aplus:=\groupG/\unipotentU_+,\qquad
 \Aminus:=\unipotentU_-\backslash\groupG,\qquad
 \Bplus:=\groupG/\borelB_+,\qquad
 \Bminus:=\borelB_-\backslash\groupG,
\\
 \pi_\pm:\mathcal A_\pm\longrightarrow\mathcal B_\pm,\qquad
 \pi_+(g\unipotentU_+)=g\borelB_+,\quad
 \pi_-(\unipotentU_-g)=\borelB_-g.
\end{gather*}
We call points of \(\mathcal A_\pm\) plus/minus decorated flags, call points
of \(\mathcal B_\pm\) plus/minus flags, and call \(\pi_\pm(\flagA)\) the
underlying flag of \(\flagA\).  For \(\epsilon\in\{+,-\}\) and
\(\flagB\in\mathcal B_\epsilon\), an element
\(\flagA\in\mathcal A_\epsilon\) satisfying \(\pi_\epsilon(\flagA)=\flagB\)
is called a \emph{decoration} of \(\flagB\).  We write \(\pi\) when the sign is
clear.  We also denote by the superscript \(\mathsf t\) the isomorphisms
between these spaces induced by the transpose:
\[
 \begin{aligned}
  (g\unipotentU_+)^{\mathsf t}&:=\unipotentU_-g^{\mathsf t},&
  (\unipotentU_-g)^{\mathsf t}&:=g^{\mathsf t}\unipotentU_+,\\
  (g\borelB_+)^{\mathsf t}&:=\borelB_-g^{\mathsf t},&
  (\borelB_-g)^{\mathsf t}&:=g^{\mathsf t}\borelB_+.
 \end{aligned}
\]
We write the natural left actions of \(\groupG\) on
\(\mathcal A_+,\mathcal B_+\) and the natural right actions of the same group on
\(\mathcal A_-,\mathcal B_-\), placing the dot on the side of the action, as
\[
 \begin{gathered}
  a\mathbin{\cdot}(g\unipotentU_+):=ag\unipotentU_+,
  \qquad
  a\mathbin{\cdot}(g\borelB_+):=ag\borelB_+,\\
  (\unipotentU_-g)\mathbin{\cdot}a:=\unipotentU_-ga,
  \qquad
  (\borelB_-g)\mathbin{\cdot}a:=\borelB_-ga
 \end{gathered}
 \qquad(a,g\in\groupG).
\]
On a product containing plus and minus factors, we use the diagonal left
action in which \(a\) acts from the left on the plus factors and \(a^{-1}\)
acts from the right on the minus factors.  We further equip the decorated flag
spaces with the decoration actions of \(\torusH\),
\[
 (g\unipotentU_+)\mathbin{\cdot}h:=gh\unipotentU_+,
 \qquad
 h\mathbin{\cdot}(\unipotentU_-g):=\unipotentU_-hg
 \qquad(h\in\torusH).
\]
These actions commute with the \(\groupG\)-actions above.  Define the
fundamental weights by \(\langle\omega_i,\alpha_j^\vee\rangle=\delta_{ij}\).

For every \(w\in W\), define the relative position of two flags of the same
sign by
\[
 \begin{aligned}
  \pos_+(g\borelB_+,h\borelB_+)=w
  &\quad\Longleftrightarrow\quad
  g^{-1}h\in\borelB_+w\borelB_+,\\
  \pos_-(\borelB_-g,\borelB_-h)=w
  &\quad\Longleftrightarrow\quad
  gh^{-1}\in\borelB_-w\borelB_-.
 \end{aligned}
\]
For \(\epsilon\in\{+,-\}\), when the sign is clear, we also write
\(\flagB\xrightarrow{w}\flagB'\) for
\(\pos_\epsilon(\flagB,\flagB')=w\), and when only the plus side is used, we
abbreviate \(\pos:=\pos_+\).  When \(w\) occurs in a flag coset or in a
conjugate, we use any representative in \(N_{\groupG}(\torusH)\).  The
resulting cosets and subgroups are independent of the representative.  Two
flags of the same sign whose relative position is \(w_0\) are called
\emph{generic}.  For a minus flag \(\flagB=\borelB_-x\) and a plus flag
\(\flagB'=y\borelB_+\), define
\[
 \flagB\pitchfork \flagB'
 \quad\Longleftrightarrow\quad
 xy\in \borelB_-\borelB_+,
\]
and also call two flags of opposite signs satisfying this relation generic.

For a word \(\mathbf p=p_1\cdots p_m\in\Word(I)\), set
\[
 \mathbf p^{\mathrm{op}}:=p_m\cdots p_1.
\]
Define the Demazure product \(\operatorname{dem}:\Word(I)\to W\) recursively by
\[
 \operatorname{dem}(\varnothing):=e,\qquad
 \operatorname{dem}(\mathbf a i):=
 \begin{cases}
  \operatorname{dem}(\mathbf a)s_i,
  &\len(\operatorname{dem}(\mathbf a)s_i)=\len(\operatorname{dem}(\mathbf a))+1,\\
  \operatorname{dem}(\mathbf a),&\text{otherwise}
 \end{cases}.
\]
Here \(e\in W\) is the identity element.

For a word \(\mathbf p=p_1\cdots p_m\in\Word(I)\), a flag chain of type
\(\mathbf p\) is the following \(\Bplus\)-valued sequence, for which we also
use the indicated abbreviation:
\[
 (\flagB^0\xrightarrow{\mathbf p}\flagB^m)
 :=
  (\flagB^0\xrightarrow{s_{p_1}}\flagB^1
  \xrightarrow{s_{p_2}}\cdots
  \xrightarrow{s_{p_m}}\flagB^m).
\]
We similarly define a minus-side \(\Bminus\)-flag chain by
\(\flagB_0,\ldots,\flagB_m\in\Bminus\) and
\(\pos_-(\flagB_{a-1},\flagB_a)=s_{p_a}\ (1\le a\le m)\).

\begin{lemma}[Uniqueness of a flag chain of reduced type]
\label{lem:reduced-flag-chain}
Let \(\epsilon\in\{+,-\}\) and \(u\in W\), and let
\(\mathbf u\in\Word(I)\) be a reduced expression of \(u\).  If
\(\pos_\epsilon(\flagB,\flagB')=u\), then there exists a unique
\(\mathcal B_\epsilon\)-flag chain of type \(\mathbf u\) joining the two flags
\(\flagB,\flagB'\), and it depends algebraically on
\(\flagB,\flagB'\).
\end{lemma}

\begin{proof}
See \cite[Lemma~2.5]{ShenWeng2021}.
\end{proof}

\subsection{Basic properties of decorated flags}

Write
\[
 \Delta_{\omega_i}^{\groupG}:\groupG\longrightarrow\mathbb A^1
\]
for the generalized minor obtained from the highest-weight matrix coefficient
in the \(i\)-th fundamental representation.  Define the resulting pair
function, invariant under the diagonal \(\groupG\)-action above, by
\[
 \Delta_i:\Aminus\times\Aplus\longrightarrow\mathbb A^1,\qquad
 \Delta_i(\unipotentU_-x,y\unipotentU_+)
 :=\Delta_{\omega_i}^{\groupG}(xy)
\]
\cite[Section~2.1 and Definition~3.20]{ShenWeng2021}.

\begin{lemma}[Cartan semi-invariance of fundamental generalized minors]
\label{lem:minor-cartan-semi-invariance}
For every \(\flagA\in\Aminus\), \(\flagA'\in\Aplus\), and
\(h,h'\in\torusH\),
\[
 \Delta_i\bigl(
  h\mathbin{\cdot}\flagA,
  \flagA'\mathbin{\cdot}h'
 \bigr)
 =
 \omega_i(h)\omega_i(h')
 \Delta_i(\flagA,\flagA').
\]
\end{lemma}

\begin{proof}
Writing \(\flagA=\unipotentU_-x\) and
\(\flagA'=y\unipotentU_+\), the left and right
\(\torusH\)-semi-invariance of the generalized minor gives
\[
 \Delta_i(\unipotentU_-h_-x,yh_+\unipotentU_+)
 =\Delta_{\omega_i}^{\groupG}(h_-xyh_+)
 =\omega_i(h_-)\omega_i(h_+)\Delta_{\omega_i}^{\groupG}(xy).
\]
\end{proof}

\subsection{Quivers and seeds}

\begin{definition}[Quiver]
\label{def:quiver}
In this paper, a quiver is data consisting of a finite set \(I\), a subset
\(I_{\mathrm{uf}}\subset I\), a map
\[
 \eps:I\times I\longrightarrow\mathbb Q,
\]
and \(d=(d_i)_{i\in I}\in\mathbb Z_{>0}^{I}\), written as
\[
 Q=(I,I_{\mathrm{uf}},\eps,d),
\]
such that, for every \(i,j\in I\),
\[
 i\in I_{\mathrm{uf}}\ \text{or}\ j\in I_{\mathrm{uf}}
 \quad\Longrightarrow\quad \eps_{ij}\in\mathbb Z,
 \qquad
 d_i\eps_{ij}=-d_j\eps_{ji}.
\]
Write these four components as
\[
 I(Q):=I,\qquad I_{\mathrm{uf}}(Q):=I_{\mathrm{uf}},\qquad
 \eps^Q:=\eps,\qquad d^Q:=d,
\]
and call them the vertex set, mutable-vertex set, exchange matrix, and
multiplier, respectively.  We call
\(I_{\mathrm f}(Q):=I(Q)\setminus I_{\mathrm{uf}}(Q)\) the frozen-vertex set.
Thus \(\eps^Q_{ij}\) can be a nonintegral rational number only when
\(i,j\) are both frozen vertices.
\end{definition}

For a finite set \(S\), write
\[
 \Tcal_S
 :=
 \operatorname{Spec}
 \mathbb C[A_s^{\pm1}\mid s\in S]
\]
for the algebraic torus with standard coordinates \((A_s)_{s\in S}\).  Define
the algebraic torus associated with a quiver \(Q\) by
\(\Tcal_Q:=\Tcal_{I(Q)}\).

\begin{definition}[Freezing and defrosting]
\label{def:freeze-defrost}
For a quiver \(Q\) and \(D\subset I_{\mathrm{uf}}(Q)\), define
\[
 \operatorname{freeze}_D(Q)
 :=\bigl(I(Q),I_{\mathrm{uf}}(Q)\setminus D,\eps^Q,d^Q\bigr).
\]
On the other hand, a subset \(D\subset I(Q)\setminus I_{\mathrm{uf}}(Q)\) of
frozen vertices is \emph{defrostable} in \(Q\) if, for every
\(i,j\in I(Q)\),
\begin{equation}
 i\in D\ \text{or}\ j\in D
 \quad\Longrightarrow\quad \eps^Q_{ij}\in\mathbb Z.
 \label{eq:defrost-integrality}
\end{equation}
For a defrostable \(D\), define
\[
 \operatorname{defrost}_D(Q)
 :=\bigl(I(Q),I_{\mathrm{uf}}(Q)\cup D,\eps^Q,d^Q\bigr).
\]
Because the entries of the exchange matrix of \(Q\) incident to mutable
vertices are already integral, condition~\eqref{eq:defrost-integrality} is
necessary and sufficient for this quadruple to be a quiver.
\end{definition}

\begin{definition}[Quiver amalgamation]
\label{def:quiver-amalgamation}
Let \((Q_\alpha)_{\alpha\in A}\) be a family of quivers indexed by a finite
set \(A\).  Let \(J\) be a finite set, let \(F\subseteq J\), and let
\[
 \iota_\alpha:I(Q_\alpha)\longrightarrow J
 \quad(\alpha\in A),
 \qquad
 d=(d_x)_{x\in J}\in\mathbb Z_{>0}^{J}.
\]
Assume that
\[
 \iota_\alpha\bigl(I_{\mathrm f}(Q_\alpha)\bigr)\subseteq F,
 \qquad
 d^{Q_\alpha}_u=d_{\iota_\alpha(u)}
 \qquad(\alpha\in A,\ u\in I(Q_\alpha)),
\]
and that, for \(\alpha,\beta\in A\),
\(u\in I_{\mathrm{uf}}(Q_\alpha)\), and \(v\in I(Q_\beta)\),
\[
 \iota_\alpha(u)=\iota_\beta(v)
 \quad\Longrightarrow\quad
 \alpha=\beta\ \text{ and }\ u=v.
\]
Define the amalgamation along these maps by
\[
 Q:=\bigast_{\alpha\in A}Q_\alpha
 :=(J,J\setminus F,\eps^{\mathrm{am}},d),
\]
where
\begin{equation}
 \eps^{\mathrm{am}}_{xy}
 :=
 \sum_{\alpha\in A}
 \sum_{\substack{u\in I(Q_\alpha)\\ \iota_\alpha(u)=x}}
 \sum_{\substack{v\in I(Q_\alpha)\\ \iota_\alpha(v)=y}}
 \eps^{Q_\alpha}_{uv}
 \qquad(x,y\in J).
 \label{eq:quiver-amalgamation-sum}
\end{equation}
If \(x\notin F\) or \(y\notin F\), every summand in
\eqref{eq:quiver-amalgamation-sum} is integral.  The multiplier condition also
gives
\(d_x\eps^{\mathrm{am}}_{xy}=-d_y\eps^{\mathrm{am}}_{yx}\), so this is a
quiver.

If \(D\) is a common frozen subset of \(Q_1,Q_2\) and their multipliers agree
on \(D\), write \(Q_1\ast_DQ_2\) for this amalgamation with
\[
 J=I(Q_1)\sqcup_D I(Q_2),
 \qquad
 F=I_{\mathrm f}(Q_1)\sqcup_D I_{\mathrm f}(Q_2),
\]
and the canonical maps from the two vertex sets.
\end{definition}

\begin{definition}[Seed]
\label{def:seed}
A seed is a triple \(\Sigma=(Q,\Tcal_\Sigma,(A_i)_{i\in I(Q)})\), where \(Q\)
is a quiver, \(\Tcal_\Sigma\) is an algebraic torus, and
\((A_i)_{i\in I(Q)}\) is a \(\mathbb Z\)-basis of
\(X^*(\Tcal_\Sigma)\).  Thus
\[
 \C[\Tcal_\Sigma]
 =
 \C[A_i^{\pm1}\mid i\in I(Q)].
\]
Write the coordinate basis as the map
\[
 \mathbf A:I(Q)\longrightarrow X^*(\Tcal_\Sigma),
 \qquad i\longmapsto A_i.
\]
Under the identification \(\Tcal_\Sigma\cong\Tcal_Q\) induced by this basis,
we also write \(\Sigma=(Q,\mathbf A)\).  In particular, write \(\Sigma_Q\) for
the seed consisting of a quiver \(Q\) and the standard coordinate basis of
\(\Tcal_Q\).
\end{definition}

For a seed \(\Sigma\), write
\[
 I(\Sigma):=I(Q),\qquad
 I_{\mathrm{uf}}(\Sigma):=I_{\mathrm{uf}}(Q),\qquad
 I_{\mathrm f}(\Sigma):=I(Q)\setminus I_{\mathrm{uf}}(Q),
 \qquad \eps^\Sigma:=\eps^Q.
\]
We extend freezing on the quiver part, and defrosting for a defrostable \(D\),
to seeds by retaining the torus and coordinate basis, and write the results as
\(\operatorname{freeze}_D(\Sigma)\) and
\(\operatorname{defrost}_D(\Sigma)\), respectively.  The direct sum
\(\Sigma_1\sqcup\Sigma_2\) of two seeds is defined by the product of the tori
and the direct sums of the character bases and exchange matrices, with the
entries of the exchange matrix between different factors set to zero.  A
frozen Laurent monomial is a function of the form
\[
 \prod_{f\in I_{\mathrm f}(\Sigma)}A_f^{n_f},
 \qquad n_f\in\mathbb Z.
\]

\begin{definition}[Realization of a seed on a scheme]
A realization of a seed \(\Sigma=(Q,\mathbf A)\) on a scheme \(Y\) is a
birational map
\[
 \mathbf A_\Sigma:Y\dashrightarrow\Tcal_\Sigma.
\]
For each character \(A_i:\Tcal_\Sigma\to\Gm\), write its pullback as
\[
 A_i^\Sigma:= \mathbf A_\Sigma^* A_i : Y\dashrightarrow\Gm.
\]
\end{definition}

\begin{definition}[Pushforward of a seed realization]
For a birational map \(g:Y\dashrightarrow Y'\) and a realization
\(\mathbf A_\Sigma\) of a seed \(\Sigma\) on \(Y\), define its pushforward by
\[
 g_*\mathbf A_\Sigma
 :=\mathbf A_\Sigma\circ g^{-1}
 :Y'\dashrightarrow\Tcal_\Sigma.
\]
This is a realization of \(\Sigma\) on \(Y'\).  When the realization
\(\mathbf A_\Sigma\) is fixed, write \(g_*\Sigma\) for the seed equipped with
this pushforward.
\end{definition}

\begin{definition}[Quiver mutation and relabeling]
\label{def:mutation-relabeling}
For a quiver \(Q\) and \(k\in I_{\mathrm{uf}}(Q)\), set
\[
 I\bigl(\mu_k(Q)\bigr):=I(Q),\qquad
 I_{\mathrm{uf}}\bigl(\mu_k(Q)\bigr):=I_{\mathrm{uf}}(Q),\qquad
 d^{\mu_k(Q)}:=d^Q,
\]
and define the remaining component by
\[
 \eps^{\mu_k(Q)}_{ij}
 =
 \begin{cases}
  -\eps^Q_{ij},
    &i=k\ \text{or}\ j=k,\\[1mm]
  \eps^Q_{ij}
  +\dfrac{|\eps^Q_{ik}|\eps^Q_{kj}
          +\eps^Q_{ik}|\eps^Q_{kj}|}{2},
    &i,j\ne k
 \end{cases}.
\]
We also write \(\mu_k\) for the corresponding rational coordinate
transformation of seed \(\mathcal A\)-tori,
\[
 \mu_k:\Tcal_Q\dashrightarrow\Tcal_{\mu_k(Q)},
\]
defined by
\[
 \mu_k^*(A'_i)=A_i\quad(i\ne k),\qquad
 \mu_k^*(A'_k)
 =
 A_k^{-1}\bigg(
  \prod_{\eps^Q_{ki}>0}A_i^{\eps^Q_{ki}}
  +
  \prod_{\eps^Q_{ki}<0}A_i^{-\eps^Q_{ki}}
 \bigg).
\]

For a bijection of sets \(\rho:I(Q)\xrightarrow{\sim}I'\), define the
relabeling \(\rho_*Q\) by
\[
 I(\rho_*Q)=I',\qquad
 I_{\mathrm{uf}}(\rho_*Q)=\rho(I_{\mathrm{uf}}(Q)),\qquad
 d^{\rho_*Q}_{\rho(i)}=d^Q_i,\qquad
 \eps^{\rho_*Q}_{\rho(i),\rho(j)}=\eps^Q_{ij}.
\]
Call the corresponding seed-torus isomorphism
\[
 \rho_{\Tcal}:\Tcal_Q\xrightarrow{\sim}\Tcal_{\rho_*Q},
 \qquad
 \rho_{\Tcal}^*(A'_{\rho(i)})=A_i
\]
the relabeling map.
\end{definition}

\section{Decomposition of decorated flag configurations}
\label{sec:flag-product}

\subsection{Double Bott--Samelson cells and braid varieties}
\label{subsec:double-moduli}

\begin{definition}[Decorated double Bott--Samelson cell]
\label{def:double-bs-cell}
For words \(\mathbf p=p_1\cdots p_m\) and
\(\mathbf q=q_1\cdots q_n\), call
\begin{equation}
 \Conf(\mathbf p,\mathbf q)
 :=
 \groupG\backslash
 \left\{
 (\flagA^0;\flagB^\bullet;\flagB_\bullet;\flagA_n)
 \ \middle|\
 \begin{gathered}
  \flagA^0\in\Aplus,\quad
  \flagA_n\in\Aminus, \\
  \flagB^\bullet=(\flagB^0,\ldots,\flagB^m)\in\Bplus^{m+1},\\
  \flagB_\bullet=(\flagB_0,\ldots,\flagB_n)\in\Bminus^{n+1},\\
  \pi_+(\flagA^0)=\flagB^0,\quad \pi_-(\flagA_n)=\flagB_n,\\
  \flagB^0\xrightarrow{s_{p_1}}\flagB^1
  \xrightarrow{s_{p_2}}\cdots
  \xrightarrow{s_{p_m}}\flagB^m,\\
  \flagB_0\xrightarrow{s_{q_1}}\flagB_1
  \xrightarrow{s_{q_2}}\cdots
  \xrightarrow{s_{q_n}}\flagB_n,\\
  \flagB_0\pitchfork\flagB^0,\quad \flagB_n\pitchfork\flagB^m
 \end{gathered}
 \right\}
 \label{eq:definition-double-conf}
\end{equation}
the decorated double Bott--Samelson cell.  The quotient uses the diagonal
\(\groupG\)-action defined above.  This corresponds to Shen--Weng's
\(\Conf^{\,\mathbf p}_{\mathbf q}(\mathcal A_{\mathrm{sc}})\)
\cite[Definition~2.21 and Remark~2.22]{ShenWeng2021}.
\end{definition}

\begin{definition}[Half-decorated double Bott--Samelson cells]
\label{def:half-conf}
For words \(\mathbf p=p_1\cdots p_m\) and
\(\mathbf q=q_1\cdots q_n\), define
\begin{equation}
 \Confhalf^+(\mathbf p,\mathbf q)
 :=
 \groupG\backslash
 \left\{
 (\flagA^0;\flagB^\bullet;\flagB_\bullet)
 \ \middle|\
 \substack{
  \text{the conditions in~\eqref{eq:definition-double-conf}}\\
  \text{that do not involve }\flagA_n
  }
 \right\},
 \label{eq:definition-half-conf-plus}
\end{equation}
and
\begin{equation}
 \Confhalf^-(\mathbf p,\mathbf q)
 :=
 \groupG\backslash
 \left\{
 (\flagB^\bullet;\flagB_\bullet;\flagA_n)
 \ \middle|\
 \substack{
  \text{the conditions in~\eqref{eq:definition-double-conf}}\\
  \text{that do not involve }\flagA^0
  }
 \right\}.
 \label{eq:definition-half-conf-minus}
\end{equation}
Write the maps that forget the terminal minus decoration and the initial plus
decoration, respectively, as
\begin{align*}
 \operatorname{half}_+:
 \Conf(\mathbf p,\mathbf q)
 &\longrightarrow\Confhalf^+(\mathbf p,\mathbf q),\\
 \operatorname{half}_-:
 \Conf(\mathbf p,\mathbf q)
 &\longrightarrow\Confhalf^-(\mathbf p,\mathbf q).
\end{align*}
\end{definition}

\begin{figure}[htbp]
\centering
\begin{minipage}{\textwidth}
\centering
\(\Conf(\mathbf p,\mathbf q)\)

\smallskip
\begin{tikzcd}[column sep=2.5em,row sep=2.3em]
 \flagA^0
 \arrow[r,"s_{p_1}"]
 \arrow[d,dash,"\pitchfork" description]
 & \flagB^1 \arrow[r]
 & \cdots \arrow[r]
 & \flagB^{m-1} \arrow[r,"s_{p_m}"]
 & \flagB^m \arrow[d,dash,"\pitchfork" description]
 \\
 \flagB_0 \arrow[r,"s_{q_1}"']
 & \flagB_1 \arrow[r]
 & \cdots \arrow[r]
 & \flagB_{n-1} \arrow[r,"s_{q_n}"']
 & \flagA_n
\end{tikzcd}

\medskip
\(\Confhalf^+(\mathbf p,\mathbf q)\)

\smallskip
\begin{tikzcd}[column sep=2.5em,row sep=2.3em]
 \flagA^0
 \arrow[r,"s_{p_1}"]
 \arrow[d,dash,"\pitchfork" description]
 & \flagB^1 \arrow[r]
 & \cdots \arrow[r]
 & \flagB^{m-1} \arrow[r,"s_{p_m}"]
 & \flagB^m \arrow[d,dash,"\pitchfork" description]
 \\
 \flagB_0 \arrow[r,"s_{q_1}"']
 & \flagB_1 \arrow[r]
 & \cdots \arrow[r]
 & \flagB_{n-1} \arrow[r,"s_{q_n}"']
 & \flagB_n
\end{tikzcd}

\medskip
\(\Confhalf^-(\mathbf p,\mathbf q)\)

\smallskip
\begin{tikzcd}[column sep=2.5em,row sep=2.3em]
 \flagB^0
 \arrow[r,"s_{p_1}"]
 \arrow[d,dash,"\pitchfork" description]
 & \flagB^1 \arrow[r]
 & \cdots \arrow[r]
 & \flagB^{m-1} \arrow[r,"s_{p_m}"]
 & \flagB^m \arrow[d,dash,"\pitchfork" description]
 \\
 \flagB_0 \arrow[r,"s_{q_1}"']
 & \flagB_1 \arrow[r]
 & \cdots \arrow[r]
 & \flagB_{n-1} \arrow[r,"s_{q_n}"']
 & \flagA_n
\end{tikzcd}
\end{minipage}
\caption{Flag configurations for the double Bott--Samelson cell and its two
half-decorated versions.}
\label{fig:double-bs-configurations}
\end{figure}

\begin{definition}[Double Bott--Samelson transpose]
\label{def:double-bs-transposition}
Define the transpose of flag sequences by
\[
 (\flagB^\bullet)^{\mathsf t}
 :=\bigl((\flagB^m)^{\mathsf t},\ldots,(\flagB^0)^{\mathsf t}\bigr),
 \qquad
 (\flagB_\bullet)^{\mathsf t}
 :=\bigl((\flagB_n)^{\mathsf t},\ldots,(\flagB_0)^{\mathsf t}\bigr).
\]
\[
\begin{aligned}
 (-)^{\mathsf t}:
 \Conf(\mathbf p,\mathbf q)
 &\xrightarrow{\sim}
 \Conf(\mathbf q^{\mathrm{op}},\mathbf p^{\mathrm{op}}),
 \\
 [\flagA^0;\flagB^\bullet;\flagB_\bullet;\flagA_n]
 &\longmapsto
 \bigl[(\flagA_n)^{\mathsf t};
       (\flagB_\bullet)^{\mathsf t};
       (\flagB^\bullet)^{\mathsf t};
       (\flagA^0)^{\mathsf t}\bigr]
\end{aligned}
\]
is the transposition map \cite[Subsection~2.3]{ShenWeng2021}.  The same
operation induces the isomorphism between half-decorated versions
\begin{equation}
\begin{aligned}
 (-)^{\mathsf t}:
 \Confhalf^+(\mathbf p,\mathbf q)
 &\xrightarrow{\sim}
 \Confhalf^-(\mathbf q^{\mathrm{op}},
             \mathbf p^{\mathrm{op}}),\\
 [\flagA^0;\flagB^\bullet;\flagB_\bullet]
 &\longmapsto
 \bigl[(\flagB_\bullet)^{\mathsf t};
       (\flagB^\bullet)^{\mathsf t};
       (\flagA^0)^{\mathsf t}\bigr].
\end{aligned}
 \label{eq:half-conf-transposition}
\end{equation}
\end{definition}

\begin{definition}[Subspaces determined by endpoint relative positions]
\label{def:endpoint-strata}
For every \(u,v\in W\), set
\[
 \begin{aligned}
 \Conf(\mathbf p,\mathbf q)_{u,*}
 &:=%
 \left\{[\flagA^0;\flagB^\bullet;\flagB_\bullet;\flagA_n]
 \in\Conf(\mathbf p,\mathbf q)\ \middle|\
 \pos_+(\flagB^0,\flagB^m)=u\right\},\\
 \Conf(\mathbf p,\mathbf q)_{*,v}
 &:=%
 \left\{[\flagA^0;\flagB^\bullet;\flagB_\bullet;\flagA_n]
 \in\Conf(\mathbf p,\mathbf q)\ \middle|\
 \pos_-(\flagB_0,\flagB_n)=v\right\},\\
 \Conf(\mathbf p,\mathbf q)_{u,v}
 &:=%
 \Conf(\mathbf p,\mathbf q)_{u,*}
 \cap
 \Conf(\mathbf p,\mathbf q)_{*,v}.
\end{aligned}
\]
Here \(*\) indicates that no endpoint condition is imposed for that subscript.
For the half-decorated versions, use the same endpoint conditions to define
\(\Confhalf^\epsilon(\mathbf p,\mathbf q)_{u,*}\),
\(\Confhalf^\epsilon(\mathbf p,\mathbf q)_{*,v}\), and
\(\Confhalf^\epsilon(\mathbf p,\mathbf q)_{u,v}\)
\((\epsilon\in\{+,-\})\).
\end{definition}

Set
\[
 (\Aplus\times\Bplus)_{w_0}
 :=\{(\flagA,\flagB)\in\Aplus\times\Bplus
       \mid\pos(\pi(\flagA),\flagB)=w_0\}
\]
and call its elements \emph{pinning pairs}.  The diagonal \(\groupG\)-action is
simply transitive on the space \((\Aplus\times\Bplus)_{w_0}\) of pinning
pairs, and its orbit map
\begin{equation}
 \groupG\xrightarrow{\sim}(\Aplus\times\Bplus)_{w_0},
 \qquad
 g\longmapsto(g\unipotentU_+,gw_0\borelB_+)
 \label{eq:plus-opposite-pair-orbit}
\end{equation}
is an isomorphism.  Using this isomorphism, for each \(w\in W\), define the
\(\groupG\)-equivariant morphism
\[
 \partial_w:
 (\Aplus\times\Bplus)_{w_0}
 \longrightarrow
 \Bplus,
 \qquad
 \partial_w(g\unipotentU_+,gw_0\borelB_+)
 :=gw\borelB_+.
\]
Then
\[
 \pos\bigl(\pi(\flagA),\partial_w(\flagA,\flagB)\bigr)=w.
\]

\begin{definition}[Braid variety]
\label{def:braid-variety}
For a word \(\mathbf p=p_1\cdots p_m\), let
\(u=\operatorname{dem}(\mathbf p)\), and call
\begin{equation}
 X(\mathbf p):=
 \groupG\backslash
 \left\{
 (\flagA,\flagB;\flagB^\bullet)
 \ \middle|\
 \begin{gathered}
  (\flagA,\flagB)\in(\Aplus\times\Bplus)_{w_0},\\
  \flagB^\bullet=(\flagB^0,\ldots,\flagB^m)\in\Bplus^{m+1},\\
  \flagB^0=\pi(\flagA),\quad
  \flagB^m=\partial_u(\flagA,\flagB),\\
  \flagB^0\xrightarrow{s_{p_1}}\flagB^1
  \xrightarrow{s_{p_2}}\cdots
  \xrightarrow{s_{p_m}}\flagB^m
 \end{gathered}
 \right\}
 \label{eq:braid-variety}
\end{equation}
the braid variety associated with \(\mathbf p\)
\cite[Subsection~3.3]{CasalsEtAlBraidVarieties}.
\end{definition}

\begin{definition}[Root-subgroup coordinates on a braid variety]
\label{def:endpoint-factorization}
By~\eqref{eq:plus-opposite-pair-orbit}, a point of \(X(\mathbf p)\) can be
written uniquely as
\[
 [\unipotentU_+,w_0\borelB_+;\flagB^\bullet].
\]
Then \(\flagB^0=\borelB_+\), and
\begin{equation}
 \flagB^a
 =B_{p_1}(z_1)\cdots B_{p_a}(z_a)\borelB_+
  \qquad(1\le a\le m)
 \label{eq:endpoint-factorization}
\end{equation}
defines unique algebraic functions
\(z_a:X(\mathbf p)\to\mathbb A^1\).
\end{definition}

\subsection{Decomposition of flag configurations}
\label{subsec:double-flag-decomposition}

For \(w\in W\), set
\[
 \unipotentU_+(w):=\unipotentU_+\cap w\unipotentU_-w^{-1}.
\]
The Bruhat decomposition gives an isomorphism of algebraic varieties
\begin{equation}
 \unipotentU_+(w)
 \xrightarrow{\sim}
 \{\flagB\in\Bplus\mid\pos_+(\borelB_+,\flagB)=w\},
 \qquad
 n\longmapsto nw\borelB_+.
 \label{eq:component-bruhat-chart}
\end{equation}
There is also an isomorphism of algebraic varieties
\begin{equation}
 \groupG
 \xrightarrow{\sim}
 \{(\flagA,\flagB)\in\Aplus\times\Bminus
     \mid \flagB\pitchfork\pi_+(\flagA)\},
 \qquad
 g\longmapsto(g\unipotentU_+,\borelB_-g^{-1}).
 \label{eq:component-opposite-pair-isomorphism}
\end{equation}

Let \(\mathbf p=p_1\cdots p_m\) and \(\mathbf q=q_1\cdots q_n\) be words.

\begin{definition}[Braid-variety factor maps]
\label{def:double-braid-factors}
Set \(u:=\operatorname{dem}(\mathbf p)\) and
\(v:=\operatorname{dem}(\mathbf q)\).  First define the plus-side
braid-variety factor map
\[
 \operatorname{braid}_+:
 \Confhalf^+(\mathbf p,\mathbf q)_{u,*}
 \longrightarrow X(\mathbf p)
\]
as follows.  For
\(x=[\flagA^0;\flagB^\bullet;\flagB_\bullet]
\in\Confhalf^+(\mathbf p,\mathbf q)_{u,*}\), write the unique elements
\(g\in\groupG\) and \(n\in\unipotentU_+(u)\) determined
by~\eqref{eq:component-opposite-pair-isomorphism}
and~\eqref{eq:component-bruhat-chart} as
\[
 \flagA^0=g \unipotentU_+,
 \qquad
 \flagB_0=\borelB_-g^{-1},
 \qquad
 \flagB^m=g n u\borelB_+
\]
and set
\begin{equation}
 \operatorname{braid}_+(x)
 :=[\flagA^0,g n w_0\borelB_+;\flagB^\bullet].
 \label{eq:braid-plus-definition}
\end{equation}
Next, using the transpose isomorphism in~\eqref{eq:half-conf-transposition},
\[
 (-)^{\mathsf t}:
 \Confhalf^-(\mathbf p,\mathbf q)_{*,v}
 \xrightarrow{\sim}
 \Confhalf^+(\mathbf q^{\mathrm{op}},
             \mathbf p^{\mathrm{op}})_{v^{-1},*}
\]
define the minus-side braid-variety factor map
\begin{equation}
 \begin{aligned}
 \operatorname{braid}_-:
 \Confhalf^-(\mathbf p,\mathbf q)_{*,v}
 \longrightarrow X(\mathbf q^{\mathrm{op}}),\quad
 x\longmapsto
 \operatorname{braid}_+
 \bigl(x^{\mathsf t}\bigr).
 \end{aligned}
 \label{eq:braid-minus-definition}
\end{equation}
Below, we also write \(\operatorname{braid}_+\) and
\(\operatorname{braid}_-\), respectively, for the composites obtained by
restricting the forgetful maps to the corresponding endpoint subspaces,
\[
\begin{aligned}
 \operatorname{braid}_+\circ\operatorname{half}_+:
 \Conf(\mathbf p,\mathbf q)_{u,*}
 &\longrightarrow X(\mathbf p),\\
 \operatorname{braid}_-\circ\operatorname{half}_-:
  \Conf(\mathbf p,\mathbf q)_{*,v}
  &\longrightarrow X(\mathbf q^{\mathrm{op}}).
\end{aligned}
\]
Restricting these maps to \(\Conf(\mathbf p,\mathbf q)_{u,v}\) defines the
combined map
\begin{equation}
 \operatorname{braid}
 :=\bigl(\operatorname{braid}_+,\operatorname{braid}_-\bigr):
 \Conf(\mathbf p,\mathbf q)_{u,v}
 \longrightarrow
 X(\mathbf p)\times X(\mathbf q^{\mathrm{op}}).
 \label{eq:double-braid-factor-map}
\end{equation}
\end{definition}

\begin{definition}[Reduction to reduced flag sequences]
\label{def:double-reduction}
Let \(u,v\in W\), and let \(\mathbf u,\mathbf v\) be reduced words for
\(u,v\), respectively.  First define the map
\[
 \operatorname{red}_{\mathbf u,*}:
 \Conf(\mathbf p,\mathbf q)_{u,*}
 \longrightarrow\Conf(\mathbf u,\mathbf q)
\]
as follows.  For
\(x=[\flagA^0;\flagB^\bullet;\flagB_\bullet;\flagA_n]\), let
\(\flagD^\bullet\) be the unique reduced flag sequence of type \(\mathbf u\)
joining the two flags \(\flagB^0,\flagB^m\), and set
\[
 \operatorname{red}_{\mathbf u,*}(x)
 :=[\flagA^0;\flagD^\bullet;\flagB_\bullet;\flagA_n].
\]
Similarly, define the map
\[
 \operatorname{red}_{*,\mathbf v}:
 \Conf(\mathbf p,\mathbf q)_{*,v}
 \longrightarrow\Conf(\mathbf p,\mathbf v)
\]
as follows.  For
\(x=[\flagA^0;\flagB^\bullet;\flagB_\bullet;\flagA_n]\), let
\(\flagD_\bullet\) be the unique reduced flag sequence of type \(\mathbf v\)
joining the two flags \(\flagB_0,\flagB_n\), and set
\[
 \operatorname{red}_{*,\mathbf v}(x)
 :=[\flagA^0;\flagB^\bullet;\flagD_\bullet;\flagA_n].
\]
Under both endpoint conditions, define the map
\[
 \operatorname{red}_{\mathbf u,\mathbf v}:
 \Conf(\mathbf p,\mathbf q)_{u,v}
 \longrightarrow\Conf(\mathbf u,\mathbf v)
\]
by
\[
 \operatorname{red}_{\mathbf u,\mathbf v}(x)
 :=[\flagA^0;\flagD^\bullet;\flagD_\bullet;\flagA_n].
\]
These maps are well-defined by Lemma~\ref{lem:reduced-flag-chain}.
\end{definition}

\begin{proposition}
\label{prop:one-sided-flag-product}
Let \(u:=\operatorname{dem}(\mathbf p)\) and
\(v:=\operatorname{dem}(\mathbf q)\), and let \(\mathbf u,\mathbf v\) be
reduced words for \(u,v\), respectively.  Then
\begin{align}
 \Split_+
 &:=(\operatorname{red}_{\mathbf u,*},\operatorname{braid}_+):
 \Conf(\mathbf p,\mathbf q)_{u,*}
 \xrightarrow{\sim}
 \Conf(\mathbf u,\mathbf q)\times X(\mathbf p),
 \label{eq:plus-one-sided-split}\\
 \Split_-
 &:=(\operatorname{red}_{*,\mathbf v},\operatorname{braid}_-):
 \Conf(\mathbf p,\mathbf q)_{*,v}
 \xrightarrow{\sim}
 \Conf(\mathbf p,\mathbf v)
 \times X(\mathbf q^{\mathrm{op}})
 \label{eq:minus-one-sided-split}
\end{align}
are isomorphisms of algebraic varieties.
\end{proposition}

\begin{proof}
Let \((c,x)\) be a point on the right-hand side
of~\eqref{eq:plus-one-sided-split}, where
\[
 c=[\flagA^0;\flagD^\bullet;\flagB_\bullet;\flagA_n],
 \qquad x\in X(\mathbf p).
\]
Using the pinning pair of \(\operatorname{braid}_+(c)\), we can write the two
points uniquely as
\[
 \operatorname{braid}_+(c)=[F_+;\flagD^\bullet],
 \qquad
 x=[F_+;\flagB_x^\bullet].
\]
The two plus flag sequences have the same initial and terminal flags, so
\[
 (c,x)\longmapsto
 [\flagA^0;\flagB_x^\bullet;\flagB_\bullet;\flagA_n]
\]
defines the inverse of \(\Split_+\).  It is algebraic
by~\eqref{eq:component-opposite-pair-isomorphism},
\eqref{eq:component-bruhat-chart}, and~\eqref{eq:plus-opposite-pair-orbit}.
\eqref{eq:minus-one-sided-split} follows from the transpose
isomorphism in Definition~\ref{def:double-bs-transposition}.
\end{proof}

\begin{theorem}
\label{thm:flag-product}
Let \(u:=\operatorname{dem}(\mathbf p)\) and
\(v:=\operatorname{dem}(\mathbf q)\), and let \(\mathbf u,\mathbf v\) be
reduced words for \(u,v\), respectively.  Then
\begin{equation}
 \Split
 :=
 \bigl(
  \operatorname{red}_{\mathbf u,\mathbf v},
  \operatorname{braid}
 \bigr):
 \Conf(\mathbf p,\mathbf q)_{u,v}
 \longrightarrow
 \Conf(\mathbf u,\mathbf v)
 \times X(\mathbf p)\times X(\mathbf q^{\mathrm{op}})
 \label{eq:double-split-definition}
\end{equation}
is an isomorphism of algebraic varieties.
\end{theorem}

\begin{proof}
When \(\Split_-\) is restricted to \(\Conf(\mathbf p,\mathbf q)_{u,v}\), its
first factor belongs to \(\Conf(\mathbf p,\mathbf v)_{u,*}\).  Apply
\(\Split_+\) to this first factor.  By
Proposition~\ref{prop:one-sided-flag-product}, the composite
\[
 \Conf(\mathbf p,\mathbf q)_{u,v}
 \xrightarrow{\sim}
 \Conf(\mathbf u,\mathbf v)
 \times X(\mathbf p)\times X(\mathbf q^{\mathrm{op}})
\]
is an isomorphism.  By definition, this composite equals
\(\bigl(\operatorname{red}_{\mathbf u,\mathbf v},
\operatorname{braid}_+,\operatorname{braid}_-\bigr)=\Split\).
\end{proof}

\begin{corollary}
\label{cor:half-one-sided-flag-product}
Under the assumptions of Proposition~\ref{prop:one-sided-flag-product}, the
one-sided reductions and braid-variety factor maps commute with the forgetful
maps.  We use the same notation for the reductions induced on the
half-decorated versions.  Then
\begin{align}
 \Split_{\mathrm{half}}^+
 &:=(\operatorname{red}_{\mathbf u,*},\operatorname{braid}_+):
 \Confhalf^+(\mathbf p,\mathbf q)_{u,*}
 \xrightarrow{\sim}
 \Confhalf^+(\mathbf u,\mathbf q)
 \times X(\mathbf p),
 \label{eq:plus-half-one-sided-split}\\
 \Split_{\mathrm{half}}^-
 &:=(\operatorname{red}_{*,\mathbf v},\operatorname{braid}_-):
 \Confhalf^-(\mathbf p,\mathbf q)_{*,v}
 \xrightarrow{\sim}
 \Confhalf^-(\mathbf p,\mathbf v)
 \times X(\mathbf q^{\mathrm{op}})
 \label{eq:minus-half-one-sided-split}
\end{align}
are isomorphisms of algebraic varieties.
\end{corollary}

\begin{proof}
Each fiber of \(\operatorname{half}_+\) is the set of choices of the forgotten
minus decoration and is a \(\torusH\)-torsor.  The map \(\Split_+\) commutes
with this action, and the action remains only on the first factor.  Therefore,
passing~\eqref{eq:plus-one-sided-split} to the quotient gives
\eqref{eq:plus-half-one-sided-split}.  The minus side follows from the
transpose isomorphism in~\eqref{eq:half-conf-transposition}.
\end{proof}

\subsection{Descent to positive braids}
\label{subsec:positive-braid-descent}

Write \(\Br^+\) for the quotient monoid obtained by imposing the braid
relations on \(\Word(I)\), and call its elements \emph{positive braids}.
Because the braid relations preserve word length, word length induces
\(\len:\Br^+\to\mathbb N\).

Shen--Weng proved that, between the flag sequences associated with two words
representing the same positive braid, flag propagation along braid relations
gives a canonical isomorphism independent of the choice of a sequence of braid
relations \cite[Theorem~2.18]{ShenWeng2021}.  This isomorphism preserves the
initial and terminal flags of the flag sequence and commutes with the
\(\groupG\)-action.  Consequently, the varieties \(\Conf\),
\(\Confhalf^\pm\), and \(X\) defined above for words glue under these
isomorphisms to define algebraic varieties associated with positive braids.
Below, for positive braids \(p,q\), we use the same notation, such as
\(\Conf(p,q)\) and \(X(p)\), for the corresponding algebraic varieties.

The braid-variety factor maps and reductions to reduced flag sequences
constructed in the preceding section are also compatible with this gluing and
define the maps
\[
 \begin{aligned}
  \operatorname{braid}_+&:
  \Confhalf^+(p,q)_{u,*}\longrightarrow X(p),\\
  \operatorname{braid}_-&:
  \Confhalf^-(p,q)_{*,v}\longrightarrow X(q^{\mathrm{op}}),\\
  \operatorname{red}&:
  \Conf(p,q)_{u,v}\longrightarrow\Conf(u,v)
 \end{aligned}
\]
and the isomorphism
\[
 \Split
 :=(\operatorname{red},\operatorname{braid}_+,\operatorname{braid}_-):
 \Conf(p,q)_{u,v}
 \xrightarrow{\sim}
 \Conf(u,v)\times X(p)\times X(q^{\mathrm{op}}).
\]
Similarly, the one-sided reductions
\(\operatorname{red}_{u,*},\operatorname{red}_{*,v}\), the one-sided
decompositions \(\Split_\pm\), and their half-decorated versions
\(\Split_{\mathrm{half}}^\pm\) define maps for positive braids.

\section{Sweep mutation sequences along weaves}

\subsection{Strings of a double word and the Shen--Weng quiver}
\label{subsec:double-word-data}

In this subsection, we review the cluster structure on a double
Bott--Samelson cell given by Shen--Weng \cite{ShenWeng2021}.  Its cluster
coordinates are given by objects called double words.  We also review the
diagrammatic representation of a double word called a string expression.

Define two copies of the index set \(I\), together with their disjoint union,
by
\[
 I^+:=\{i^+\mid i\in I\},\qquad
 I^-:=\{i^-\mid i\in I\},\qquad
 \widetilde I:=I^+\sqcup I^-,
\]
and define the map that forgets the tag by
\[
 |-|:\widetilde I\longrightarrow I,
 \qquad |i^\epsilon|:=i
 \quad(i\in I,\ \epsilon\in\{+,-\}).
\]
When a sign is used numerically, we regard \(+\) as \(1\) and \(-\) as
\(-1\), and define \(\operatorname{sgn}(i^\epsilon):=\epsilon\).  A word over
\(\widetilde I\) is called a \emph{double word}.  When displaying a double
word, as in
\[
 \beta=\beta(1)\cdots\beta(m)\in\Word(\widetilde I),
\]
we use ordinary parentheses for letter indices, and we denote by
\[
 \operatorname{Pos}(\beta):=\{1,\ldots,m\}
\]
its set of positions.
Write \(\mathbf p,\mathbf q\in\Word(I)\) for the words obtained by removing
the tags from the plus and minus subsequences of a double word \(\beta\),
respectively.  When two double words \(\beta,\delta\) are considered
simultaneously, write \(\mathbf p',\mathbf q'\) for the words obtained in the
same way from \(\delta\).  Write \(p,q,p',q'\) for the positive braids
represented by these words, respectively.  Below, by the canonical
isomorphisms of Section~\ref{subsec:positive-braid-descent}, we identify the
variety associated with a word with the variety associated with the positive
braid represented by that word.

\begin{definition}[Strings of a double word]
Let
\[
 \beta=\beta(1)\cdots\beta(m)\in\Word(\widetilde I)
\]
be a double word.  Write \(0\) for the left boundary and \(\infty\)
for the right boundary, and totally order
\(\{0\}\sqcup\operatorname{Pos}(\beta)\sqcup\{\infty\}\) so that
\(0<t<\infty\) for \(t\in\operatorname{Pos}(\beta)\).  Define the set of
positions of each color \(i\in I\) by
\[
 \operatorname{Pos}_i(\beta)
 :=\{t\in\operatorname{Pos}(\beta)\mid |\beta(t)|=i\},
\]
and use the totally ordered subset
\[
 \overline{\operatorname{Pos}}_i(\beta)
 :=
 \{0\}\sqcup\operatorname{Pos}_i(\beta)\sqcup\{\infty\}.
\]
For \(r,s\in\overline{\operatorname{Pos}}_i(\beta)\), write
\[
 r\mathrel{\prec}s
 \quad\Longleftrightarrow\quad
 r<s\ \text{and}\
 \neg\exists\,t\in\operatorname{Pos}_i(\beta)\ \bigl(r<t<s\bigr).
\]
Define the set
\begin{equation}
 \Vcal(\beta)
 :=
 \{(i,[r,s])\mid
   i\in I,\ r,s\in\overline{\operatorname{Pos}}_i(\beta),\
   r\mathrel{\prec}s\}
 \label{eq:string-set}
\end{equation}
and call its elements the \emph{strings} of \(\beta\).

Set
\[
 \begin{aligned}
 I_{\partial_L}(\beta)
 &:=
 \{(i,[0,s])\in\Vcal(\beta)\},\\
 I_{\partial_R}(\beta)
 &:=
 \{(i,[r,\infty])\in\Vcal(\beta)\},\\
 I_{\mathrm{int}}(\beta)
 &:=
 \{(i,[r,s])\in\Vcal(\beta)\mid
  r,s\in\operatorname{Pos}_i(\beta)\}.
 \end{aligned}
\]
Below, we call the elements of \(I_{\partial_L}(\beta)\) and
\(I_{\partial_R}(\beta)\) boundary strings, and call the elements of
\(I_{\mathrm{int}}(\beta)\) internal strings.
\end{definition}

\begin{definition}[Slice of a string]
\label{def:string-slice}
For a double word \(\beta\) and
\(\ell=(i,[r,s])\in\Vcal(\beta)\), define
\[
 \operatorname{Slice}_\beta(\ell)
 :=\{t\in\{0\}\sqcup\operatorname{Pos}(\beta)\mid r\le t<s\}.
\]
\end{definition}

\begin{definition}[Elementary triangle quiver]
\label{def:SW-basic-triangle}
For \(i^\epsilon\in\widetilde I\), all vertices of
\(Q_\triangle(i^\epsilon)\) are frozen, and its vertex set is
\[
 I\bigl(Q_\triangle(i^\epsilon)\bigr)
 =
 (I\setminus\{i\})\sqcup
 \{i_L,i_R\}.
\]
The multiplier at a vertex \(j\in I\setminus\{i\}\) is \(d_j\), and the
multipliers at \(i_L,i_R\) are \(d_i\).  Define the exchange matrix by
\begin{align}
 \eps_{i_L,i_R}
 &=-\epsilon,
 \label{eq:SW-basic-horizontal}\\
 \eps_{j,i_L}
 &=-\eps_{j,i_R}
   =\frac{\epsilon c_{ji}}2
 \qquad(j\ne i),
 \label{eq:SW-basic-slanted}\\
 \eps_{j,k}
 &=0
 \qquad(j,k\ne i).
 \label{eq:SW-basic-other}
\end{align}
\end{definition}

\begin{figure}[htbp]
\centering
\begin{tikzpicture}
\begin{scope}[xshift=-3.2cm]
 \node[font=\small] at (0,2.35) {\(Q_\triangle(i^+)\)};
 \draw[wqguide] (-2.15,1.60)--(2.15,1.60);
 \draw[wqguide] (-2.15,0.55)--(2.15,0.55);
 \draw[wqguide] (-2.15,-0.50)--(2.15,-0.50);
 \draw[wqguide] (-2.15,-1.55)--(2.15,-1.55);
 \node[wqlargefrozen] (plusa) at (0,1.60) {\(j\)};
 \node[wqlargefrozen] (plusl) at (-1.35,0.55) {\(i_L\)};
 \node[wqlargefrozen] (plusr) at (1.35,0.55) {\(i_R\)};
 \node[wqlargefrozen] (plusb) at (0,-0.50) {\(k\)};
 \node[wqlargefrozen] (plusc) at (0,-1.55) {\(\ell\)};
 \draw[wqstringarrow] (plusr)--(plusl);
 \draw[wqstringhalf] (plusl)--(plusa);
 \draw[wqstringhalf] (plusa)--(plusr);
 \draw[wqstringhalf] (plusl)--(plusb);
 \draw[wqstringhalf] (plusb)--(plusr);
 \draw[wqstringhalf] (plusl)--(plusc);
 \draw[wqstringhalf] (plusc)--(plusr);
\end{scope}
\begin{scope}[xshift=3.2cm]
 \node[font=\small] at (0,2.35) {\(Q_\triangle(i^-)\)};
 \draw[wqguide] (-2.15,1.60)--(2.15,1.60);
 \draw[wqguide] (-2.15,0.55)--(2.15,0.55);
 \draw[wqguide] (-2.15,-0.50)--(2.15,-0.50);
 \draw[wqguide] (-2.15,-1.55)--(2.15,-1.55);
 \node[wqlargefrozen] (minusa) at (0,1.60) {\(j\)};
 \node[wqlargefrozen] (minusl) at (-1.35,0.55) {\(i_L\)};
 \node[wqlargefrozen] (minusr) at (1.35,0.55) {\(i_R\)};
 \node[wqlargefrozen] (minusb) at (0,-0.50) {\(k\)};
 \node[wqlargefrozen] (minusc) at (0,-1.55) {\(\ell\)};
 \draw[wqstringarrow] (minusl)--(minusr);
 \draw[wqstringhalf] (minusa)--(minusl);
 \draw[wqstringhalf] (minusr)--(minusa);
 \draw[wqstringhalf] (minusb)--(minusl);
 \draw[wqstringhalf] (minusr)--(minusb);
 \draw[wqstringhalf] (minusc)--(minusl);
 \draw[wqstringhalf] (minusr)--(minusc);
\end{scope}
\end{tikzpicture}
\caption{Elementary triangle quivers.  The displayed
\(j,k,\ell\in I\setminus\{i\}\) represent simply-laced colors adjacent to
\(i\).  A dotted edge has exchange-matrix entry \(1/2\), and a solid edge has
entry \(1\).  The exchange-matrix entries are determined
by~\eqref{eq:SW-basic-horizontal}--\eqref{eq:SW-basic-other}, and all vertices
are frozen.}
\label{fig:SW-basic-triangles}
\end{figure}
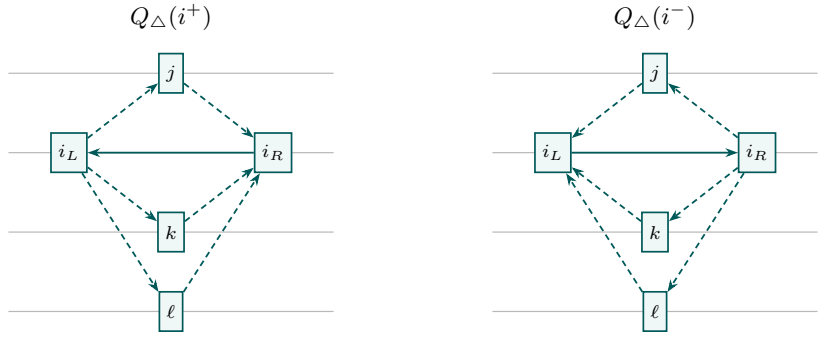

For a position \(t\in\operatorname{Pos}(\beta)\) and \(i:=|\beta(t)|\), define
the map
\[
 \iota_t^\beta:
 I\bigl(Q_\triangle(\beta(t))\bigr)
 \longrightarrow\Vcal(\beta)
\]
by the following three formulas:
\begin{alignat}{2}
 \iota_t^\beta(i_L)
 &:=(i,[r,t])
 &\quad&\bigl(r\mathrel{\prec}t\bigr),
 \label{eq:SW-basic-map-left}\\
 \iota_t^\beta(i_R)
 &:=(i,[t,s])
 &\quad&\bigl(t\mathrel{\prec}s\bigr),
 \label{eq:SW-basic-map-right}\\
 \iota_t^\beta(j)
 &:=(j,[r,s])
 &\quad&\bigl(j\ne i,\ r<t<s,\ r\mathrel{\prec}s\bigr).
 \label{eq:SW-basic-map-crossing}
\end{alignat}
With multiplier \(d_{(i,[r,s])}=d_i\) on \(\Vcal(\beta)\), these maps define
an amalgamation with \(J=F=\Vcal(\beta)\).
By the Shen--Weng integrality property
\cite[Remark~3.6]{ShenWeng2021}, \(I_{\mathrm{int}}(\beta)\) is defrostable in
the amalgamation along this family of maps.  We therefore defrost these
internal strings and define
\begin{equation}
 Q(\beta)
 :=
 \operatorname{defrost}_{I_{\mathrm{int}}(\beta)}
 \left(
  \bigast_{t\in\operatorname{Pos}(\beta)}
  Q_\triangle(\beta(t))
 \right).
 \label{eq:SW-quiver-as-amalgamation}
\end{equation}

Write the Shen--Weng cluster coordinate corresponding to each
\(\ell\in\Vcal(\beta)\) as
\[
 A_\ell^\beta:\Conf(p,q)\longrightarrow\mathbb A^1
\]
\cite[Definition~3.20]{ShenWeng2021}.  Write the
birational map with these coordinates as
\[
 \mathbf A_\beta^{\mathrm{SW}}
 :=(A_\ell^\beta)_{\ell\in\Vcal(\beta)}:
 \Conf(p,q)\dashrightarrow\Tcal_{Q(\beta)}.
\]

The string expression for \(\beta=1122112\) in type \(A_2\), overlaid with the
quiver in~\eqref{eq:SW-quiver-as-amalgamation}, is shown in
Figure~\ref{fig:SW-1122112-string-quiver}.  Here the \(+\)-tags on all letters
are suppressed.
\begin{figure}[htbp]
\centering
\begin{tikzpicture}[x=0.94cm,y=0.90cm]
\path[use as bounding box] (-0.15,-5.15) rectangle (16.15,3.65);

\node[anchor=west,font=\small] at (0,3.35)
  {\textup{(a)} string expression};
\draw[wqstring] (0,2.20)--(16,2.20);
\draw[wqstring] (0,0.60)--(16,0.60);
\node[wqlevel,anchor=east] at (-0.12,2.20) {\(1\)};
\node[wqlevel,anchor=east] at (-0.12,0.60) {\(2\)};

\node[wqstringlabel] at (1.20,2.47) {$[0,1]$};
\node[wqstringlabel] at (3.60,2.47) {$[1,2]$};
\node[wqstringlabel] at (7.30,2.47) {$[2,5]$};
\node[wqstringlabel] at (11.00,2.47) {$[5,6]$};
\node[wqstringlabel] at (14.10,2.47) {$[6,\infty]$};
\node[wqstringlabel] at (3.20,0.33) {$[0,3]$};
\node[wqstringlabel] at (7.25,0.33) {$[3,4]$};
\node[wqstringlabel] at (11.10,0.33) {$[4,7]$};
\node[wqstringlabel] at (15.05,0.33) {$[7,\infty]$};

\node[wqletter] at (2.40,2.20) {$1$};
\node[wqletter] at (4.80,2.20) {$1$};
\node[wqletter] at (6.40,0.60) {$2$};
\node[wqletter] at (8.10,0.60) {$2$};
\node[wqletter] at (9.80,2.20) {$1$};
\node[wqletter] at (12.20,2.20) {$1$};
\node[wqletter] at (14.10,0.60) {$2$};

\begin{scope}[yshift=-4.20cm]
\node[anchor=west,font=\small] at (0,3.35)
  {\textup{(b)} \(Q(1122112)\) on the string expression};
\draw[wqstring] (0,2.20)--(16,2.20);
\draw[wqstring] (0,0.60)--(16,0.60);
\node[wqlevel,anchor=east] at (-0.12,2.20) {\(1\)};
\node[wqlevel,anchor=east] at (-0.12,0.60) {\(2\)};

\node[wqlargefrozen]  (q10) at (1.20,2.20) {$[0,1]$};
\node[wqlargemutable] (q11) at (3.60,2.20) {$[1,2]$};
\node[wqlargemutable] (q12) at (7.30,2.20) {$[2,5]$};
\node[wqlargemutable] (q13) at (11.00,2.20) {$[5,6]$};
\node[wqlargefrozen]  (q14) at (14.10,2.20) {$[6,\infty]$};
\node[wqlargefrozen]  (q20) at (3.20,0.60) {$[0,3]$};
\node[wqlargemutable] (q21) at (7.25,0.60) {$[3,4]$};
\node[wqlargemutable] (q22) at (11.10,0.60) {$[4,7]$};
\node[wqlargefrozen]  (q23) at (15.05,0.60) {$[7,\infty]$};

\draw[wqstringarrow] (q11)--(q10);
\draw[wqstringarrow] (q12)--(q11);
\draw[wqstringarrow] (q13)--(q12);
\draw[wqstringarrow] (q14)--(q13);
\draw[wqstringarrow] (q21)--(q20);
\draw[wqstringarrow] (q22)--(q21);
\draw[wqstringarrow] (q23)--(q22);
\draw[wqstringhalf] (q10)--(q20);
\draw[wqstringarrow] (q20)--(q12);
\draw[wqstringarrow] (q12)--(q22);
\draw[wqstringarrow] (q22)--(q14);
\draw[wqstringhalf] (q14)--(q23);
\path (q10)--node[wqhalf,pos=0.50] {$\frac12$} (q20);
\path (q14)--node[wqhalf,pos=0.50] {$\frac12$} (q23);

\node[wqletter] at (2.40,2.20) {$1$};
\node[wqletter] at (4.80,2.20) {$1$};
\node[wqletter] at (6.40,0.60) {$2$};
\node[wqletter] at (8.10,0.60) {$2$};
\node[wqletter] at (9.80,2.20) {$1$};
\node[wqletter] at (12.20,2.20) {$1$};
\node[wqletter] at (14.10,0.60) {$2$};
\end{scope}
\end{tikzpicture}
\caption{The string expression for \(\beta=1122112\) and its associated quiver.
Dotted edges represent exchange-matrix entries equal to \(1/2\).}
\label{fig:SW-1122112-string-quiver}
\end{figure}
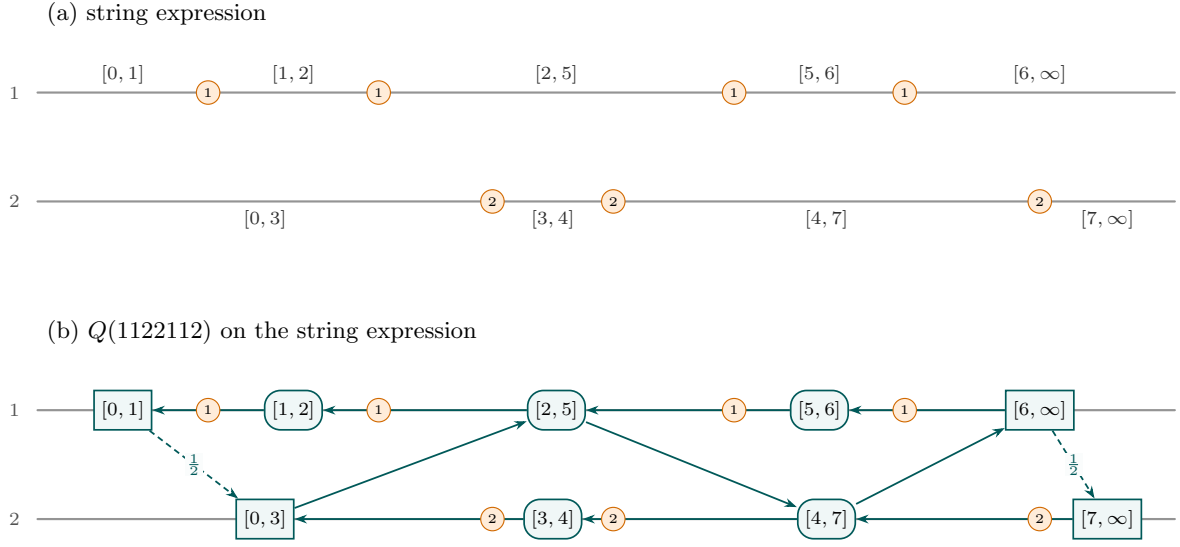

\subsection{Demazure weaves and double weaves}
\label{subsec:demazure-double-weave}

Following \cite{CasalsEtAlBraidVarieties}, we review the construction of a
quiver associated with a Demazure weave.  It defines a cluster chart on a
braid variety.  Below, in order to compare it with the cluster chart on a
double Bott--Samelson cell, we introduce double Demazure weaves, which extend
Demazure weaves.  As shown in
Proposition~\ref{prop:double-weave-plus-minus-decomposition}, however, the
resulting quiver decomposes into its plus and minus parts.

For distinct \(i,j\in I\), set
\[
 m_{ij}:=\operatorname{ord}(s_is_j)\in\{2,3,4,6\},\qquad
 \mathbf b_{ij}:=\underbrace{ijij\cdots}_{m_{ij}\ \mathrm{letters}},\qquad
 \mathbf b_{ji}:=\underbrace{jiji\cdots}_{m_{ij}\ \mathrm{letters}}.
\]

\begin{definition}[Demazure weave]
\label{def:demazure-weave}
Let \(\mathsf{Weave}(I)\) be the free monoidal category generated by \(I\) as
object generators and by the following morphisms:
\[
 \mathsf c_i:ii\longrightarrow i
 \quad(i\in I),
 \qquad
 \mathsf b_{ij}:\mathbf b_{ij}\longrightarrow\mathbf b_{ji}
 \quad(i,j\in I,\ i\ne j).
\]
Thus the objects are the elements of \(\Word(I)\), and the tensor product is
concatenation of words.  We call the generating morphism \(\mathsf c_i\) a
\emph{contraction}, call \(\mathsf b_{ij}\) a \emph{braid move}, and call
them collectively \emph{elementary weaves}.  The morphisms of
\(\mathsf{Weave}(I)\) are called \emph{Demazure weaves}
\cite[Section~4]{CasalsGorskyGorskySimental2020}.  Below, using string
diagrams, we represent a weave \(\fW:\mathbf p\to\mathbf p'\) as a directed
colored graph with initial boundary word \(\mathbf p\) and terminal boundary
word \(\mathbf p'\).  A contraction is represented by a trivalent vertex, and
a braid move by a \(2m_{ij}\)-valent vertex.  Write \(V(\fW)\) and \(E(\fW)\)
for the vertex set and edge set of this directed colored graph, respectively.
Write \(V_3(\fW)\subset V(\fW)\) for its set of trivalent vertices.
\end{definition}

Below, we sometimes call a Demazure weave simply a weave.

Horizontally reflecting the diagram and reading the word in each slice in
reverse order turns a Demazure weave \(\fW:\mathbf p\to\mathbf p'\) into the
Demazure weave
\begin{equation}
 \fW^{\mathrm{op}}:\mathbf p^{\mathrm{op}}\longrightarrow(\mathbf p')^{\mathrm{op}}.
 \label{eq:opposite-weave}
\end{equation}
The reflection defines natural bijections on edges and on trivalent vertices.
Moreover, elementary weaves preserve the Demazure product, so for every
\(\fW:\mathbf p\to\mathbf p'\),
\begin{equation}
 \operatorname{dem}(\mathbf p)=\operatorname{dem}(\mathbf p').
 \label{eq:demazure-weave-preserves-dem}
\end{equation}

We also define a double-word version of a weave.  For a word \(\mathbf b\) and
a sign \(\epsilon\), let \(\mathbf b^\epsilon\) denote the double word
obtained by attaching the tag \(\epsilon\) to every letter.

\begin{definition}[Double Demazure weave]
\label{def:double-weave}
Let \(\mathsf{DWeave}(I)\) be the free monoidal category generated by
\(\widetilde I\) as object generators and by the following morphisms:
\[
 \begin{gathered}
  \mathsf c_i^\epsilon: i^\epsilon i^\epsilon\longrightarrow i^\epsilon,
  \qquad
  \mathsf b_{ij}^\epsilon: \mathbf b_{ij}^\epsilon\longrightarrow\mathbf b_{ji}^\epsilon
  \quad(i\ne j,\ \epsilon\in\{+,-\}),\\
  \mathsf s_{i^+ j^-}: i^+j^-\longrightarrow j^-i^+,
  \qquad
  \mathsf s_{j^- i^+}: j^-i^+\longrightarrow i^+j^-
  \quad(i,j\in I).
 \end{gathered}
\]
Thus its objects are double words.  We call the generating morphisms in the
first line \emph{tagged elementary weaves}, call those in the second line
\emph{shuffles}, and call both collectively \emph{elementary double weaves}.
The morphisms of \(\mathsf{DWeave}(I)\) are called \emph{double Demazure
weaves}.  As for weaves, we represent a double weave as a directed colored
graph and write \(V(\fW)\), \(E(\fW)\), and
\(V_3(\fW)\subset V(\fW)\) for its vertex set, edge set, and set of trivalent
vertices, respectively.
\end{definition}

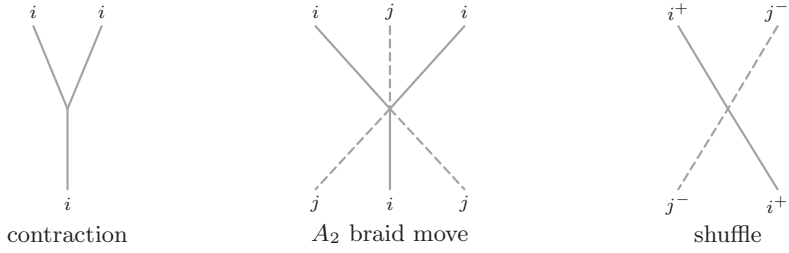
\begin{figure}[htbp]
\centering
\begin{tikzpicture}[x=0.82cm,y=0.56cm]
\path[use as bounding box] (-0.25,-0.45) rectangle (14.55,5.55);

\coordinate (basic-contraction) at (1.65,2.75);
\node[wqboundarylabel] at (1.10,5.05) {\(i\)};
\node[wqboundarylabel] at (2.20,5.05) {\(i\)};
\draw[wqweaveone,wqweavedisplay]
  (1.10,4.70)--(basic-contraction);
\draw[wqweaveone,wqweavedisplay]
  (2.20,4.70)--(basic-contraction);
\draw[wqweaveone,wqweavedisplay]
  (basic-contraction)--(1.65,0.85);
\node[wqboundarylabel] at (1.65,0.50) {\(i\)};
\node[wqpanellabel] at (1.65,-0.20) {contraction};

\coordinate (basic-braid) at (6.85,2.75);
\node[wqboundarylabel] at (5.65,5.05) {\(i\)};
\node[wqboundarylabel] at (6.85,5.05) {\(j\)};
\node[wqboundarylabel] at (8.05,5.05) {\(i\)};
\draw[wqweaveone,wqweavedisplay]
  (5.65,4.70)--(basic-braid);
\draw[wqweavetwo,wqweavedisplay]
  (6.85,4.70)--(basic-braid);
\draw[wqweaveone,wqweavedisplay]
  (8.05,4.70)--(basic-braid);
\draw[wqweavetwo,wqweavedisplay]
  (basic-braid)--(5.65,0.85);
\draw[wqweaveone,wqweavedisplay]
  (basic-braid)--(6.85,0.85);
\draw[wqweavetwo,wqweavedisplay]
  (basic-braid)--(8.05,0.85);
\node[wqboundarylabel] at (5.65,0.50) {\(j\)};
\node[wqboundarylabel] at (6.85,0.50) {\(i\)};
\node[wqboundarylabel] at (8.05,0.50) {\(j\)};
\node[wqpanellabel] at (6.85,-0.20) {\(A_2\) braid move};

\coordinate (basic-shuffle) at (12.30,2.75);
\node[wqboundarylabel] at (11.50,5.05) {\(i^+\)};
\node[wqboundarylabel] at (13.10,5.05) {\(j^-\)};
\draw[wqweaveone,wqweavedisplay]
  (11.50,4.70)--(basic-shuffle)--(13.10,0.85);
\draw[wqweavetwo,wqweavedisplay]
  (13.10,4.70)--(basic-shuffle)--(11.50,0.85);
\node[wqboundarylabel] at (11.50,0.50) {\(j^-\)};
\node[wqboundarylabel] at (13.10,0.50) {\(i^+\)};
\node[wqpanellabel] at (12.30,-0.20) {shuffle};
\end{tikzpicture}
\caption{Elementary weaves and shuffles.}
\label{fig:basic-weave-generators}
\end{figure}

A Demazure weave can be regarded as a double Demazure weave by attaching a
plus tag to every edge.

Suppose that a double weave \(\fW:\beta\to\eta\), double words
\(\mathbf L,\mathbf R\), and an elementary double weave
\(\sigma:\mathbf b\to\mathbf b'\) satisfy
\[
 \eta=\mathbf L\,\mathbf b\,\mathbf R.
\]
Then
\begin{equation}
 \fW'
 :=(\id_{\mathbf L}\otimes\sigma\otimes\id_{\mathbf R})\circ\fW:
 \beta\longrightarrow\mathbf L\,\mathbf b'\,\mathbf R
 \label{eq:double-weave-one-step-extension}
\end{equation}
is called the \emph{one-step extension} of \(\fW\) by \(\sigma\).

Below, we abbreviate double Demazure weave to double weave.

To define the boundary pairing of cycles while separating the two tag
components, we use the lattices with componentwise pairings
\[
 X^*(\torusH^2)
 :=X^*(\torusH)\oplus X^*(\torusH),
 \qquad
 X_*(\torusH^2)
 :=X_*(\torusH)\oplus X_*(\torusH).
\]
The pairing is
\[
 \left\langle(\lambda_+,\lambda_-),
 (\mu_+^\vee,\mu_-^\vee)\right\rangle
 :=
 \langle\lambda_+,\mu_+^\vee\rangle
 +\langle\lambda_-,\mu_-^\vee\rangle
\]
and the pairing of a weight and coweight supported on different tags is zero.
For \(i\in I\), set
\[
 \begin{gathered}
 \alpha_{i^+}:=(\alpha_i,0),\qquad
 \alpha_{i^-}:=(0,\alpha_i),\\
 \alpha_{i^+}^\vee:=(\alpha_i^\vee,0),\qquad
 \alpha_{i^-}^\vee:=(0,\alpha_i^\vee),\\
 s_{i^+}:=(s_i,1),\qquad
 s_{i^-}:=(1,s_i).
 \end{gathered}
\]
Define the root sequence by
\begin{align}
 \rho_{\beta,r}
 &:=s_{\beta(1)}\cdots s_{\beta(r-1)}(\alpha_{\beta(r)}),&
 \rho_{\beta,r}^\vee
 &:=s_{\beta(1)}\cdots s_{\beta(r-1)}(\alpha_{\beta(r)}^\vee),
 \label{eq:component-root-sequence}
\end{align}
and the oriented order by
\begin{equation}
 \operatorname{ord}_\beta(r,s)
 :=
 \begin{cases}
  \operatorname{sgn}(\beta(r)),&r<s,\\
  -\operatorname{sgn}(\beta(r)),&s<r,\\
  0,&r=s
 \end{cases}.
 \label{eq:component-oriented-order}
\end{equation}

\begin{definition}[Boundary pairing of a double word]
\label{def:double-boundary-pairing}
Define
\begin{equation}
 \begin{aligned}
 \Omega_\beta:
 \mathbb Z^{\operatorname{Pos}(\beta)}
 \times\mathbb Z^{\operatorname{Pos}(\beta)}
 \longrightarrow\tfrac12\mathbb Z, \quad
 (a,b)
 \longmapsto
 \frac12\sum_{r,s=1}^{m}
 \operatorname{ord}_\beta(r,s)a_r b_s
 \left\langle\rho_{\beta,r},\rho_{\beta,s}^\vee\right\rangle.
 \end{aligned}
 \label{eq:double-boundary-pairing}
\end{equation}
\end{definition}

Let \(\mathbf b\to\mathbf b'\) be a rank-two braid move or contraction, and
write \(\mathbf b=b_1\cdots b_m\in\Word(I)\) and
\(\mathbf b'=b'_1\cdots b'_{m'}\in\Word(I)\).  Set
\[
 x_{\mathbf b}(t_1,\ldots,t_m)
 :=x_{b_1}(t_1)\cdots x_{b_m}(t_m)
\]
and define the rational map
\[
 R_{\mathbf b\to\mathbf b'}:
 (\Gm)^{\operatorname{Pos}(\mathbf b)}
 \dashrightarrow
 (\Gm)^{\operatorname{Pos}(\mathbf b')}
\]
as follows.  In the braid-move case, on the common open subset on which the
Lusztig factorization coordinates on both sides are defined, it is the unique
subtraction-free birational map
\((t_1,\ldots,t_m)\mapsto(t'_1,\ldots,t'_{m'})\) satisfying
\begin{equation}
 x_{\mathbf b}(t_1,\ldots,t_m)
 =x_{\mathbf b'}(t'_1,\ldots,t'_{m'})
 \label{eq:component-positive-lusztig-factorization}
\end{equation}
\cite[Proposition~2.5]{Lusztig1994},
\cite[Theorem~5.2 and Proposition~7.1]{BerensteinZelevinsky2001}.

In the contraction case, define, using the identity
\(x_i(t_1)x_i(t_2)=x_i(t_1+t_2)\),
\[
 R_{ii\to i}(t_1,t_2):=t_1+t_2.
\]
Write the map obtained by the same construction for the Langlands dual root
datum as
\[
 R_{\mathbf b\to\mathbf b'}^\vee:
 (\Gm)^{\operatorname{Pos}(\mathbf b)}
 \dashrightarrow
 (\Gm)^{\operatorname{Pos}(\mathbf b')},
\]
and write the integer min-plus tropicalizations as
\begin{equation}
 \begin{aligned}
  \Phi_{\mathbf b\to\mathbf b'}
  &:=\operatorname{Trop}(R_{\mathbf b\to\mathbf b'}):
  \mathbb Z^{\operatorname{Pos}(\mathbf b)}
  \longrightarrow
  \mathbb Z^{\operatorname{Pos}(\mathbf b')},\\
  \Phi_{\mathbf b\to\mathbf b'}^\vee
  &:=\operatorname{Trop}(R_{\mathbf b\to\mathbf b'}^\vee):
  \mathbb Z^{\operatorname{Pos}(\mathbf b)}
  \longrightarrow
  \mathbb Z^{\operatorname{Pos}(\mathbf b')}.
 \end{aligned}
 \label{eq:lusztig-weight-transport}
\end{equation}
For a contraction, the dual map is the same addition map, and hence
\[
 \Phi_{ii\to i}(a,c)
 =\Phi_{ii\to i}^\vee(a,c)=\min(a,c).
\]

\begin{definition}[Tropical transport]
\label{def:lusztig-weight-transport}
For an elementary weave \(\sigma\) representing a braid move or contraction
with a single tag, write \(\mathbf b\to\mathbf b'\) for the transformation of
its source and target with the tag forgotten, and define
\[
 \Phi_\sigma:=\Phi_{\mathbf b\to\mathbf b'},
 \qquad
 \Phi_\sigma^\vee:=\Phi_{\mathbf b\to\mathbf b'}^\vee.
\]
For a shuffle \(\sigma:i^+j^-\to j^-i^+\) or
\(\sigma:j^-i^+\to i^+j^-\), define
\begin{equation}
 \Phi_\sigma(a,c):=(c,a),
 \qquad
 \Phi_\sigma^\vee(a,c):=(c,a).
 \label{eq:shuffle-chart-transport}
\end{equation}
\end{definition}

\begin{remark}[minus tag]
For a braid move or contraction with a minus tag, we use the factorization map
for the transformation
\(\mathbf b^{\mathrm{op}}\to(\mathbf b')^{\mathrm{op}}\), in which the orders
of the source and target are reversed.  Here the coordinate at position
\(m+1-a\) of \(\mathbf b^{\mathrm{op}}\) corresponds to position \(a\) of
\(\mathbf b=b_1\cdots b_m\), and the coordinates are indexed in the original
orders of \(\mathbf b,\mathbf b'\).  Indeed, applying the transpose
\((-)^{\mathsf t}\) to~\eqref{eq:component-positive-lusztig-factorization}
gives
\[
 y_{b_m}(t_m)\cdots y_{b_1}(t_1)
 =y_{b'_{m'}}(t'_{m'})\cdots y_{b'_1}(t'_1),
 \qquad
 t'=R_{\mathbf b\to\mathbf b'}(t).
\]
Thus, with this indexing, even in the minus-tag case we have
\(\Phi_\sigma=\Phi_{\mathbf b\to\mathbf b'}\).
\end{remark}

For example, for the type \(A_2\) braid move,
\[
 R_{iji\to jij}(t_1,t_2,t_3)
 =\left(
   \frac{t_2t_3}{t_1+t_3},
   t_1+t_3,
   \frac{t_1t_2}{t_1+t_3}
  \right),
\]
and hence
\begin{equation}
 \Phi_{iji\to jij}(a,b,c)
 =\bigl(b+c-\min(a,c),\min(a,c),a+b-\min(a,c)\bigr).
 \label{eq:Phi3}
\end{equation}

\begin{definition}[{Lusztig cycle
  (cf.~\cite[Definition~4.8 and Section~6.1]{CasalsEtAlBraidVarieties})}]
\label{def:lusztig-cycle}
Let \(\fW\) be a double weave from \(\beta\) to \(\delta\).  If a map
\(\gamma:E(\fW)\to\mathbb N\) has, at every elementary double weave \(\tau\)
in \(\fW\), output-edge weights obtained by applying \(\Phi_\tau\) to its
input-edge weights, we call \(\gamma\) a \emph{Lusztig cycle}.  Similarly, if
a map \(\gamma^\vee:E(\fW)\to\mathbb N\) satisfies the rule given by
\(\Phi_\tau^\vee\) at every \(\tau\), we call \(\gamma^\vee\) a
\emph{Langlands dual Lusztig cycle}.
\end{definition}

\begin{definition}[Lusztig cycles starting at trivalent vertices and terminal weights]
\label{def:based-lusztig-cycle}
Let \(\fW\) be a double weave from \(\beta\) to \(\delta\), and let
\(\sigma\in V_3(\fW)\).  In the slice immediately after \(\sigma\), assign
weight \(1\) to the output edge of \(\sigma\) and weight \(0\) to every other
edge.  Proceeding toward the terminal end, at each elementary double weave
\(\tau\) on the terminal side of \(\sigma\), define the output-edge weights by
applying \(\Phi_\tau\) to the input-edge weights, and assign weight \(0\) to
the edges on the initial side of \(\sigma\).  This defines a map
\[
 \gamma_\sigma:E(\fW)\longrightarrow\mathbb N
\]
which we call the Lusztig cycle starting at \(\sigma\).

Write \(d_e,d_\sigma\) for the multipliers associated with the colors of an
edge \(e\) and the contraction vertex \(\sigma\), respectively.  Applying
\cite[Lemma~6.1]{CasalsEtAlBraidVarieties} in each tag component and
using~\eqref{eq:shuffle-chart-transport} at shuffles shows that
\((d_e/d_\sigma)\gamma_\sigma(e)\) is a natural number for every
\(e\in E(\fW)\).  We therefore define the map
\(\gamma_\sigma^\vee:E(\fW)\to\mathbb N\) by
\begin{equation}
 \gamma_\sigma^\vee(e):=\frac{d_e}{d_\sigma}\gamma_\sigma(e)
 \qquad(e\in E(\fW))
 \label{eq:dual-cycle-scaling}
\end{equation}
and call it the Langlands dual Lusztig cycle starting at \(\sigma\).  Reading
the terminal boundary edges in the order of the positions in \(\delta\), write
the restrictions of \(\gamma_\sigma\) and \(\gamma_\sigma^\vee\) to these
edges as
\[
 \partial\gamma_\sigma,
 \partial\gamma_\sigma^\vee
 \in\mathbb N^{\operatorname{Pos}(\delta)}
\]
and call them the \emph{terminal weights} of \(\gamma_\sigma\) and
\(\gamma_\sigma^\vee\), respectively.
\end{definition}

\begin{remark}
\label{rem:based-lusztig-cycle}
As also noted in \cite{CasalsEtAlBraidVarieties}, there is a slight abuse of
terminology here.  The map \(\gamma_\sigma\) satisfies the Lusztig-cycle
condition at every elementary double weave other than \(\sigma\).  At
\(\sigma\), however, its input weights are \((0,0)\), while its output weight
\(1\) differs from \(\Phi_\sigma(0,0)=0\).  Hence it is not a Lusztig cycle in
the sense of Definition~\ref{def:lusztig-cycle}.
\end{remark}

\begin{definition}[Local intersection pairing]
\label{def:local-intersection}
For two three-component vectors, set
\[
 \omega((a_1,a_2,a_3),(b_1,b_2,b_3))
 :=\det\!\begin{pmatrix}
  1&1&1\\
  a_1&a_2&a_3\\
  b_1&b_2&b_3
 \end{pmatrix}.
\]
For an elementary double weave \(\sigma:\mathbf b\to\mathbf b'\), define the
local term
\[
 \sharp_\sigma:
 \left(
  \mathbb Z^{\operatorname{Pos}(\mathbf b)}
  \times\mathbb Z^{\operatorname{Pos}(\mathbf b')}
  \right)^{\!2}
 \longrightarrow\tfrac12\mathbb Z
\]
as follows.  Let \(a,b\in\mathbb Z^{\operatorname{Pos}(\mathbf b)}\) and
\(a',b'\in\mathbb Z^{\operatorname{Pos}(\mathbf b')}\).  At an elementary
braid or shuffle vertex, define
\begin{equation}
 \sharp_\sigma\bigl((a;a')\mathbin{\cdot}(b;b')\bigr)
 :=\Omega_{\mathbf b'}(a',b')-\Omega_{\mathbf b}(a,b).
 \label{eq:braid-local}
\end{equation}
For a contraction \(\sigma:i^\epsilon i^\epsilon\to i^\epsilon\), write
\(a=(a_1,a_2)\), \(b=(b_1,b_2)\), \(a'=(a'_1)\), and \(b'=(b'_1)\), and set
\begin{equation}
 \sharp_\sigma\bigl((a;a')\mathbin{\cdot}(b;b')\bigr)
 :=
 \begin{cases}
  \omega\bigl((a_1,a_2;a'_1),(b_1,b_2;b'_1)\bigr),&\epsilon=+,\\
  \omega\bigl((a_2,a_1;a'_1),(b_2,b_1;b'_1)\bigr),&\epsilon=-
 \end{cases}.
 \label{eq:contraction-local}
\end{equation}

Let \(\fW:\beta\to\delta\) be a double weave and let
\(\lambda,\mu\in\mathbb Z^{E(\fW)}\).  For each elementary double weave
\(\sigma:\mathbf b\to\mathbf b'\), write the restrictions to the position
sets of its source and target as
\[
 \lambda_\sigma
 \in\mathbb Z^{\operatorname{Pos}(\mathbf b)},
 \qquad
 \lambda'_\sigma
 \in\mathbb Z^{\operatorname{Pos}(\mathbf b')}
\]
and use analogous notation for \(\mu\).  Define, by summing the local terms,
\begin{equation}
 \sharp_\fW:\mathbb Z^{E(\fW)}\times\mathbb Z^{E(\fW)}
 \longrightarrow\tfrac12\mathbb Z,\qquad
 \sharp_\fW(\lambda\mathbin{\cdot}\mu):=
 \sum_{\sigma\in V(\fW)}
 \sharp_\sigma\left(
  (\lambda_\sigma;\lambda'_\sigma)
  \mathbin{\cdot}
  (\mu_\sigma;\mu'_\sigma)
 \right).
 \label{eq:local-intersection}
\end{equation}
\end{definition}
\begin{definition}[Double weave quiver]
\label{def:weave-quiver}
For a double weave \(\fW:\beta\to\delta\), set
\begin{equation}
 \eps^\fW_{\sigma\tau}
 :=\sharp_\fW(\gamma_\sigma^\vee\mathbin{\cdot}\gamma_\tau)
   -\Omega_\delta(\partial\gamma_\sigma^\vee,\partial\gamma_\tau),
 \qquad \sigma,\tau\in V_3(\fW),
 \label{eq:def-weave-quiver}
\end{equation}
and set
\begin{equation}
 \Dcal(\fW):=\{\sigma\in V_3(\fW)\mid\partial\gamma_\sigma\ne0\},
 \label{eq:weave-frozen-set}
\end{equation}
and define
\[
 Q(\fW):=
 \bigl(V_3(\fW),V_3(\fW)\setminus\Dcal(\fW),
       \eps^\fW,(d_\sigma)_{\sigma\in V_3(\fW)}\bigr).
\]
We apply this definition also to ordinary Demazure weaves and use the same
notation \(Q(\fW)\).
\end{definition}

From the diagram of a double weave \(\fW:\beta\to\delta\), retain only the
plus-tagged part and forget the tags.  Let \(\fW_+\) be the resulting Demazure
weave.  Similarly, let \(\fW_-\) be the Demazure weave obtained by retaining
only the minus-tagged part.  Thus
\begin{equation}
 \fW_+:\mathbf p\longrightarrow\mathbf p',
 \qquad
 \fW_-:\mathbf q\longrightarrow\mathbf q'.
 \label{eq:double-weave-projections}
\end{equation}
See Figure~\ref{fig:double-weave-plus-minus-separation}.

\begin{proposition}[Plus--minus separation of the double weave quiver]
\label{prop:double-weave-plus-minus-decomposition}
Under the vertex bijections defined by the projections
in~\eqref{eq:double-weave-projections} and horizontal reflection of the lower
part,
\begin{equation}
 \begin{aligned}
  V_3(\fW)
  &=V_3(\fW_+)\sqcup V_3(\fW_-^{\mathrm{op}}),\\
  \Dcal(\fW)
  &=\Dcal(\fW_+)\sqcup\Dcal(\fW_-^{\mathrm{op}}),\\
  Q(\fW)
  &=Q(\fW_+)\sqcup Q(\fW_-^{\mathrm{op}}).
 \end{aligned}
 \label{eq:double-weave-quiver-decomposition}
\end{equation}
In particular, \(\eps^\fW_{\sigma\tau}=0\) for two vertices \(\sigma,\tau\)
belonging to different tag components.
\end{proposition}

\begin{proof}
Each Lusztig cycle is supported only on edges with the same tag as its
contraction.
Because the root--coroot pairing between different tags is zero, both the local
intersection term and the terminal boundary term are also zero.
\end{proof}

\begin{figure}[htbp]
\centering
\begin{tikzpicture}[x=0.72cm,y=0.66cm]
\path[use as bounding box] (-0.35,-5.18) rectangle (13.75,4.40);

\coordinate (dwsep-w-plus-contract) at (0.35,3.05);
\coordinate (dwsep-w-minus-contract) at (3.50,2.82);
\coordinate (dwsep-w-shuffle-one) at (4.02,1.23);
\coordinate (dwsep-w-shuffle-two) at (3.10,-0.12);
\coordinate (dwsep-w-braid) at (1.50,-2.18);

\node[wqboundarylabel] at (0.00,4.05) {\(1^+\)};
\node[wqboundarylabel] at (0.70,4.05) {\(1^+\)};
\node[wqboundarylabel] at (1.50,4.05) {\(2^+\)};
\node[wqboundarylabel] at (2.35,4.05) {\(1^-\)};
\node[wqboundarylabel] at (3.15,4.05) {\(2^-\)};
\node[wqboundarylabel] at (3.85,4.05) {\(2^-\)};
\node[wqboundarylabel] at (4.65,4.05) {\(1^+\)};

\draw[wqweaveone,wqweavedisplay,wqweaveplus] (0.00,3.75)--(dwsep-w-plus-contract);
\draw[wqweaveone,wqweavedisplay,wqweaveplus] (0.70,3.75)--(dwsep-w-plus-contract);
\draw[wqweaveone,wqweavedisplay,wqweaveplus]
  (dwsep-w-plus-contract)--(0.35,-1.12)--(dwsep-w-braid);
\draw[wqweavetwo,wqweavedisplay,wqweaveplus] (1.50,3.75)--(dwsep-w-braid);
\draw[wqweaveone,wqweavedisplay,wqweaveplus]
  (4.65,3.75)--(4.65,1.78)--(dwsep-w-shuffle-one)
  --(dwsep-w-shuffle-two)--(dwsep-w-braid);
\draw[wqweavetwo,wqweavedisplay,wqweaveplus] (dwsep-w-braid)--(0.50,-3.75);
\draw[wqweaveone,wqweavedisplay,wqweaveplus] (dwsep-w-braid)--(1.50,-3.75);
\draw[wqweavetwo,wqweavedisplay,wqweaveplus] (dwsep-w-braid)--(2.50,-3.75);

\draw[wqweaveone,wqweavedisplay,wqweaveminus]
  (2.35,3.75)--(2.35,0.65)--(3.42,-1.05)--(3.42,-3.75);
\draw[wqweavetwo,wqweavedisplay,wqweaveminus] (3.15,3.75)--(dwsep-w-minus-contract);
\draw[wqweavetwo,wqweavedisplay,wqweaveminus] (3.85,3.75)--(dwsep-w-minus-contract);
\draw[wqweavetwo,wqweavedisplay,wqweaveminus]
  (dwsep-w-minus-contract)--(dwsep-w-shuffle-one)
  --(4.25,0.15)--(4.25,-3.75);

\node[wqboundarylabel] at (0.50,-4.05) {\(2^+\)};
\node[wqboundarylabel] at (1.50,-4.05) {\(1^+\)};
\node[wqboundarylabel] at (2.50,-4.05) {\(2^+\)};
\node[wqboundarylabel] at (3.42,-4.05) {\(1^-\)};
\node[wqboundarylabel] at (4.25,-4.05) {\(2^-\)};
\node[wqpanellabel] at (2.30,-4.90) {\(\fW\)};

\coordinate (dwsep-plus-contract) at (6.95,3.05);
\coordinate (dwsep-plus-braid) at (8.15,-2.18);

\node[wqboundarylabel] at (6.60,4.05) {\(1\)};
\node[wqboundarylabel] at (7.30,4.05) {\(1\)};
\node[wqboundarylabel] at (8.15,4.05) {\(2\)};
\node[wqboundarylabel] at (9.35,4.05) {\(1\)};

\draw[wqweaveone,wqweavedisplay,wqweaveplus] (6.60,3.75)--(dwsep-plus-contract);
\draw[wqweaveone,wqweavedisplay,wqweaveplus] (7.30,3.75)--(dwsep-plus-contract);
\draw[wqweaveone,wqweavedisplay,wqweaveplus]
  (dwsep-plus-contract)--(6.95,-1.12)--(dwsep-plus-braid);
\draw[wqweavetwo,wqweavedisplay,wqweaveplus] (8.15,3.75)--(dwsep-plus-braid);
\draw[wqweaveone,wqweavedisplay,wqweaveplus] (9.35,3.75)--(9.35,-0.82)--(dwsep-plus-braid);
\draw[wqweavetwo,wqweavedisplay,wqweaveplus] (dwsep-plus-braid)--(7.15,-3.75);
\draw[wqweaveone,wqweavedisplay,wqweaveplus] (dwsep-plus-braid)--(8.15,-3.75);
\draw[wqweavetwo,wqweavedisplay,wqweaveplus] (dwsep-plus-braid)--(9.15,-3.75);

\node[wqboundarylabel] at (7.15,-4.05) {\(2\)};
\node[wqboundarylabel] at (8.15,-4.05) {\(1\)};
\node[wqboundarylabel] at (9.15,-4.05) {\(2\)};
\node[wqpanellabel] at (8.15,-4.90) {\(\fW_+\)};

\coordinate (dwsep-minus-contract) at (11.90,3.05);

\node[wqboundarylabel] at (11.55,4.05) {\(2\)};
\node[wqboundarylabel] at (12.25,4.05) {\(2\)};
\node[wqboundarylabel] at (13.35,4.05) {\(1\)};

\draw[wqweavetwo,wqweavedisplay,wqweaveminus] (11.55,3.75)--(dwsep-minus-contract);
\draw[wqweavetwo,wqweavedisplay,wqweaveminus] (12.25,3.75)--(dwsep-minus-contract);
\draw[wqweavetwo,wqweavedisplay,wqweaveminus] (dwsep-minus-contract)--(11.90,-3.75);
\draw[wqweaveone,wqweavedisplay,wqweaveminus] (13.35,3.75)--(13.35,-3.75);

\node[wqboundarylabel] at (11.90,-4.05) {\(2\)};
\node[wqboundarylabel] at (13.35,-4.05) {\(1\)};
\node[wqpanellabel] at (12.62,-4.90) {\(\fW_-^{\mathrm{op}}\)};
\end{tikzpicture}
\caption{An example of the plus--minus separation of a double weave appearing
in Proposition~\ref{prop:double-weave-plus-minus-decomposition}.}
\label{fig:double-weave-plus-minus-separation}
\end{figure}
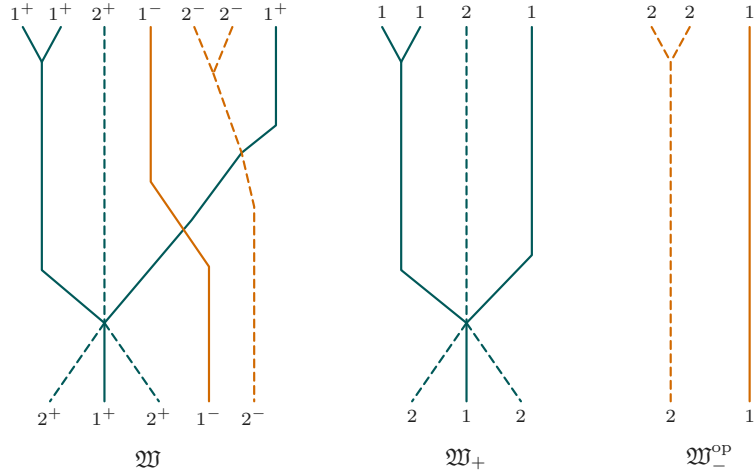

Let \(\fW\) be an ordinary Demazure weave from \(\mathbf p\) to
\(\mathbf p'\).  For \(\sigma\in V_3(\fW)\), write \(A_\sigma^\fW\) for the
corresponding cluster variable on \(X(p)\) given by
\cite[Theorem~5.12]{CasalsEtAlBraidVarieties}.  Write the rational map with
these components as
\[
 \mathbf A_\fW:=(A_\sigma^\fW)_{\sigma\in V_3(\fW)}:
 X(p)\dashrightarrow\Tcal_{Q(\fW)}.
\]
If the terminal word is reduced, this is a birational map
\cite[Lemma~5.14]{CasalsEtAlBraidVarieties}.

For a double weave \(\fW\), under the vertex identifications of
Proposition~\ref{prop:double-weave-plus-minus-decomposition}, define
\(A_\sigma^\fW:=A_\sigma^{\fW_+}\) on the plus component and
\(A_\sigma^\fW:=A_\sigma^{\fW_-^{\mathrm{op}}}\) on the minus component.
Write
\[
 \mathbf A_\fW
 :=(A_\sigma^\fW)_{\sigma\in V_3(\fW)}:
 X(p)\times X(q^{\mathrm{op}})
 \dashrightarrow\Tcal_{Q(\fW)}
\]
and set \(\Sigma_\fW:=\Sigma_{Q(\fW)}\).  If
\(\mathbf p',\mathbf q'\) are reduced, then \(\mathbf A_\fW\) equals
\(\mathbf A_{\fW_+}\times\mathbf A_{\fW_-^{\mathrm{op}}}\) and is therefore
birational.  Hence it is a realization of \(\Sigma_\fW\) on
\(X(p)\times X(q^{\mathrm{op}})\).

\subsection{Terminal amalgamation and sweep mutation sequences along weaves}
\label{subsec:terminal-amalgamation}

In this subsection, we construct the sweep mutation sequence along a weave by
induction on the identity double weave and one-step extensions, and show that
its target quiver equals the defrosted terminal amalgamation.  

We first define the terminal amalgamation.
For a double word \(\beta=\beta(1)\cdots\beta(m)\), using each map \(\iota_t^\beta\),
write the strings to the left and right of position \(t\) as
\[
 \operatorname{left}_\beta(t)
 :=\iota_t^\beta(i_L),
 \qquad
 \operatorname{right}_\beta(t)
 :=\iota_t^\beta(i_R)
 \qquad(i=|\beta(t)|).
\]
Because \(\operatorname{Pos}_i(\beta)\) was defined after forgetting the tags,
the \(+\)-positions and \(-\)-positions jointly delimit the strings of the same
color.  Write \(\mathbf e_\ell\) for the standard basis vector of
\(\mathbb Z^{\Vcal(\beta)}\) corresponding to \(\ell\in\Vcal(\beta)\), and
define the incidence map by
\begin{equation}
 \operatorname{inc}_\beta:
 \mathbb Z^{\operatorname{Pos}(\beta)}
 \longrightarrow\mathbb Z^{\Vcal(\beta)},
 \qquad
 \operatorname{inc}_\beta\lambda
 :=\sum_{t=1}^m
 \operatorname{sgn}(\beta(t))\lambda(t)
 \bigl(\mathbf e_{\operatorname{left}_\beta(t)}
       -\mathbf e_{\operatorname{right}_\beta(t)}\bigr).
 \label{eq:component-terminal-incidence}
\end{equation}

\begin{definition}[Frozen extension of the terminal quiver]
\label{def:extended}
Let \(\fW:\beta\to\delta\) be a double weave.  Using the incidence map
\(\operatorname{inc}_\delta\) in~\eqref{eq:component-terminal-incidence},
define the vertex set of \(Q^{\mathrm{ext}}(\delta;\fW)\) by
\[
 I\bigl(Q^{\mathrm{ext}}(\delta;\fW)\bigr)
 :=\Vcal(\delta)\sqcup\Dcal(\fW)
\]
and define the exchange-matrix entries by the following formulas.
\begin{equation}
\begin{aligned}
 \widetilde\eps_{\ell\ell'}
 &=\eps^{Q(\delta)}_{\ell\ell'},
 &\qquad
 \widetilde\eps_{\ell\sigma}
 &=\bigl(\operatorname{inc}_\delta(\partial\gamma_\sigma)\bigr)_\ell,\\
 \widetilde\eps_{\sigma\ell}
 &=-\bigl(\operatorname{inc}_\delta
   (\partial\gamma_\sigma^\vee)\bigr)_\ell,
 &\qquad
 \widetilde\eps_{\sigma\tau}
 &=\Omega_\delta
   \bigl(\partial\gamma_\sigma^\vee,\partial\gamma_\tau\bigr).
\end{aligned}
\label{eq:ext-exchange-matrix}
\end{equation}
Here \(\ell,\ell'\in\Vcal(\delta)\) and
\(\sigma,\tau\in\Dcal(\fW)\).  Define the mutable-vertex set by
\[
 I_{\mathrm{uf}}\bigl(Q^{\mathrm{ext}}(\delta;\fW)\bigr)
 :=I_{\mathrm{uf}}(Q(\delta))
\]
and let the multiplier be \(d_\ell=d_i\) for \(\ell=(i,[r,s])\) and
\(d_\sigma\) for \(\sigma\in\Dcal(\fW)\).  By~\eqref{eq:dual-cycle-scaling}
and the symmetrizability of the root--coroot pairing,
\(\widetilde\eps\) is skew-symmetrizable.  The mixed entries are integral, and
all newly added vertices are frozen, so these data define a quiver.
\end{definition}

\begin{figure}[htbp]
 \centering
 \begin{tikzpicture}[x=0.75cm,y=0.75cm]
  \path[use as bounding box] (-0.45,-0.40) rectangle (15.05,5.25);

  \node[wqdiagramlabel] at (2.30,4.95)
    {$\text{boundary weight of }\gamma_\sigma\text{ at }p$};
  \coordinate (input-left) at (1.65,4.45);
  \coordinate (input-right) at (2.95,4.45);
  \node[wqweavefrozen] (source-sigma) at (2.30,3.80) {};
  \draw[wqweaveone] (input-left)--(source-sigma);
  \draw[wqweaveone] (input-right)--(source-sigma);
  \node[wqdiagramlabel,above=1mm of source-sigma] {$\sigma$};
  \draw[wqcycleedge] (source-sigma)--(2.30,3.15);
  \node[wqdiagramlabel,text=wqblue] at (2.30,2.85) {$\vdots$};
  \draw[wqcycleedge] (2.30,2.55)--(2.30,1.30);

  \draw[wqweaveone] (0.20,1.30)--(4.40,1.30);
  \node[wqdiagramlabel,below=1.2mm] (position-p) at (2.30,1.30) {$p$};
  \node[wqdiagramlabel,anchor=base west]
    at ([xshift=0.4mm]position-p.base east)
    {$\in\operatorname{Pos}(\delta)$};
  \node[wqdiagramlabel,text=wqblue,right=1mm] at (2.30,1.85)
    {$a=(\partial\gamma_\sigma)(p)$};

  \draw[wqmaparrow]
    (5.05,2.55)--node[wqdiagramlabel,above]
    {$\operatorname{inc}_\delta$}(6.45,2.55);

  \node[wqdiagramlabel] at (10.55,4.95)
    {$\text{corresponding triangle in }
      Q^{\mathrm{ext}}(\delta;\fW)$};
  \node[wqlargefrozen] (left-terminal) at (8.30,1.30)
    {$\operatorname{left}_\delta(p)$};
  \node[wqlargefrozen] (right-terminal) at (12.80,1.30)
    {$\operatorname{right}_\delta(p)$};
  \node[wqweavefrozen] (sigma) at (10.55,3.80) {};
  \node[wqdiagramlabel,above=1mm of sigma] {$\sigma$};
  \wqextensiontriangle{left-terminal}{sigma}{right-terminal}
  \path (left-terminal)--node[wqdiagramlabel,text=wqorange,above left]
    {$a$}(sigma);
  \path (sigma)--node[wqdiagramlabel,text=wqorange,above right]
    {$a$}(right-terminal);
  \draw[wqstringarrow]
    (right-terminal.west)--node[wqdiagramlabel,below] {$Q(\delta)$}
    (left-terminal.east);
\end{tikzpicture}
 \caption{Extension of the quiver \(Q(\delta)\) by the terminal weights of
 Lusztig cycles.  In this figure, the weave is assumed to be simply laced and
 to have only plus tags.}
 \label{fig:lusztig-boundary-arrow-general}
\end{figure}
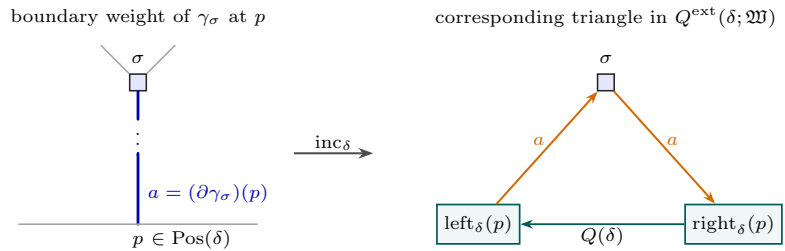

\begin{definition}[Terminal amalgamation]
\label{def:frozen-quiver-amalgamation}
Let \(\fW:\beta\to\delta\) be a double weave.  Define
\[
 Q^{\mathrm{am,fr}}(\fW)
 :=Q^{\mathrm{ext}}(\delta;\fW)
       \ast_{\Dcal(\fW)}Q(\fW).
\]
\end{definition}

\begin{proposition}
\label{prop:double-amalgamation-exchange-matrix}
Let \(\fW:\beta\to\delta\) be a double weave and let
\(\sigma,\tau\in V_3(\fW)\).  The exchange-matrix entry of
\(Q^{\mathrm{am,fr}}(\fW)\) is given by
\begin{equation}
 \eps^{Q^{\mathrm{am,fr}}(\fW)}_{\sigma\tau}
 =\sharp_\fW(\gamma_\sigma^\vee\mathbin{\cdot}\gamma_\tau).
 \label{eq:double-am-contraction-contraction}
\end{equation}
\end{proposition}

\begin{proof}
If \(\sigma,\tau\in\Dcal(\fW)\), the amalgamation adds the entries
in~\eqref{eq:ext-exchange-matrix} and~\eqref{eq:def-weave-quiver}, so the two
boundary terms cancel.  If at least one of them does not belong to
\(\Dcal(\fW)\), a vertex outside \(\Dcal(\fW)\) is not shared with
\(Q^{\mathrm{ext}}(\delta;\fW)\), and its terminal weight is zero.  Thus the
boundary term in~\eqref{eq:def-weave-quiver} is also zero.  In either case, the
remaining entry is the local intersection pairing on the right-hand side.
\end{proof}

We now turn to the inductive construction along a weave.  We first define the
extensions of \(Q(\beta)\) on which it operates.

\begin{definition}[Extension datum for \(Q(\beta)\)]
\label{def:double-word-quiver-extension-datum}
An extension datum for \(Q(\beta)\), for a double word \(\beta\), consists of
the following data:
\begin{enumerate}
 \item a quiver \(Q\) containing \(Q(\beta)\) as a full subquiver;
 \item a family of weights indexed by the set
 \[
  T:=I(Q)\setminus\Vcal(\beta),
 \]
 written as
\[
  (a_\tau)_{\tau\in T},
  \qquad
  a_\tau\in\mathbb N^{\operatorname{Pos}(\beta)}
  \quad(\tau\in T),
 \]
 and satisfying the following condition.  For each \(\tau\in T\), if we define
\[
  a_\tau^\vee(t):=\frac{d_{|\beta(t)|}}{d_\tau}a_\tau(t)
  \quad\bigl(t\in\operatorname{Pos}(\beta)\bigr),
 \]
 then
\begin{equation}
  a_\tau^\vee\in\mathbb N^{\operatorname{Pos}(\beta)},
  \qquad
  \bigl(\eps^Q_{\ell\tau}\bigr)_{\ell\in\Vcal(\beta)}
  =\operatorname{inc}_\beta(a_\tau)
  \label{eq:extension-datum-mixed-blocks}
 \end{equation}
 holds.
\end{enumerate}
Write this extension datum as
\[
 \mathcal E:=\bigl(Q,(a_\tau)_{\tau\in T}\bigr).
\]
\end{definition}

By skew-symmetrizability of the exchange matrix, an extension datum satisfies
\begin{equation}
 \bigl(\eps^Q_{\tau\ell}\bigr)_{\ell\in\Vcal(\beta)}
 =-\operatorname{inc}_\beta(a_\tau^\vee).
 \label{eq:extension-datum-dual-mixed-block}
\end{equation}

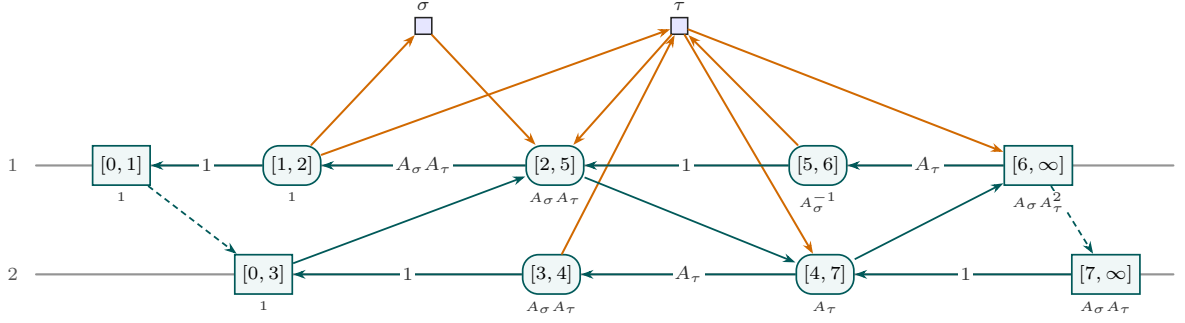
\begin{figure}[htbp]
\centering
\begin{tikzpicture}[
  x=0.94cm,y=0.90cm,
  wquvalue/.style={font=\scriptsize,fill=white,inner sep=0.65pt,
    text=black!82},
  wqmvalue/.style={font=\tiny,fill=white,inner sep=0.45pt,
    text=black!82}
]
\path[use as bounding box] (-0.20,-0.28) rectangle (16.15,5.30);
\draw[wqstring] (0,2.00)--(16,2.00);
\draw[wqstring] (0,0.40)--(16,0.40);
\node[wqlevel,anchor=east] at (-0.12,2.00) {$1$};
\node[wqlevel,anchor=east] at (-0.12,0.40) {$2$};

\node[wqlargefrozen]  (q10) at (1.20,2.00) {$[0,1]$};
\node[wqlargemutable] (q11) at (3.60,2.00) {$[1,2]$};
\node[wqlargemutable] (q12) at (7.30,2.00) {$[2,5]$};
\node[wqlargemutable] (q13) at (11.00,2.00) {$[5,6]$};
\node[wqlargefrozen]  (q14) at (14.10,2.00) {$[6,\infty]$};
\node[wqlargefrozen]  (q20) at (3.20,0.40) {$[0,3]$};
\node[wqlargemutable] (q21) at (7.25,0.40) {$[3,4]$};
\node[wqlargemutable] (q22) at (11.10,0.40) {$[4,7]$};
\node[wqlargefrozen]  (q23) at (15.05,0.40) {$[7,\infty]$};

\node[wqweavefrozen] (ext-sigma) at (5.45,4.05) {};
\node[wqweavefrozen] (ext-tau) at (9.05,4.05) {};
\node[wqvlabel,above=0.7mm of ext-sigma] {$\sigma$};
\node[wqvlabel,above=0.7mm of ext-tau] {$\tau$};

\wqextensiontriangle{q11}{ext-sigma}{q12}
\wqextensiontriangle{q11}{ext-tau}{q12}
\wqextensiontriangle{q21}{ext-tau}{q22}
\wqextensiontriangle{q13}{ext-tau}{q14}

\draw[wqstringarrow] (q11)--(q10);
\draw[wqstringarrow] (q12)--(q11);
\draw[wqstringarrow] (q13)--(q12);
\draw[wqstringarrow] (q14)--(q13);
\draw[wqstringarrow] (q21)--(q20);
\draw[wqstringarrow] (q22)--(q21);
\draw[wqstringarrow] (q23)--(q22);
\draw[wqstringhalf] (q10)--(q20);
\draw[wqstringarrow] (q20)--(q12);
\draw[wqstringarrow] (q12)--(q22);
\draw[wqstringarrow] (q22)--(q14);
\draw[wqstringhalf] (q14)--(q23);

\node[wqmvalue,anchor=north] at ([yshift=-0.65mm]q10.south) {$1$};
\node[wqmvalue,anchor=north] at ([yshift=-0.65mm]q11.south) {$1$};
\node[wqmvalue,anchor=north] at ([yshift=-0.65mm]q12.south) {$A_\sigma A_\tau$};
\node[wqmvalue,anchor=north] at ([yshift=-0.65mm]q13.south) {$A_\sigma^{-1}$};
\node[wqmvalue,anchor=north] at ([yshift=-0.65mm]q14.south) {$A_\sigma A_\tau^2$};
\node[wqmvalue,anchor=north] at ([yshift=-0.65mm]q20.south) {$1$};
\node[wqmvalue,anchor=north] at ([yshift=-0.65mm]q21.south) {$A_\sigma A_\tau$};
\node[wqmvalue,anchor=north] at ([yshift=-0.65mm]q22.south) {$A_\tau$};
\node[wqmvalue,anchor=north] at ([yshift=-0.65mm]q23.south) {$A_\sigma A_\tau$};

\node[wquvalue] at (2.40,2.00) {$1$};
\node[wquvalue] at (5.45,2.00) {$A_\sigma A_\tau$};
\node[wquvalue] at (5.23,0.40) {$1$};
\node[wquvalue] at (9.18,0.40) {$A_\tau$};
\node[wquvalue] at (9.15,2.00) {$1$};
\node[wquvalue] at (12.55,2.00) {$A_\tau$};
\node[wquvalue] at (13.08,0.40) {$1$};
\end{tikzpicture}
\caption{An extension datum for \(Q(1122112)\) in type~\(A_2\).  Here
\(T=\{\sigma,\tau\}\), \(a_\sigma=(0,1,0,0,0,0,0)\), and
\(a_\tau=(0,1,0,1,0,1,0)\).  The two additional vertices are frozen, and
\(\eps^Q_{\sigma\tau}=0\). We put a \(u\)-variable on each horizontal edge and
a monomial \(M_\ell\) below each string vertex \(\ell\).  These functions are
defined in
Section~\ref{sec:cartan-corrections-extension-data}.}
\label{fig:SW-1122112-extension-datum}
\end{figure}

\begin{lemma}
\label{lem:terminal-extension-datum}
\(\bigl(Q^{\mathrm{ext}}(\delta;\fW), (\partial\gamma_\sigma)_{\sigma\in\Dcal(\fW)}\bigr)\) is an extension datum
for \(Q(\delta)\).
\end{lemma}

\begin{proof}
\eqref{eq:dual-cycle-scaling} and~\eqref{eq:ext-exchange-matrix} give
the conditions in Definition~\ref{def:double-word-quiver-extension-datum}.
\end{proof}

\begin{definition}[Extended cluster transformation associated with an elementary double weave]
\label{def:basic-mutation-step-with-extra-vertices}
Let \(\sigma:\mathbf b\to\mathbf b'\) be an elementary double weave from
Definition~\ref{def:double-weave}, and let
\[
 \mathcal E
 =\bigl(Q,(a_\tau)_{\tau\in T}\bigr)
\]
be an extension datum for \(Q(\mathbf b)\).  Define the extended cluster
transformation \(\mu_{\sigma}:Q\to Q'\) associated with the elementary double
weave \(\sigma\) as follows.

If \(\sigma\) is a braid move or shuffle, apply to \(Q\) the mutation sequence
and relabeling of Shen--Weng \cite[Propositions~3.7 and~3.9]{ShenWeng2021}
that give \(Q(\mathbf b)\to Q(\mathbf b')\), and let \(Q'\) be its target.

If \(\sigma:i^\epsilon i^\epsilon\to i^\epsilon\) is a contraction, mutate at
\(k=(i,[1,2])\), then relabel by
\[
 \begin{gathered}
  (i,[0,1])\longmapsto(i,[0,1]),\qquad
  (i,[1,2])\longmapsto\sigma,\qquad
  (i,[2,\infty])\longmapsto(i,[1,\infty]),\\
  (a,[0,\infty])\longmapsto(a,[0,\infty])\quad(a\ne i)
 \end{gathered}
\]
and let \(Q'\) be the resulting quiver.  See
Figure~\ref{fig:contraction-quiver-mutation}.

In either case,
\[
 I(Q')=\Vcal(\mathbf b')\sqcup V_3(\sigma)\sqcup T.
\]
\end{definition}

\begin{figure}[htbp]
\centering
\begin{tikzpicture}[x=0.92cm,y=0.72cm]
\path[use as bounding box] (-0.20,0.30) rectangle (16.00,5.95);

\node[wqdiagramlabel] at (3.10,5.70) {$Q(i^+i^+)$};
\node[wqlevel,anchor=east] at (0.02,4.05) {$i$};
\draw[wqguide] (0.15,4.05)--(6.05,4.05);

\node[wqlargefrozen] (plus-source-left) at (1.15,4.05)
  {$[0,1]$};
\node[wqlargepivot] (plus-source-pivot) at (3.10,4.05)
  {$[1,2]$};
\node[wqlargefrozen] (plus-source-right) at (5.05,4.05)
  {$[2,\infty]$};
\draw[wqstringarrow]
  (plus-source-pivot)--(plus-source-left);
\draw[wqstringarrow]
  (plus-source-right)--(plus-source-pivot);
\node[wqdiagramlabel,above=0.7mm of plus-source-pivot] {$k$};

\draw[wqtimearrow] (6.55,4.05)--(8.85,4.05);
\node[wqmovelabel,above=1.0mm] at (7.70,4.05) {$\mu_k$};
\node[wqmovelabel,below=1.0mm] at (7.70,4.05) {$k\mapsto\sigma$};

\node[wqdiagramlabel] at (12.75,5.70)
  {$\mu_\sigma\bigl(Q(i^+i^+)\bigr)$};
\node[wqlevel,anchor=east] at (9.42,4.05) {$i$};
\draw[wqguide] (9.55,4.05)--(15.95,4.05);

\node[wqlargefrozen] (plus-target-left) at (10.85,4.05)
  {$[0,1]$};
\node[wqlargefrozen] (plus-target-right) at (14.65,4.05)
  {$[1,\infty]$};
\node[wqcontraction] (plus-target-sigma) at (12.75,4.65) {};
\draw[wqstringarrow]
  (plus-target-right)--(plus-target-left);
\wqextensiontriangle
  {plus-target-left}{plus-target-sigma}{plus-target-right}

\node[wqvlabel,above=0.8mm of plus-target-sigma] {$\sigma$};

\node[wqdiagramlabel] at (3.10,2.65) {$Q(i^-i^-)$};
\node[wqlevel,anchor=east] at (0.02,1.00) {$i$};
\draw[wqguide] (0.15,1.00)--(6.05,1.00);

\node[wqlargefrozen] (minus-source-left) at (1.15,1.00)
  {$[0,1]$};
\node[wqlargepivot] (minus-source-pivot) at (3.10,1.00)
  {$[1,2]$};
\node[wqlargefrozen] (minus-source-right) at (5.05,1.00)
  {$[2,\infty]$};
\draw[wqstringarrow]
  (minus-source-left)--(minus-source-pivot);
\draw[wqstringarrow]
  (minus-source-pivot)--(minus-source-right);
\node[wqdiagramlabel,above=0.7mm of minus-source-pivot] {$k$};

\draw[wqtimearrow] (6.55,1.00)--(8.85,1.00);
\node[wqmovelabel,above=1.0mm] at (7.70,1.00) {$\mu_k$};
\node[wqmovelabel,below=1.0mm] at (7.70,1.00) {$k\mapsto\sigma$};

\node[wqdiagramlabel] at (12.75,2.65)
  {$\mu_\sigma\bigl(Q(i^-i^-)\bigr)$};
\node[wqlevel,anchor=east] at (9.42,1.00) {$i$};
\draw[wqguide] (9.55,1.00)--(15.95,1.00);

\node[wqlargefrozen] (minus-target-left) at (10.85,1.00)
  {$[0,1]$};
\node[wqlargefrozen] (minus-target-right) at (14.65,1.00)
  {$[1,\infty]$};
\node[wqcontraction] (minus-target-sigma) at (12.75,1.60) {};
\draw[wqstringarrow]
  (minus-target-left)--(minus-target-right);
\wqextensiontriangle
  {minus-target-right}{minus-target-sigma}{minus-target-left}

\node[wqvlabel,above=0.8mm of minus-target-sigma] {$\sigma$};
\end{tikzpicture}
\caption{The cluster transformation associated with a contraction.  The top
row is the case \(\sigma:i^+i^+\to i^+\), and the bottom row is the case
\(\sigma:i^-i^-\to i^-\).  In both cases, after mutation at \(k=(i,[1,2])\),
the vertex \(k\) is relabeled as \(\sigma\).}
\label{fig:contraction-quiver-mutation}
\end{figure}
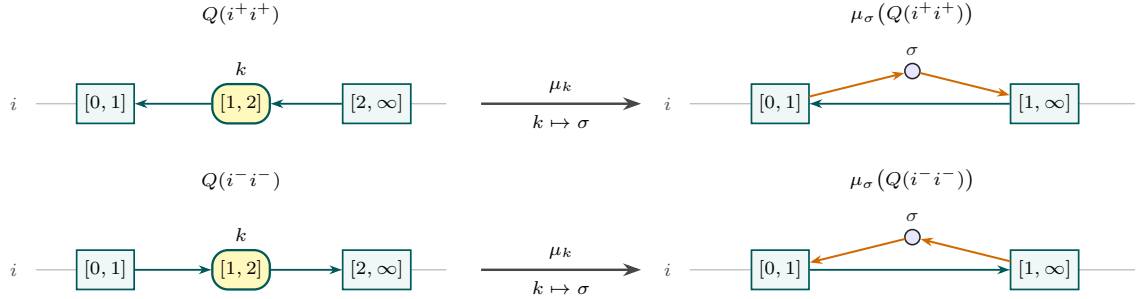

\begin{lemma}[One-step transition of extension data]
\label{lem:basic-mutation-step-with-extra-vertices}
Under the setup of Definition~\ref{def:basic-mutation-step-with-extra-vertices},
set
\[
 T':=I(Q')\setminus\Vcal(\mathbf b')=T\sqcup V_3(\sigma).
\]
For each \(\tau\in T'\), define
\(a'_\tau\in\mathbb N^{\operatorname{Pos}(\mathbf b')}\) by
\[
 a'_\tau
 :=
 \begin{cases}
  \Phi_\sigma(a_\tau),&\tau\in T,\\
  (1),&\tau\in V_3(\sigma)
 \end{cases}.
\]
Then
\[
 \mathcal E'
 :=\bigl(Q',(a'_\tau)_{\tau\in T'}\bigr)
\]
is an extension datum for \(Q(\mathbf b')\), and
\[
 (a'_\tau)^\vee
 =
 \begin{cases}
  \Phi_\sigma^\vee(a_\tau^\vee),&\tau\in T,\\
  (1),&\tau\in V_3(\sigma)
 \end{cases}.
\]
\end{lemma}

\begin{proof}
For a braid move or contraction with a single tag,
\cite[Lemma~6.1]{CasalsEtAlBraidVarieties}, and for a shuffle,
equation~\eqref{eq:shuffle-chart-transport}, give
\((a'_\tau)^\vee=\Phi_\sigma^\vee(a_\tau^\vee)\) for \(\tau\in T\).  At the
new vertex \(\sigma\) added by a contraction, \(d_\sigma=d_i\), so
\((a'_\sigma)^\vee=(1)\).

Because the mutation vertices of \(\mu_\sigma\) belong to
\(\Vcal(\mathbf b)\), the change in the
\(\Vcal(\mathbf b)\times\Vcal(\mathbf b)\) block is independent of the
connections with \(T\).  For a braid move or shuffle, the local cluster
transformations of Shen--Weng
\cite[Propositions~3.7 and~3.9]{ShenWeng2021} carry this block to the exchange
matrix of \(Q(\mathbf b')\).  Moreover, applying the same mutation formula to
the \(\tau\)-row and \(\tau\)-column for each \(\tau\in T\), their changes are
the min-plus tropicalizations of the factorization coordinate transformation;
by~\eqref{eq:lusztig-weight-transport}, they are carried to
\(-\operatorname{inc}_{\mathbf b'}\Phi_\sigma^\vee(a_\tau^\vee)\) and
\(\operatorname{inc}_{\mathbf b'}\Phi_\sigma(a_\tau)\), respectively.  The
seed isomorphism induced by the transpose
\cite[Proposition~3.9]{ShenWeng2021} gives the same conclusion in the minus-tag
case.

For a contraction, only the middle string is mutated.  Substituting
\eqref{eq:extension-datum-mixed-blocks},
\eqref{eq:extension-datum-dual-mixed-block}, and
\(\Phi_{ii\to i}=\Phi_{ii\to i}^\vee=\min\) into the mutation formula for the
exchange matrix, and then performing the relabeling in
Definition~\ref{def:basic-mutation-step-with-extra-vertices}, makes the block
on the string vertices the exchange matrix of \(Q(\mathbf b')\), and gives
\[
 \bigl(\eps^{Q'}_{\tau\ell}\bigr)_{\ell\in\Vcal(\mathbf b')}
 =-\operatorname{inc}_{\mathbf b'}((a'_\tau)^\vee),
 \qquad
 \bigl(\eps^{Q'}_{\ell\tau}\bigr)_{\ell\in\Vcal(\mathbf b')}
 =\operatorname{inc}_{\mathbf b'}(a'_\tau)
 \qquad(\tau\in T').
\]
The minus-tag case is the same calculation with the two input positions
interchanged.  Therefore \(\mathcal E'\) is an extension datum for
\(Q(\mathbf b')\).
\end{proof}

Next, we compare the mutation correction term in the remaining
\(T\times T\) block with the boundary pairing and local intersection pairing.
For this purpose, we record the local identities needed for the type \(A_2\)
braid move, shuffle, and contraction.  Below, write \([x]_+:=\max(x,0)\).

\begin{lemma}[Change of the boundary pairing under a type \(A_2\) braid move]
\label{lem:component-A2-finite-identity}
Suppose that \(i,j\in I\) satisfy \(c_{ij}=c_{ji}=-1\), and set
\(\mathbf b=iji\) and \(\mathbf b'=jij\).  For \(a,b\in\mathbb N^3\), set
\[
 a':=\Phi_{\mathbf b\to\mathbf b'}^\vee(a),
 \qquad
 b':=\Phi_{\mathbf b\to\mathbf b'}(b),
\]
and write \(e(x):=x_3-x_1\).  Then
\begin{equation}
 [-e(a)]_+[e(b)]_+-[e(a)]_+[-e(b)]_+
 =\Omega_{\mathbf b'}(a',b')-\Omega_{\mathbf b}(a,b).
 \label{eq:component-A2-finite-check}
\end{equation}
\end{lemma}

\begin{proof}
In type \(A_2\), \(\Phi_{iji\to jij}^\vee=\Phi_{iji\to jij}\).  Also,
$\Omega_{iji}(a,b)=\frac12\omega(a,b)$, so
\begin{align*}
 \Omega_{\mathbf b'}(a',b')-\Omega_{\mathbf b}(a,b)
 &=\frac12\bigl(\omega(a',b')-\omega(a,b)\bigr)\\
 &=[-e(a)]_+[e(b)]_+-[e(a)]_+[-e(b)]_+.
\end{align*}
Indeed, substituting \(\min(a_1,a_3)=a_1-[-e(a)]_+\) and
\(\min(b_1,b_3)=b_1-[-e(b)]_+\) into the two determinants cancels the linear
terms and leaves only the right-hand side.
\end{proof}

\begin{figure}[htbp]
 \centering
 \begin{tikzpicture}[x=0.92cm,y=0.78cm]
\path[use as bounding box] (-0.20,0.05) rectangle (17.05,8.05);

\node[wqdiagramlabel] at (3.70,7.50) {$Q$};
\node[wqweavefrozen] (source-tau) at (1.70,6.50) {};
\node[wqweavefrozen] (source-taup) at (5.70,6.50) {};
\node[wqvlabel,above left=0.8mm and 0.6mm of source-tau] {$\tau$};
\node[wqvlabel,above right=0.8mm and 0.6mm of source-taup] {$\tau'$};

\draw[wqqbox] (0.10,0.25) rectangle (7.30,3.55);
\node[wqqboxtitle] at (0.40,0.25) {$Q(iji)$};
\node[wqlevel,anchor=east] at (0.03,2.45) {$i$};
\node[wqlevel,anchor=east] at (0.03,0.95) {$j$};
\draw[wqguide] (0.17,2.45)--(7.17,2.45);
\draw[wqguide] (0.17,0.95)--(7.17,0.95);

\node[wqstringfrozen] (source-il) at (0.85,2.45)
  {$[0,1]$};
\node[wqpivot] (source-k) at (3.70,2.45)
  {$[1,3]$};
\node[wqstringfrozen] (source-ir) at (6.55,2.45)
  {$[3,\infty]$};
\node[wqstringfrozen] (source-jl) at (1.95,0.95)
  {$[0,2]$};
\node[wqstringfrozen] (source-jr) at (5.45,0.95)
  {$[2,\infty]$};

  \wqextensiontriangle[][bend left=9]
    {source-il}{source-tau}{source-k}
  \wqextensiontriangle{source-jl}{source-tau}{source-jr}
  \wqextensiontriangle[bend left=9]
    {source-k}{source-tau}{source-ir}
  \wqextensiontriangle[][bend left=9]
    {source-il}{source-taup}{source-k}
  \wqextensiontriangle{source-jl}{source-taup}{source-jr}
  \wqextensiontriangle[bend left=9]
    {source-k}{source-taup}{source-ir}
  \draw[wqstringarrow] (source-k)--(source-il);
  \draw[wqstringarrow] (source-ir)--(source-k);
  \draw[wqstringarrow] (source-jr)--(source-jl);
  \draw[wqstringhalf] (source-il)--(source-jl);
  \draw[wqstringarrow] (source-jl)--(source-k);
  \draw[wqstringarrow] (source-k)--(source-jr);
  \draw[wqstringhalf] (source-jr)--(source-ir);
  \draw[wqlocalarrow] (source-tau)--(source-taup);

\path (source-tau)--node[wqdiagramlabel,above=0.8mm] {$x$}
  (source-taup);
\path (source-il.north)--node[wqweightlabel,pos=0.28]
  {$a_1$}(source-tau);
\path (source-tau) to[bend left=9]
  node[wqweightlabel,pos=0.28]
  {$a_1$}(source-k.north);
\path (source-jl.north)--node[wqweightlabel,pos=0.28]
  {$a_2$}(source-tau);
\path (source-tau)--node[wqweightlabel,pos=0.28]
  {$a_2$}(source-jr.north);
\path (source-k.north) to[bend left=9]
  node[wqweightlabel,pos=0.28]
  {$a_3$}(source-tau);
\path (source-tau)--node[wqweightlabel,pos=0.28]
  {$a_3$}(source-ir.north);
\path (source-il.north)--node[wqweightlabel,pos=0.28]
  {$b_1$}(source-taup);
\path (source-taup) to[bend left=9]
  node[wqweightlabel,pos=0.28]
  {$b_1$}(source-k.north);
\path (source-jl.north)--node[wqweightlabel,pos=0.28]
  {$b_2$}(source-taup);
\path (source-taup)--node[wqweightlabel,pos=0.28]
  {$b_2$}(source-jr.north);
\path (source-k.north) to[bend left=9]
  node[wqweightlabel,pos=0.28]
  {$b_3$}(source-taup);
\path (source-taup)--node[wqweightlabel,pos=0.28]
  {$b_3$}(source-ir.north);

\draw[wqtimearrow] (7.72,2.25)--(8.43,2.25);
\node[wqmovelabel,above=0.9mm] at (8.075,2.25)
  {$\mu_\sigma$};

\node[wqdiagramlabel] at (12.90,7.50) {$Q'$};
\node[wqweavefrozen] (target-tau) at (9.95,6.50) {};
\node[wqweavefrozen] (target-taup) at (15.85,6.50) {};
\node[wqvlabel,above left=0.8mm and 0.6mm of target-tau] {$\tau$};
\node[wqvlabel,above right=0.8mm and 0.6mm of target-taup] {$\tau'$};

\draw[wqqbox] (8.95,0.25) rectangle (16.85,3.55);
\node[wqqboxtitle] at (9.25,0.25) {$Q(jij)$};
\node[wqlevel,anchor=east] at (8.88,2.45) {$i$};
\node[wqlevel,anchor=east] at (8.88,0.95) {$j$};
\draw[wqguide] (9.02,2.45)--(16.72,2.45);
\draw[wqguide] (9.02,0.95)--(16.72,0.95);

\node[wqstringfrozen] (target-il) at (10.55,2.45)
  {$[0,2]$};
\node[wqstringfrozen] (target-ir) at (15.25,2.45)
  {$[2,\infty]$};
\node[wqstringfrozen] (target-jl) at (9.55,0.95)
  {$[0,1]$};
\node[wqpivot] (target-k) at (12.90,0.95)
  {$[1,3]$};
\node[wqstringfrozen] (target-jr) at (16.25,0.95)
  {$[3,\infty]$};

  \wqextensiontriangle[][bend left=9]
    {target-jl}{target-tau}{target-k}
  \wqextensiontriangle{target-il}{target-tau}{target-ir}
  \wqextensiontriangle[bend left=9]
    {target-k}{target-tau}{target-jr}
  \wqextensiontriangle[][bend left=9]
    {target-jl}{target-taup}{target-k}
  \wqextensiontriangle{target-il}{target-taup}{target-ir}
  \wqextensiontriangle[bend left=9]
    {target-k}{target-taup}{target-jr}
  \draw[wqstringarrow] (target-ir)--(target-il);
  \draw[wqstringarrow] (target-k)--(target-jl);
  \draw[wqstringarrow] (target-jr)--(target-k);
  \draw[wqstringhalf] (target-jl)--(target-il);
  \draw[wqstringarrow] (target-il)--(target-k);
  \draw[wqstringarrow] (target-k)--(target-ir);
  \draw[wqstringhalf] (target-ir)--(target-jr);
  \draw[wqlocalarrow] (target-tau)--(target-taup);

\path (target-tau)--node[wqdiagramlabel,above=0.8mm]
  {$x+\Omega_{\mathbf b'}(a',b')-\Omega_{\mathbf b}(a,b)$}
  (target-taup);
\path (target-jl.north)--node[wqweightlabel,pos=0.28]
  {$a'_1$}(target-tau);
\path (target-tau) to[bend left=9]
  node[wqweightlabel,pos=0.28]
  {$a'_1$}(target-k.north);
\path (target-il.north)--node[wqweightlabel,pos=0.28]
  {$a'_2$}(target-tau);
\path (target-tau)--node[wqweightlabel,pos=0.28]
  {$a'_2$}(target-ir.north);
\path (target-k.north) to[bend left=9]
  node[wqweightlabel,pos=0.28]
  {$a'_3$}(target-tau);
\path (target-tau)--node[wqweightlabel,pos=0.28]
  {$a'_3$}(target-jr.north);
\path (target-jl.north)--node[wqweightlabel,pos=0.28]
  {$b'_1$}(target-taup);
\path (target-taup) to[bend left=9]
  node[wqweightlabel,pos=0.28]
  {$b'_1$}(target-k.north);
\path (target-il.north)--node[wqweightlabel,pos=0.28]
  {$b'_2$}(target-taup);
\path (target-taup)--node[wqweightlabel,pos=0.28]
  {$b'_2$}(target-ir.north);
\path (target-k.north) to[bend left=9]
  node[wqweightlabel,pos=0.28]
  {$b'_3$}(target-taup);
\path (target-taup)--node[wqweightlabel,pos=0.28]
  {$b'_3$}(target-jr.north);
\end{tikzpicture}
 \caption{The extended cluster transformation \(\mu_\sigma:Q\to Q'\)
 associated with the type \(A_2\) braid move \(\sigma:iji\to jij\).  The
 components \(a_\tau^\vee=(a_1,a_2,a_3)\) and
 \(a_{\tau'}=(b_1,b_2,b_3)\) change according to the tropical Lusztig map, and
 the change in the edge between \(\tau,\tau'\) is given by the local
 intersection term~\eqref{eq:braid-local}.}
 \label{fig:a2-extended-mutation-local-intersection}
\end{figure}
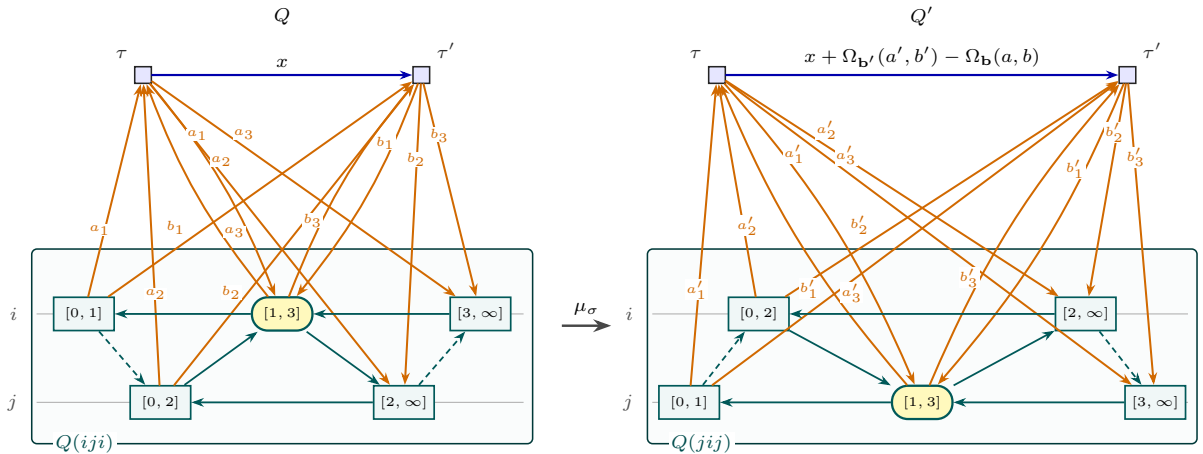

\begin{lemma}[Change of the boundary pairing under a shuffle]
\label{lem:component-shuffle-finite-identity}
Let \(i,j\in I\), set \(\mathbf b=i^+j^-\) and \(\mathbf b'=j^-i^+\), and let
\(\sigma:\mathbf b\to\mathbf b'\) be the shuffle.  For
\(a,b\in\mathbb N^{\operatorname{Pos}(\mathbf b)}\), set
\[
 a':=\Phi_\sigma^\vee(a),\qquad
 b':=\Phi_\sigma(b).
\]
Then
\[
 \Omega_{\mathbf b}(a,b)=\Omega_{\mathbf b'}(a',b')=0.
\]
If \(i=j\), then
\[
 [a_1+a_2]_+[-b_1-b_2]_+
 -[-a_1-a_2]_+[b_1+b_2]_+=0.
\]
For the reverse shuffle, the same statements hold with \(\mathbf b\) and
\(\mathbf b'\) interchanged.
\end{lemma}

\begin{proof}
The two positions have different tags, so
Definition~\ref{def:double-boundary-pairing} gives
\(\Omega_{\mathbf b}(a,b)=\Omega_{\mathbf b'}(a',b')=0\).  The second
statement follows from \(a_1,a_2,b_1,b_2\ge0\), and the reverse case reduces
to the same calculation by~\eqref{eq:shuffle-chart-transport}.
\end{proof}

\begin{lemma}[Local identity for a contraction]
\label{lem:component-contraction-finite-identity}
Let \(i\in I\), and let \(\sigma:ii\to i\) be a contraction.  Set
\[
 a=(a_1,a_2),\ b=(b_1,b_2)\in\mathbb N^2,
 \qquad
 a'=(a'_1):=\Phi_{ii\to i}^\vee(a),
 \quad
 b'=(b'_1):=\Phi_{ii\to i}(b).
\]
Thus \(a'_1=\min(a_1,a_2)\) and \(b'_1=\min(b_1,b_2)\).  Then
\begin{equation}
 [a_1-a_2]_+[b_2-b_1]_+
 -[a_2-a_1]_+[b_1-b_2]_+
 =\omega\bigl((a_1,a_2;a'_1),(b_1,b_2;b'_1)\bigr).
 \label{eq:component-contraction-existing-pair}
\end{equation}
Moreover,
\begin{equation}
 \begin{aligned}
  b_1-b_2&=\omega\bigl((0,0;1),(b_1,b_2;b'_1)\bigr),\\
  a_2-a_1&=\omega\bigl((a_1,a_2;a'_1),(0,0;1)\bigr),\\
  0&=\omega\bigl((0,0;1),(0,0;1)\bigr).
 \end{aligned}
 \label{eq:component-contraction-new-cycle-entries}
\end{equation}
\end{lemma}

\begin{proof}
\eqref{eq:component-contraction-existing-pair} follows from
\begin{align*}
 [a_1-a_2]_+[b_2-b_1]_+
   -[a_2-a_1]_+[b_1-b_2]_+
 = \det\!\begin{pmatrix}
  1&1&1\\ a_1&a_2&a'_1\\ b_1&b_2&b'_1
 \end{pmatrix}.
\end{align*}
\eqref{eq:component-contraction-new-cycle-entries} follows immediately
from the definition of \(\omega\).
\end{proof}

By these identities, the change in the \(T\times T\) block is given by the
local intersection pairing.
\begin{lemma}[Quiver mutation and the local intersection pairing]
\label{lem:extended-mutation-local-intersection}
Use the setup and notation of
Lemma~\ref{lem:basic-mutation-step-with-extra-vertices}.  Then, for
\(\tau,\tau'\in T\),
\begin{equation}
 \eps^{Q'}_{\tau\tau'}-\eps^Q_{\tau\tau'}
 =\sharp_\sigma\left(
  (a_\tau^\vee;(a'_\tau)^\vee)
  \mathbin{\cdot}(a_{\tau'};a'_{\tau'})
 \right).
 \label{eq:extra-vertex-local-intersection}
\end{equation}
Moreover, if \(\sigma\) is a contraction, then
\begin{equation}
 \eps^{Q'}_{\tau\sigma}
 =\sharp_\sigma\left(
  (a_\tau^\vee;(a'_\tau)^\vee)
  \mathbin{\cdot}(0,0;1)
 \right),
 \qquad
 \eps^{Q'}_{\sigma\tau}
 =\sharp_\sigma\left(
  (0,0;1)
  \mathbin{\cdot}(a_\tau;a'_\tau)
 \right)
 \qquad(\tau\in T).
 \label{eq:extra-vertex-new-local-intersection}
\end{equation}
\end{lemma}

\begin{proof}
The type \(A_1\times A_1\) braid move involves no mutation, and the assertion
is immediate to check.  In type \(A_2\), it follows from the mutation formula
for the exchange matrix and Lemma~\ref{lem:component-A2-finite-identity}.  The
cases of types \(B_2,G_2\) are likewise finite computations that simultaneously
apply the mutation formula for the exchange matrix to the \(\tau\)-row and the
\(\tau'\)-column, and we leave them to the reader.  For a shuffle, the result
follows from Lemma~\ref{lem:component-shuffle-finite-identity}, and for a
contraction, from Lemma~\ref{lem:component-contraction-finite-identity}.
\end{proof}

\begin{lemma}
\label{lem:component-one-step-quiver-transition}
Let \(\fW:\beta\to\delta\) be a double weave, and let
\(\fW':\beta\to\delta'\) be its one-step extension as in
\eqref{eq:double-weave-one-step-extension} by an elementary double
weave \(\sigma:\mathbf b\to\mathbf b'\).  Assume that \(\Dcal(\fW)\) is
defrostable in \(Q^{\mathrm{am,fr}}(\fW)\), and write
\begin{equation}
 \widehat Q_\fW
 :=\operatorname{defrost}_{\Dcal(\fW)}Q^{\mathrm{am,fr}}(\fW).
 \label{eq:one-step-source-defrost}
\end{equation}
Then the mutation sequence and relabeling defining \(\sigma\) induce a cluster
transformation with target \(Q'\),
\begin{equation}
 \mu_{\fW\to\fW'}:\widehat Q_\fW\longrightarrow Q'.
 \label{eq:one-step-amalgamated-mutation}
\end{equation}
Moreover, \(\Dcal(\fW')\) is defrostable in
\(Q^{\mathrm{am,fr}}(\fW')\), and
\begin{equation}
 Q'
 =\operatorname{defrost}_{\Dcal(\fW')}Q^{\mathrm{am,fr}}(\fW')
 \label{eq:component-one-step-quiver-transition}
\end{equation}
holds.
\end{lemma}

\begin{proof}
By~\eqref{eq:quiver-amalgamation-sum}
and~\eqref{eq:SW-quiver-as-amalgamation}, define a quiver \(Q\) on the vertex
set
\[
 I(Q):=\Vcal(\mathbf b)\sqcup V_3(\fW)
\]
by combining \(Q(\mathbf b)\) with the entries indexed by \(V_3(\fW)\).
For each \(\tau\in V_3(\fW)\), let \(a_\tau\) be the restriction of
\(\partial\gamma_\tau\) to the subword \(\mathbf b\).  By
\eqref{eq:dual-cycle-scaling}, its dual weight is the restriction of
\(\partial\gamma_\tau^\vee\) to \(\mathbf b\).  By
Definition~\ref{def:frozen-quiver-amalgamation} and
\eqref{eq:ext-exchange-matrix},
\[
 \mathcal E
 :=\bigl(Q,(a_\tau)_{\tau\in V_3(\fW)}\bigr)
\]
is an extension datum for \(Q(\mathbf b)\).  The mutation vertices in
Definition~\ref{def:basic-mutation-step-with-extra-vertices} map to internal
strings of \(\delta\), so they are mutable in \(\widehat Q_\fW\).  By the
nonidentification condition in Definition~\ref{def:quiver-amalgamation}, none
is identified with a vertex of another factor.  Therefore,
applying the \(\mu_\sigma\) of that definition to \(Q\) and the identity
transformation to the other factors gives the cluster transformation in
\eqref{eq:one-step-amalgamated-mutation}.

By the local conditions in Definition~\ref{def:lusztig-cycle} and
Lemma~\ref{lem:basic-mutation-step-with-extra-vertices}, the exchange-matrix
entries involving the transformed string vertices agree with the entries of
\(Q^{\mathrm{am,fr}}(\fW')\).  Here the amalgamation factors other than the
subword \(\mathbf b\) do not change.

On the output side, the restrictions of \(a_\tau^\vee\) and \(a_\tau\) are
\(\Phi_\sigma^\vee(a_\tau^\vee)\) and \(\Phi_\sigma(a_\tau)\), respectively.
Thus, for \(\tau,\tau'\in V_3(\fW)\), Lemma
\ref{lem:extended-mutation-local-intersection} and the additivity of the local
intersection pairing give
\[
 \sharp_{\fW'}\bigl(\gamma_\tau^\vee
   \mathbin{\cdot}\gamma_{\tau'}\bigr)
 =\sharp_\fW\bigl(\gamma_\tau^\vee
   \mathbin{\cdot}\gamma_{\tau'}\bigr)
  +\sharp_\sigma\left(
   (a_\tau^\vee;\Phi_\sigma^\vee(a_\tau^\vee))
   \mathbin{\cdot}
   (a_{\tau'};\Phi_\sigma(a_{\tau'}))
  \right).
\]
By Proposition~\ref{prop:double-amalgamation-exchange-matrix}, this means that
the transformed \(V_3(\fW)\times V_3(\fW)\) block agrees with that of
\(Q^{\mathrm{am,fr}}(\fW')\).

If \(\sigma\) is a contraction, then the input and output weights of the new
Langlands-dual Lusztig cycle \(\gamma_\sigma^\vee\) and Lusztig cycle
\(\gamma_\sigma\) beginning at \(\sigma\) are each \((0,0;1)\), by
Definition~\ref{def:based-lusztig-cycle}.  Therefore, the contraction formulas
in Lemmas~\ref{lem:basic-mutation-step-with-extra-vertices}
and~\ref{lem:extended-mutation-local-intersection} give the remaining entries
involving \(\sigma\).  Thus all entries of the exchange matrices agree.

The relabeling in Definition~\ref{def:basic-mutation-step-with-extra-vertices}
identifies the vertex sets and preserves the multipliers.  Moreover, the set
of mutable vertices of \(Q'\) is
\(I_{\mathrm{int}}(\delta')\sqcup V_3(\fW')\).  Hence the fact that \(Q'\) is
a quiver, together with the agreement of the exchange matrices, shows that
\(\Dcal(\fW')\) is defrostable in \(Q^{\mathrm{am,fr}}(\fW')\), and gives
\eqref{eq:component-one-step-quiver-transition}.
\end{proof}

\begin{theorem}[Sweep mutation sequences along weaves]
\label{thm:quiver}
Let \(\fW:\beta\to\delta\) be a double weave.  Then a quiver
\(Q^{\mathrm{sweep}}(\fW)\) and a sweep mutation sequence along \(\fW\)
\[
 \mu_\fW:Q(\beta)\longrightarrow Q^{\mathrm{sweep}}(\fW)
\]
can be constructed recursively with respect to the identity double weave and
one-step extensions.  In this construction, \(\Dcal(\fW)\) is defrostable in
\(Q^{\mathrm{am,fr}}(\fW)\), and
\begin{equation}
 Q^{\mathrm{sweep}}(\fW)
 =\operatorname{defrost}_{\Dcal(\fW)}
  Q^{\mathrm{am,fr}}(\fW)
 \label{eq:quiver-amalgamation-theorem}
\end{equation}
holds.
\end{theorem}

\begin{proof}
We use induction, with the identity double weave as the base case and a
one-step extension as the induction step.  At the same time as constructing
\(Q^{\mathrm{sweep}}(\fW)\) and \(\mu_\fW\), we prove that
\(\Dcal(\fW)\) is defrostable and establish
\eqref{eq:quiver-amalgamation-theorem}.  For the identity double weave,
\(\Dcal(\id_\beta)=\varnothing\) and
\(Q^{\mathrm{am,fr}}(\id_\beta)=Q(\beta)\), so it suffices to set
\[
 Q^{\mathrm{sweep}}(\id_\beta):=Q(\beta),
 \qquad
 \mu_{\id_\beta}:=\id_{Q(\beta)}.
\]

Next, let \(\fW':\beta\to\delta'\) be a one-step extension of
\(\fW:\beta\to\delta\), and assume that the sweep mutation sequence
\(\mu_\fW\) along \(\fW\) has been constructed, that \(\Dcal(\fW)\) is
defrostable, and that \eqref{eq:quiver-amalgamation-theorem} holds.
By that equation,
\[
 Q^{\mathrm{sweep}}(\fW)
 =\operatorname{defrost}_{\Dcal(\fW)}Q^{\mathrm{am,fr}}(\fW),
\]
and hence Lemma~\ref{lem:component-one-step-quiver-transition} gives a cluster
transformation
\[
 \mu_{\fW\to\fW'}:Q^{\mathrm{sweep}}(\fW)\longrightarrow Q'.
\]
Set
\begin{equation}
 Q^{\mathrm{sweep}}(\fW'):=Q',
 \qquad
 \mu_{\fW'}:=\mu_{\fW\to\fW'}\circ\mu_\fW.
 \label{eq:term-one-step-composition}
\end{equation}
By the same lemma, \(\Dcal(\fW')\) is defrostable in
\(Q^{\mathrm{am,fr}}(\fW')\), and
\[
 Q^{\mathrm{sweep}}(\fW')
 =Q'
 =\operatorname{defrost}_{\Dcal(\fW')}Q^{\mathrm{am,fr}}(\fW').
\]
This completes the induction step.
\end{proof}

Write
\[
 Q^{\mathrm{am}}(\fW)
 :=\operatorname{defrost}_{\Dcal(\fW)}Q^{\mathrm{am,fr}}(\fW).
\]
Thus Theorem~\ref{thm:quiver} gives
\(Q^{\mathrm{sweep}}(\fW)=Q^{\mathrm{am}}(\fW)\).
Let \(\Sigma_{\mathrm{sweep}}(\fW)\) be the seed obtained by applying the
sweep mutation sequence \(\mu_\fW\) to the initial Shen--Weng seed, and let
\(\Sigma_{\mathrm{sweep}}^{\mathrm{fr}}(\fW)\) be the seed obtained from it by
freezing \(\Dcal(\fW)\).

\begin{definition}[Terminal-amalgamation extension datum]
\label{def:terminal-amalgamation-extension-datum}
Let \(\fW\) be a double weave from \(\beta\) to \(\delta\).  Define the
\emph{terminal-amalgamation extension datum} by
\[
 \mathcal E_\fW
 :=\bigl(Q^{\mathrm{am}}(\fW),
         (\partial\gamma_\sigma)_{\sigma\in V_3(\fW)}\bigr).
\]
This is an extension datum for \(Q(\delta)\) by
\eqref{eq:weave-frozen-set} and
Lemma~\ref{lem:terminal-extension-datum}.
\end{definition}

We call \(\mu_\fW\) constructed in Theorem~\ref{thm:quiver} the
\emph{sweep mutation sequence along \(\fW\)}.

For every double weave \(\fW:\beta\to\delta\), the map defined by restricting
coordinates to the respective factors,
\begin{equation}
 R_\fW:
 \Tcal_{Q^{\mathrm{am}}(\fW)}
 \xrightarrow{\sim}
 \Tcal_{Q^{\mathrm{ext}}(\delta;\fW)}
 \times_{\Tcal_{\Dcal(\fW)}}
 \Tcal_{Q(\fW)}
\label{eq:amalgamation-torus-isomorphism}
\end{equation}
is an isomorphism.

\section{Cluster coordinates}
\label{sec:coordinate-compatibility}

Theorem~\ref{thm:quiver} identifies the quiver \(Q^{\mathrm{sweep}}(\fW)\) with
an amalgamation, and in this section we prove the corresponding identification
of the cluster variables.

\subsection{Cartan corrections for extension data}
\label{sec:cartan-corrections-extension-data}

\begin{definition}[\(u\)-variables]
\label{def:extension-datum-maps}
Let \(\beta=\beta(1)\cdots\beta(m)\), and let
\[
 \mathcal E=\bigl(Q,(a_\tau)_{\tau\in T}\bigr)
\]
be an extension datum for \(Q(\beta)\).  For each
\(t\in\operatorname{Pos}(\beta)\), define
\(u_t^{\mathcal E}:\Tcal_T\to\Gm\) by
\begin{equation}
 u_t^{\mathcal E}:=
 \prod_{\tau\in T}A_\tau^{a_\tau(t)}.
 \label{eq:extension-datum-u-monomial}
\end{equation}
\end{definition}

\begin{definition}[Cartan exponent]
\label{def:cartan-exponent}
Let \(\delta=\delta(1)\cdots\delta(L)\) be a double word.  For
\(0\le t\le L\), define maps
\[
 H_{t,+}^{\delta},H_{t,-}^{\delta}:
 X_*(\torusH)\times\mathbb Z^{\operatorname{Pos}(\delta)}
 \longrightarrow X_*(\torusH).
\]
For \(\eta\in X_*(\torusH)\) and
\(\lambda\in\mathbb Z^{\operatorname{Pos}(\delta)}\), their values are
determined recursively by
\begin{equation}
\begin{aligned}
 H_{0,+}^{\delta}(\eta;\lambda)&=\eta,&
 H_{t,+}^{\delta}(\eta;\lambda)&=
 \begin{cases}
  s_iH_{t-1,+}^{\delta}(\eta;\lambda)+\lambda(t)\alpha_i^\vee,
    &\delta(t)=i^+,\\
  H_{t-1,+}^{\delta}(\eta;\lambda),&\delta(t)\in I^-,
 \end{cases}\\
 H_{L,-}^{\delta}(\eta;\lambda)&=\eta,&
 H_{t-1,-}^{\delta}(\eta;\lambda)&=
 \begin{cases}
  H_{t,-}^{\delta}(\eta;\lambda),&\delta(t)\in I^+,\\
  s_iH_{t,-}^{\delta}(\eta;\lambda)+\lambda(t)\alpha_i^\vee,
    &\delta(t)=i^-.
 \end{cases}
\end{aligned}
\label{eq:cartan-exponent-recursion}
\end{equation}
where the recursive formulas apply for \(1\le t\le L\).  For
\(\ell=(i,[r,s])\in\Vcal(\delta)\), define
\[
 c_\ell^\delta:X_*(\torusH)^2
 \times\mathbb Z^{\operatorname{Pos}(\delta)}\longrightarrow\mathbb Z
\]
by
\begin{equation}
 \begin{aligned}
 c_\ell^{\delta}((\eta_+,\eta_-);\lambda)
 &:={}
 \left\langle\omega_i,H_{t,+}^\delta(\eta_+;\lambda)\right\rangle
 +\left\langle\omega_i,H_{t,-}^\delta(\eta_-;\lambda)\right\rangle,
 &&t\in\operatorname{Slice}_\delta(\ell).
 \end{aligned}
 \label{eq:cartan-exponent}
\end{equation}
Write \(c_\ell^\delta(\lambda):=c_\ell^\delta((0,0);\lambda)\), and suppress
\(\delta\) when it is clear.
\end{definition}

\begin{lemma}
\label{lem:cartan-exponent-choice-independence}
For \(\ell=(i,[r,s])\in\Vcal(\delta)\),
\(\eta\in X_*(\torusH)\),
\(\lambda\in\mathbb Z^{\operatorname{Pos}(\delta)}\), and
\(t,t'\in\operatorname{Slice}_\delta(\ell)\), we have
\[
 \left\langle\omega_i,H_{t,\pm}^\delta(\eta;\lambda)\right\rangle
 =\left\langle\omega_i,H_{t',\pm}^\delta(\eta;\lambda)\right\rangle.
\]
In particular, \(c_\ell^\delta\) is well-defined.
\end{lemma}

\begin{proof}
Between two cuts in the same \(i\)-string all letters have color \(j\ne i\).
Since \(\langle\omega_i,s_j\nu\rangle=\langle\omega_i,\nu\rangle\) and
\(\langle\omega_i,\alpha_j^\vee\rangle=0\), the assertion follows from the
forward recursion for \(+\) and the backward recursion for \(-\).
\end{proof}

Suppose \(\delta=\mathbf L\mathbf b\mathbf R\), and let
\(\lambda\in\mathbb Z^{\operatorname{Pos}(\delta)}\).  Write
\(\lambda|_{\mathbf b}(r):=\lambda(|\mathbf L|+r)\).  Each string of
\(\mathbf b\) is contained in a unique string of \(\delta\), which defines an
injection
\[
 j_{\delta,\mathbf b}:\Vcal(\mathbf b)\hookrightarrow\Vcal(\delta),
 \qquad
 (i,[r,s])\longmapsto(i,[\hat r,\hat s]),
\]
where
\[
 \hat r:=
 \begin{cases}
  |\mathbf L|+r
   & (r\in\operatorname{Pos}_i(\mathbf b)),\\
  \max\bigl(\{0\}\sqcup
   \{t\in\operatorname{Pos}_i(\delta)\mid t\le|\mathbf L|\}\bigr)
   & (r=0),
 \end{cases}
\]
\[
 \hat s:=
 \begin{cases}
  |\mathbf L|+s
   & (s\in\operatorname{Pos}_i(\mathbf b)),\\
  \min\bigl(\{t\in\operatorname{Pos}_i(\delta)
   \mid t>|\mathbf L|+|\mathbf b|\}\sqcup\{\infty\}\bigr)
   & (s=\infty).
 \end{cases}
\]
Define
\(\eta_{\mathbf b}^{\delta}(\lambda)\in X_*(\torusH)^2\) by
\begin{equation}
 \eta_{\mathbf b}^{\delta}(\lambda)
 :=\bigl(H_{|\mathbf L|,+}^{\delta}(0;\lambda),
         H_{|\mathbf L|+|\mathbf b|,-}^{\delta}(0;\lambda)\bigr).
 \label{eq:cartan-subword-values}
\end{equation}
Write \(\eta_{\mathbf b,\pm}^{\delta}(\lambda)\) for its components.

\begin{lemma}
\label{lem:cartan-subword-restriction}
For \(0\le r\le|\mathbf b|\),
\begin{equation}
 H_{|\mathbf L|+r,\pm}^{\delta}(0;\lambda)
 =H_{r,\pm}^{\mathbf b}
   \bigl(\eta_{\mathbf b,\pm}^{\delta}(\lambda);
         \lambda|_{\mathbf b}\bigr).
 \label{eq:cartan-subword-restriction}
\end{equation}
Consequently, for every \(\ell\in\Vcal(\mathbf b)\),
\begin{equation}
 c_{j_{\delta,\mathbf b}(\ell)}^\delta(\lambda)
 =c_\ell^{\mathbf b}
   \bigl(\eta_{\mathbf b}^{\delta}(\lambda);
         \lambda|_{\mathbf b}\bigr).
 \label{eq:cartan-exponent-restriction}
\end{equation}
\end{lemma}

\begin{proof}
The two identities in \eqref{eq:cartan-subword-restriction} follow by
continuing the plus recursion from the cut after \(\mathbf L\) and the minus
recursion backwards from the cut before \(\mathbf R\), respectively.  Adding
their \(\omega_i\)-pairings at corresponding cuts gives
\eqref{eq:cartan-exponent-restriction}.
\end{proof}

\begin{definition}[Cartan factors and correction]
\label{def:extension-datum-cartan-correction}
Let \(\beta=\beta(1)\cdots\beta(m)\), and let
\(\mathcal E=(Q,(a_\tau)_{\tau\in T})\) be an extension datum for
\(Q(\beta)\).  For \(\eta:T\to X_*(\torusH)\) and \(0\le t\le m\),
define the algebraic group homomorphism
\(h_{\mathcal E,t}^\pm(\eta):\Tcal_T\to\torusH\) by
\begin{equation}
 h_{\mathcal E,t}^\pm(\eta)
 :=\prod_{\tau\in T}H_{t,\pm}^\beta(\eta_\tau;a_\tau)(A_\tau).
 \label{eq:extension-datum-cartan-monomial-components}
\end{equation}

For \(\eta:T\to X_*(\torusH)^2\), write
\(\eta_\tau=(\eta_{\tau,+},\eta_{\tau,-})\) and define
\(\eta_\pm:T\to X_*(\torusH)\) by
\(\tau\mapsto\eta_{\tau,\pm}\).  For
\(\ell=(i,[r,s])\in\Vcal(\beta)\), define
\(M_\ell^{\mathcal E}(\eta):\Tcal_T\to\Gm\) by
\begin{equation}
 \begin{aligned}
 M_\ell^{\mathcal E}(\eta)
 &:={}
 \prod_{\tau\in T}A_\tau^{c_\ell^{\beta}(\eta_\tau;a_\tau)}\\
 &=\omega_i(h_{\mathcal E,t}^+(\eta_+))
   \omega_i(h_{\mathcal E,t}^-(\eta_-)),
 \qquad t\in\operatorname{Slice}_\beta(\ell),
 \end{aligned}
 \label{eq:extension-datum-cartan-monomial}
\end{equation}
where the second equality follows from
\eqref{eq:extension-datum-cartan-monomial-components} and
\eqref{eq:cartan-exponent}.  Write
\(M_\ell^{\mathcal E}:=M_\ell^{\mathcal E}(0)\).  Below, we use the same
notation for the pullbacks of \(u_t^{\mathcal E}\),
\(h_{\mathcal E,t}^\pm(\eta)\), and \(M_\ell^{\mathcal E}(\eta)\) along
\(\Tcal_Q\to\Tcal_T\).
\end{definition}

\begin{lemma}
\label{lem:cartan-exponent-exchange-identity}
For every \(\eta\in X_*(\torusH)^2\), every
\(\lambda\in\mathbb Z^{\operatorname{Pos}(\delta)}\), and
\(k\in I_{\mathrm{int}}(\delta)\),
\begin{equation}
 \sum_{\ell\in\Vcal(\delta)}
 \eps^{Q(\delta)}_{k,\ell}c_\ell^{\delta}(\eta;\lambda)
 =-\bigl(\operatorname{inc}_\delta\lambda\bigr)_k.
 \label{eq:cartan-exponent-exchange-identity}
\end{equation}
\end{lemma}

\begin{proof}
Let \(k=(i,[r,s])\), and suppress \(\delta,\eta,\lambda\) from the
notation.  Set
\[
 \epsilon_t:=\operatorname{sgn}(\delta(t)),\qquad
 \kappa(t):=\langle\alpha_i,H_{t,+}-H_{t,-}\rangle,\qquad
 \tau(t):=\langle\alpha_i,H_{t,+}+H_{t,-}\rangle.
\]
The backward recursion for a minus letter is equivalently
\(H_{t,-}=s_iH_{t-1,-}+\lambda(t)\alpha_i^\vee\).  Expanding
\eqref{eq:SW-quiver-as-amalgamation} and
\eqref{eq:quiver-amalgamation-sum}, the contributions at \(r\), at
\(r<t<s\), and at \(s\), using
\eqref{eq:SW-basic-horizontal}--\eqref{eq:SW-basic-other}, are
\[
 \frac{\epsilon_r}{2}\tau(r-1),\qquad
 -\frac12\bigl(\kappa(t)-\kappa(t-1)\bigr),\qquad
 -\frac{\epsilon_s}{2}\tau(s),
\]
respectively.  Summing gives
\[
 \sum_{\ell\in\Vcal(\delta)}
 \eps^{Q(\delta)}_{k,\ell}c_\ell^\delta(\eta;\lambda)
 =\frac{\epsilon_r}{2}\tau(r-1)
  -\frac12\bigl(\kappa(s-1)-\kappa(r)\bigr)
  -\frac{\epsilon_s}{2}\tau(s).
\]
The recursions at the two positions of color \(i\) give
\[
 \kappa(r)=-\epsilon_r\tau(r-1)+2\epsilon_r\lambda(r),
 \qquad
 \kappa(s-1)=-\epsilon_s\tau(s)+2\epsilon_s\lambda(s).
\]
Substituting and using \eqref{eq:component-terminal-incidence} proves the
assertion.
\end{proof}

\begin{lemma}
\label{lem:cartan-correction-exchange-homogeneity}
For every \(\eta:T\to X_*(\torusH)^2\) and every
\(i\in I_{\mathrm{uf}}(Q)\cap\Vcal(\beta)\),
\begin{equation}
 \prod_{\ell\in\Vcal(\beta)}
 M_\ell^{\mathcal E}(\eta)^{\eps^Q_{i\ell}}
 \prod_{\tau\in T}A_\tau^{\eps^Q_{i\tau}}=1.
 \label{eq:component-cartan-monomial-identity}
\end{equation}
\end{lemma}

\begin{proof}
The exponent of \(A_\tau\) is
\[
 \sum_{\ell\in\Vcal(\beta)}
 \eps^Q_{i\ell}c_\ell^{\beta}(\eta_\tau;a_\tau)+\eps^Q_{i\tau}
 =-(\operatorname{inc}_\beta a_\tau)_i
  +(\operatorname{inc}_\beta a_\tau)_i=0
\]
by Lemma~\ref{lem:cartan-exponent-exchange-identity} and
\eqref{eq:extension-datum-mixed-blocks}.
\end{proof}

\begin{example}
For the weave \(\fW:1122112\to212\) in \eqref{eq:a2-direct-weave}, consider
the final quiver in Figure~\ref{fig:a2-five-step-mutation-late}.  The boundary weights of
the Lusztig cycles are
\[
 \partial\gamma_{\sigma_1}=(0,0,0),\qquad
 \partial\gamma_{\sigma_2}=\partial\gamma_{\sigma_3}=(1,0,0),\qquad
 \partial\gamma_{\sigma_4}=(0,0,1).
\]
Hence the \(u\)-variables of \(\mathcal E_\fW\) are
\[
 (u_1,u_2,u_3)
 =\bigl(A_{\sigma_2}A_{\sigma_3},1,A_{\sigma_4}\bigr).
\]
The corresponding plus Cartan factors in \(\torusH\subset\mathrm{SL}_3\) are
\[
\begin{alignedat}{2}
 h_0^+&=\operatorname{diag}(1,1,1),\qquad&
 h_1^+&=\operatorname{diag}\left(
  1,A_{\sigma_2}A_{\sigma_3},
  (A_{\sigma_2}A_{\sigma_3})^{-1}\right),\\
 h_2^+&=\operatorname{diag}\left(
  A_{\sigma_2}A_{\sigma_3},1,
  (A_{\sigma_2}A_{\sigma_3})^{-1}\right),\qquad&
 h_3^+&=\operatorname{diag}\left(
  A_{\sigma_2}A_{\sigma_3},
  \frac{A_{\sigma_4}}{A_{\sigma_2}A_{\sigma_3}},
  A_{\sigma_4}^{-1}\right).
\end{alignedat}
\]
The Cartan corrections for the five strings of \(212\) are
\[
\begin{array}{c|ccccc}
 \ell
 &(2,[0,1])&(1,[0,2])&(2,[1,3])&(1,[2,\infty])&(2,[3,\infty])\\ \hline
 M_\ell
 &1&1&A_{\sigma_2}A_{\sigma_3}&A_{\sigma_2}A_{\sigma_3}&A_{\sigma_4}
\end{array}
\]
\end{example}

\subsection{Amalgamated coordinates}

Let \(\fW\) be a double weave from \(\beta\) to \(\delta\).  We do not
assume that the terminal word \(\delta\) is reduced.

For every \(t\in\operatorname{Pos}(\delta)\), \(\ell\in\Vcal(\delta)\), and
\(s\in\{0\}\sqcup\operatorname{Pos}(\delta)\), define rational maps
\[
 u_t^\fW,M_\ell^\fW:
 X(p)\times X(q^{\mathrm{op}})\dashrightarrow\Gm,
 \qquad
 h_s^{\fW,+},h_s^{\fW,-}:
 X(p)\times X(q^{\mathrm{op}})\dashrightarrow\torusH
\]
by
\[
 u_t^\fW:=\mathbf A_\fW^*u_t^{\mathcal E_\fW},\qquad
 M_\ell^\fW:=\mathbf A_\fW^*M_\ell^{\mathcal E_\fW},\qquad
 h_s^{\fW,\pm}:=\mathbf A_\fW^*h_{\mathcal E_\fW,s}^\pm(0).
\]
By \eqref{eq:extension-datum-u-monomial}
and~\eqref{eq:extension-datum-cartan-monomial} and the definition of
\(\Dcal(\fW)\),
\begin{equation}
 \begin{aligned}
  u_t^\fW
  &=\prod_{\sigma\in\Dcal(\fW)}
    (A_\sigma^\fW)^{(\partial\gamma_\sigma)(t)},
  &\qquad
  M_\ell^\fW
  &=\prod_{\sigma\in\Dcal(\fW)}
    (A_\sigma^\fW)^{c_\ell^\delta(\partial\gamma_\sigma)}.
 \end{aligned}
 \label{eq:weave-terminal-framing-monomial}
\end{equation}
The first identity in \eqref{eq:weave-terminal-framing-monomial}, on
the plus component, is the formula for \(u_e\) in
\cite[Section~5]{CasalsEtAlBraidVarieties} restricted to the terminal edges, and on the minus
component, it is the transpose of that formula.
\begin{definition}[Amalgamated coordinate map]
\label{def:amalgamated-functions}
For \(c\in\Conf(p',q')\) and
\(x\in X(p)\times X(q^{\mathrm{op}})\), define the amalgamated coordinates
corresponding to \(\ell\in\Vcal(\delta)\) and \(\sigma\in V_3(\fW)\) by
\begin{equation}
 A_\ell^{\mathrm{am}}(c,x)
 :=
 A_\ell^\delta(c)\,M_\ell^\fW(x),
 \label{eq:terminal-coordinate}
\end{equation}
and
\[
 A_\sigma^{\mathrm{am}}(c,x):=A_\sigma^\fW(x).
\]
Combining these coordinates gives a rational map
\begin{equation}
 \mathbf A_{\delta,\fW}^{\mathrm{am}}
 :=
 (A_v^{\mathrm{am}})_{v\in\Vcal(\delta)\sqcup V_3(\fW)}:
 \Conf(p',q')
 \times X(p)\times X(q^{\mathrm{op}})
 \dashrightarrow\Tcal_{Q^{\mathrm{am}}(\fW)}.
 \label{eq:def-Aam}
\end{equation}
Under the inverse \(R_\fW^{-1}\) of the torus isomorphism
in~\eqref{eq:amalgamation-torus-isomorphism}, this is the map obtained by
identifying the extended Shen--Weng coordinates and the weave coordinates on
the shared components \(A_\sigma^{\mathrm{ext}}=A_\sigma^\fW\).

When \(\fW_-=\id_{\mathbf q}\), the coordinates in \(V_3(\fW)\) and the
Cartan factors \(h_t^{\fW,+},h_t^{\fW,-}\) depend only on \(X(p)\).  Hence
\(\mathbf A_{\delta,\fW}^{\mathrm{am}}\) factors uniquely through the map
that forgets the second braid-variety factor.  We also denote this
factorization by
\[
 \mathbf A_{\delta,\fW}^{\mathrm{am}}:
 \Conf(p',q)\times X(p)
 \dashrightarrow\Tcal_{Q^{\mathrm{am}}(\fW)}.
\]
Similarly, when \(\fW_+=\id_{\mathbf p}\), we also write
\[
 \mathbf A_{\delta,\fW}^{\mathrm{am}}:
 \Conf(p,q')
 \times X(q^{\mathrm{op}})
 \dashrightarrow\Tcal_{Q^{\mathrm{am}}(\fW)}.
\]
\end{definition}

\begin{proposition}
If \(\mathbf p',\mathbf q'\) are reduced words for their respective Demazure
products, then the coordinate transformation
\[
 \mathbf A_{\delta,\fW}^{\mathrm{am}}
 \circ
 \bigl(\mathbf A_\delta^{\mathrm{SW}}\times\mathbf A_\fW\bigr)^{-1}:
 \Tcal_{Q(\delta)}\times\Tcal_{Q(\fW)}
 \xrightarrow{\sim}\Tcal_{Q^{\mathrm{am}}(\fW)}
\]
is an isomorphism of algebraic tori.  Its inverse is given by
\begin{equation}
 A_\sigma=A_\sigma^{\mathrm{am}}
 \quad(\sigma\in V_3(\fW)),\qquad
 A_\ell^\delta
 =
 A_\ell^{\mathrm{am}}(M_\ell^\fW)^{-1}
 \quad(\ell\in\Vcal(\delta)).
 \label{eq:amalgamated-to-product}
\end{equation}
In particular, \(\mathbf A_{\delta,\fW}^{\mathrm{am}}\) is a birational map.
\end{proposition}

\begin{proof}
By \eqref{eq:weave-terminal-framing-monomial}, \(M_\ell^\fW\) is a
Laurent monomial in the double-weave coordinates.  Hence
\eqref{eq:amalgamated-to-product} is the Laurent-monomial inverse of
\eqref{eq:terminal-coordinate}.  The last assertion follows from the
birationality of
\(\mathbf A_\delta^{\mathrm{SW}}\times\mathbf A_\fW\).
\end{proof}

Under the assumptions of the proposition, set
\[
 \Sigma_{\mathrm{am}}(\fW):=\Sigma_{Q^{\mathrm{am}}(\fW)}.
\]
The map \(\mathbf A_{\delta,\fW}^{\mathrm{am}}\) in
\eqref{eq:def-Aam} is a realization of this seed.

Set
\[
 u:=\operatorname{dem}(p)
   =\operatorname{dem}(p'),
 \qquad
 v:=\operatorname{dem}(q)
   =\operatorname{dem}(q').
\]
In the construction below, we do not assume that \(\mathbf p',\mathbf q'\) are
reduced.

\begin{definition}[Flag propagation along a double weave]
\label{def:component-terminal-propagation}
Let \(y=[\flagA^0;\flagB^\bullet;\flagB_\bullet;\flagA_n]\) be an element of
\(\Conf(p,q)_{u,v}\).  Apply the top-to-bottom construction of framed flags in
\cite[Section~5]{CasalsEtAlBraidVarieties} to the upper flag sequence, stopping at
the terminal end of \(\fW_+\).  Apply the same construction for
\(\fW_-^{\mathrm{op}}\) to the lower flag sequence and transpose it, and at a
shuffle relabel the two flag sequences.  Denote the resulting flag
configuration by \(\operatorname{prop}_{\fW}(y)\), and write the rational map
thus defined as
\[
 \operatorname{prop}_{\fW}:
 \Conf(p,q)_{u,v}
 \dashrightarrow
 \Conf(p',q')_{u,v}.
\]
In non-simply-laced type, use the construction in
\cite[Section~6.1]{CasalsEtAlBraidVarieties}.
\end{definition}

\begin{proposition}[Flag propagation and one-step extensions]
\label{prop:component-terminal-propagation}
For a one-step extension \(\fW':\beta\to\delta'\) as in
\eqref{eq:double-weave-one-step-extension}, write
\(\mathbf p'',\mathbf q''\) for the words obtained from \(\delta'\), and
\(p'',q''\) for the positive braids represented by these words.  The flag
propagation associated with
\(\id_{\mathbf L}\otimes\sigma\otimes\id_{\mathbf R}\) satisfies
\begin{equation}
 \begin{gathered}
  \operatorname{prop}_{
   \id_{\mathbf L}\otimes\sigma\otimes\id_{\mathbf R}}:
  \Conf(p',q')_{u,v}
  \dashrightarrow
  \Conf(p'',q'')_{u,v},\\
  \operatorname{prop}_{\fW'}
  =\operatorname{prop}_{
    \id_{\mathbf L}\otimes\sigma\otimes\id_{\mathbf R}}
   \circ\operatorname{prop}_{\fW}.
 \end{gathered}
 \label{eq:component-one-step-terminal-propagation}
\end{equation}
Moreover, under a one-step extension of a double weave, for every existing
\(\tau\in V_3(\fW)\),
\begin{equation}
 A_\tau^{\fW'}=A_\tau^\fW.
 \label{eq:component-existing-coordinate-preservation}
\end{equation}
If \(\fW=\id_\beta\), then \(\operatorname{prop}_{\fW}=\id\).  Furthermore,
if \(\mathbf p',\mathbf q'\) are reduced words for \(u,v\), respectively,
\begin{equation}
 \operatorname{prop}_{\fW}
 =\operatorname{red}
 \label{eq:reduced-terminal-propagation}
\end{equation}
holds.
\end{proposition}

\begin{proof}
The construction cited in Definition~\ref{def:component-terminal-propagation}
is a top-to-bottom recursion.  Thus, continuing with the final local move gives
\eqref{eq:component-one-step-terminal-propagation}.  Since the trivalent
vertex coordinates already introduced do not change, it also gives
\eqref{eq:component-existing-coordinate-preservation}.  On the minus
component, transpose this argument, and at a shuffle relabel the vertices.  The
assertion for the identity double weave also follows from the construction.
Finally, if \(\mathbf p',\mathbf q'\) are reduced, then the propagated flag
sequences and the flag sequences in Definition~\ref{def:double-reduction} have
the same endpoints and the same reduced types.  The uniqueness in
Lemma~\ref{lem:reduced-flag-chain} therefore gives
\eqref{eq:reduced-terminal-propagation}.
\end{proof}

\begin{definition}[Coordinate map on the initial configuration]
\label{def:component-intermediate-functions}
Using the braid-variety factor map of
Definition~\ref{def:double-braid-factors}, define the coordinate map on the
initial configuration by
\begin{equation}
 \mathbf A^\fW
 :=
 \mathbf A_{\delta,\fW}^{\mathrm{am}}
 \circ(\operatorname{prop}_{\fW},\operatorname{braid}):
 \Conf(p,q)_{u,v}
 \dashrightarrow\Tcal_{Q^{\mathrm{am}}(\fW)},
 \label{eq:component-partial-coordinate-map}
\end{equation}
and write \(A_v^\fW\) for its \(v\)-component.
\end{definition}

\begin{lemma}
\label{lem:component-terminal-functions-well-defined}
If \(\fW=\id_\beta\) and \(\delta=\beta\), then
\begin{equation}
 \mathbf A^\fW=\mathbf A_\beta^{\mathrm{SW}}.
 \label{eq:component-identity-coordinate-map}
\end{equation}
\end{lemma}

\begin{proof}
For the identity double weave, \(V_3(\fW)=\varnothing\) and
\(\operatorname{prop}_{\fW}=\id\), while
\(\mathbf A_{\beta,\fW}^{\mathrm{am}}\) is precisely the Shen--Weng
coordinate map.  This gives
\eqref{eq:component-identity-coordinate-map}.
\end{proof}

\subsection{Coordinate compatibility of sweep mutation sequences along weaves}
\label{subsec:coordinate-compatibility}

For an ordinary Demazure weave \(\fW:\mathbf p\to\mathbf p'\), write
\[
 \widetilde z_e^{\fW}:X(p)\dashrightarrow\mathbb A^1,
 \qquad
 u_e^{\fW}:X(p)\dashrightarrow\Gm
\]
for the two labels of a solid weave edge \(e\) in the top-to-bottom
construction of framed flags in
\cite[Section~5]
{CasalsEtAlBraidVarieties}.  For a double weave, we use the labels of \(\fW_+\)
on the plus component.  On the minus component, we use the labels of
\(\fW_-^{\mathrm{op}}\) via the horizontal reflection in
Proposition~\ref{prop:double-weave-plus-minus-decomposition}.

\begin{lemma}
\label{lem:framed-edge-cartan-transport}
Let \(i\in I\), \(g,g'\in\groupG\), \(h,h'\in\torusH\),
\(z\in\mathbb G_a\), and \(u\in\Gm\) satisfy
\[
 g'=gB_i(z),
 \qquad
 h'=s_i(h)\alpha_i^\vee(u).
\]
Set \(\widetilde z:=\alpha_i(h)^{-1}z\).  Then
\begin{equation}
 g'h'=ghB_i(\widetilde z)\alpha_i^\vee(u).
 \label{eq:framed-edge-cartan-action}
\end{equation}
In particular, the framed flags represented by \(gh\) and \(g'h'\) satisfy
the compatibility condition of
\cite[Section~5]{CasalsEtAlBraidVarieties} for a solid edge labeled by
\((\widetilde z,u)\).
\end{lemma}

\begin{proof}
The definitions of \(B_i\) and the Weyl-group action on \(\torusH\) give
\[
 B_i(z)s_i(h)=hB_i\bigl(\alpha_i(h)^{-1}z\bigr).
\]
Hence
\[
 g'h'
 =gB_i(z)s_i(h)\alpha_i^\vee(u)
 =ghB_i(\widetilde z)\alpha_i^\vee(u).
\]
\end{proof}

Let \(\mathcal E=(Q,(a_\tau)_{\tau\in T})\) be an extension datum for
\(Q(\mathbf b)\), and let \(\eta:T\to X_*(\torusH)^2\).  For
\(\tau\in T\), define
\(\deg_{\mathcal E,\tau}(\eta_\tau)\in\mathbb Z^{I(Q)}\) by
\[
 (\deg_{\mathcal E,\tau}(\eta_\tau))_v
 =\begin{cases}
 c_v^{\mathbf b}(\eta_\tau;a_\tau),&v\in\Vcal(\mathbf b),\\
 1,&v=\tau,\\
 0,&v\in T\setminus\{\tau\}.
 \end{cases}
\]

\begin{lemma}
\label{lem:cartan-degree-tropical-transport}
Let \(\sigma:\mathbf b\to\mathbf b'\) be a braid move or a shuffle, and let
\(\mathcal E=(Q,(a_\tau)_{\tau\in T})\) be an extension datum for
\(Q(\mathbf b)\).  Let \(\eta:T\to X_*(\torusH)^2\).  Let
\(\mu_\sigma:Q\to Q'\) be the extended cluster
transformation of Definition~\ref{def:basic-mutation-step-with-extra-vertices},
and let \(\mathcal E'=(Q',(a'_\tau)_{\tau\in T})\) be the extension datum
obtained in Lemma~\ref{lem:basic-mutation-step-with-extra-vertices}.  Then the
integer min-plus tropicalization
\[
 \operatorname{Trop}(\mu_\sigma):
 \mathbb Z^{I(Q)}\longrightarrow\mathbb Z^{I(Q')}
\]
satisfies
\begin{equation}
 \operatorname{Trop}(\mu_\sigma)
 \bigl(\deg_{\mathcal E,\tau}(\eta_\tau)\bigr)
 =\deg_{\mathcal E',\tau}(\eta_\tau)
 \qquad(\tau\in T).
 \label{eq:cartan-degree-tropical-transport}
\end{equation}
\end{lemma}

\begin{proof}
Fix \(\tau\in T\), replace \(A_\tau\) by \(zA_\tau\), and leave the other
\(A_{\tau'}\) fixed.  The edge labels \(u_t^{\mathcal E}\) are then multiplied
by \(z^{a_\tau(t)}\), while the two boundary decorations are replaced by their
respective \(\eta_{\tau,-}(z)\)- and \(\eta_{\tau,+}(z)\)-translates.  By
Lemma~\ref{lem:framed-edge-cartan-transport}, Lemma
\ref{lem:minor-cartan-semi-invariance}, and
\eqref{eq:extension-datum-cartan-monomial}, the source degrees are
\(\deg_{\mathcal E,\tau}(\eta_\tau)\).  The framed braid and shuffle
identities carry the edge exponents from \(a_\tau\) to \(a'_\tau\) by
\eqref{eq:lusztig-weight-transport}, without changing the two boundary
decorations.  The Shen--Weng pair-minor identities
\cite[Section~3]{ShenWeng2021} therefore give target
degree \(\deg_{\mathcal E',\tau}(\eta_\tau)\).  Finally, Lemma
\ref{lem:cartan-correction-exchange-homogeneity} makes every exchange relation
homogeneous in \(z\), so its target degree is the integer min-plus
tropicalization of its source degree.  For a move in the minus component, the
same argument applies after horizontal reflection, which reverses the local
cuts, exchanges the boundary decorations, and turns the backward recursion in
\eqref{eq:cartan-exponent-recursion} into the forward recursion while
exchanging the two \(\omega_i\)-summands in
\eqref{eq:cartan-exponent}.  This proves
\eqref{eq:cartan-degree-tropical-transport}.
\end{proof}

For the rest of this subsection, let \(\fW\) be a double weave from \(\beta\)
to \(\delta\), and let
\[
 \fW':\beta\longrightarrow\delta',
 \qquad
 \delta=\mathbf L\,\mathbf b\,\mathbf R,
 \qquad
 \delta'=\mathbf L\,\mathbf b'\,\mathbf R,
 \qquad
 \sigma:\mathbf b\longrightarrow\mathbf b'
\]
be a one-step extension as in
\eqref{eq:double-weave-one-step-extension}.  The coordinate maps
\(\mathbf A^\fW\) and \(\mathbf A^{\fW'}\) are those of
Definition~\ref{def:component-intermediate-functions}.

\begin{theorem}
\label{thm:one-step-coordinate-term}
For the one-step extension \(\fW\to\fW'\) specified above, regardless of which
elementary double weave \(\sigma\) occurs, under the correspondence with the
vertices of the target quiver determined by \(\mu_{\fW\to\fW'}\),
\begin{equation}
 \mathbf A^{\fW'}
 =\mu_{\fW\to\fW'}\circ\mathbf A^\fW:
 \Conf(p,q)_{u,v}
 \dashrightarrow\Tcal_{Q^{\mathrm{am}}(\fW')}
 \label{eq:one-step-coordinate-compatibility}
\end{equation}
holds.
\end{theorem}

\begin{proof}
Use the quiver \(Q\) defined in the proof of
Lemma~\ref{lem:component-one-step-quiver-transition}.  Here
\(T=V_3(\fW)\).  For each \(\tau\in T\), let \(a_\tau\) be the restriction
of \(\partial\gamma_\tau\) to \(\mathbf b\).  Then
\[
 \mathcal E:=\bigl(Q,(a_\tau)_{\tau\in T}\bigr)
\]
is an extension datum for \(Q(\mathbf b)\).  By
\eqref{eq:dual-cycle-scaling}, \(a_\tau^\vee\) is the restriction of
\(\partial\gamma_\tau^\vee\) to \(\mathbf b\).  Set
\(\eta:T\to X_*(\torusH)^2\) by
\[
 \eta_\tau:=\eta_{\mathbf b}^\delta(\partial\gamma_\tau).
\]
Let \(\mathcal E'\) be the extension datum on the target side obtained in
Lemma~\ref{lem:basic-mutation-step-with-extra-vertices}.

Let \(y\in\Conf(p,q)_{u,v}\) be generic, and set
\(c=\operatorname{prop}_\fW(y)\) and
\(y_{\mathrm{br}}=\operatorname{braid}(y)\).  Let \(x\in\Tcal_Q\) be the
restriction of \(\mathbf A^\fW(y)\) to the factor \(Q\).  By
\eqref{eq:terminal-coordinate}, \eqref{eq:cartan-exponent-restriction}, and
\eqref{eq:extension-datum-cartan-monomial},
\begin{equation}
 A_\ell(x)
 =A_{j_{\delta,\mathbf b}(\ell)}^\delta(c)
  M_\ell^{\mathcal E}(\eta)(x)
 \qquad(\ell\in\Vcal(\mathbf b)).
 \label{eq:local-coordinate-cartan-factor}
\end{equation}
\eqref{eq:cartan-subword-restriction}
and~\eqref{eq:extension-datum-cartan-monomial-components} also give
\begin{equation}
 h_{\mathcal E,t}^\pm(\eta_\pm)(x)
 =h_{|\mathbf L|+t}^{\fW,\pm}(y_{\mathrm{br}})
 \qquad(0\le t\le|\mathbf b|).
 \label{eq:local-cartan-factor-restriction}
\end{equation}

We identify the left-hand side of
\eqref{eq:local-coordinate-cartan-factor} with the pair minor in the framed
construction.  Let \(\flagA_t^-\) and \(\flagA_t^+\) be the compatible
decorations on the diagonal at the ambient cut \(|\mathbf L|+t\) used in
\cite[Definition~3.20]{ShenWeng2021} to define
\(A_{j_{\delta,\mathbf b}(\ell)}^\delta(c)\), where
\(\ell=(i,[r,s])\) and
\(t\in\operatorname{Slice}_{\mathbf b}(\ell)\).  By
Lemma~\ref{lem:minor-cartan-semi-invariance},
\begin{equation}
 \Delta_i\left(
   h_{\mathcal E,t}^-(\eta_-)(x)\mathbin{\cdot}\flagA_t^-,
   \flagA_t^+\mathbin{\cdot}h_{\mathcal E,t}^+(\eta_+)(x)
  \right)=
 A_{j_{\delta,\mathbf b}(\ell)}^\delta(c)
 M_\ell^{\mathcal E}(\eta)(x)
 =A_\ell(x).
 \label{eq:local-framed-pair-minor}
\end{equation}
By \cite[Proposition~3.11, Definition~5.8, Lemma~5.9, and Theorem~5.12]
{CasalsEtAlBraidVarieties}, the framed labels determine a unique compatible
collection of framed flags.  Together with
\(h_0^{\fW,+}=h_{|\delta|}^{\fW,-}=1\),
Lemma~\ref{lem:framed-edge-cartan-transport},
\eqref{eq:weave-terminal-framing-monomial}, and
\eqref{eq:local-cartan-factor-restriction} identify their restrictions with
the two \(H\)-translates in \eqref{eq:local-framed-pair-minor}.  On the minus
component, this identification is obtained after transposition on
\(\fW_-^{\mathrm{op}}\).

For a braid move or shuffle, the transformed boundary weight agrees with
\(\partial\gamma_\tau\) outside \(\mathbf b\) by
\eqref{eq:lusztig-weight-transport}.  The two entries in
\eqref{eq:cartan-subword-values} use only the restrictions to
\(\mathbf L\) and \(\mathbf R\), respectively.  Hence the target factor has
the same \(\eta_\tau\).  In \eqref{eq:local-coordinate-cartan-factor}, the
Shen--Weng pair-minor identities
\cite[Proposition~3.25]{ShenWeng2021} transform the first factor, while
Lemma~\ref{lem:cartan-correction-exchange-homogeneity} factors out the common
Cartan monomial and Lemma~\ref{lem:cartan-degree-tropical-transport}
identifies its target exponent.  Thus the string coordinates transform by the
required mutations and relabeling.

Suppose next that \(\sigma:i^+i^+\to i^+\).  Write
\[
 k_-=(i,[0,1]),\qquad k=(i,[1,2]),\qquad k_+=(i,[2,\infty]),
\]
and let \(u_1,u_2,u_s\) be the \(u\)-labels on the two input edges and the
output edge, respectively.  Evaluating the framed trivalent identity of
\cite[Definition~5.8 and Lemma~5.9]{CasalsEtAlBraidVarieties} in the
\(i\)-th fundamental representation against the common negative decoration,
and using \eqref{eq:local-framed-pair-minor}, gives
\begin{equation}
 A_k(x)u_s=A_{k_-}(x)u_2+A_{k_+}(x)u_1.
 \label{eq:geometric-contraction-identity}
\end{equation}

Set \(m_\tau=\min(a_\tau(1),a_\tau(2))\).  At the output edge,
\(a'_\sigma(1)=1\) and \(a'_\tau(1)=m_\tau\).  By
\eqref{eq:weave-terminal-framing-monomial} and
\eqref{eq:component-existing-coordinate-preservation},
\[
 u_r=\prod_{\tau\in T}A_\tau(x)^{a_\tau(r)}\quad(r=1,2),
 \qquad
 u_s=A_\sigma^{\fW'}(y_{\mathrm{br}})
     \prod_{\tau\in T}A_\tau(x)^{m_\tau}.
\]
Since
\(m_\tau+[a_\tau(2)-a_\tau(1)]_+=a_\tau(2)\) and
\(m_\tau+[a_\tau(1)-a_\tau(2)]_+=a_\tau(1)\), the extended mutation of
Definition~\ref{def:basic-mutation-step-with-extra-vertices}, with the target
exponents from Lemma~\ref{lem:basic-mutation-step-with-extra-vertices}, gives
\[
 \begin{aligned}
 &A_k(x)(\mu_\sigma^*A_\sigma)(x)
  \prod_{\tau\in T}A_\tau(x)^{m_\tau}\\
 &\qquad={}
  A_{k_-}(x)\prod_{\tau\in T}A_\tau(x)^{a_\tau(2)}
  +A_{k_+}(x)\prod_{\tau\in T}A_\tau(x)^{a_\tau(1)}\\
 &\qquad={}
  A_{k_-}(x)u_2+A_{k_+}(x)u_1.
 \end{aligned}
\]
Comparing this with \eqref{eq:geometric-contraction-identity} and the formula
for \(u_s\), and cancelling the nonzero common factors gives
\[
 A_\sigma^{\fW'}(y_{\mathrm{br}})
 =(\mu_\sigma^*A_\sigma)(x).
\]
The same cited trivalent construction preserves the two pair minors at
\(k_-\) and \(k_+\).  Through \eqref{eq:local-framed-pair-minor}, this gives
the required string relabeling.
For \(j\ne i\), the local \(j\)-string has no endpoint in the block, so its
pair minor may be evaluated at the same cut on the two sides and is unchanged.

For \(\sigma:i^-i^-\to i^-\), apply the preceding calculation to the
horizontally reflected lower component \(\fW_-^{\mathrm{op}}\).  Under
transposition, the cut \(t\) corresponds to \(2-t\), the \(u\)-labels
\(u_1,u_2\) are interchanged, and the two boundary decorations are exchanged.
Hence the backward recursion in
\eqref{eq:cartan-exponent-recursion} becomes the forward recursion and
the two \(\omega_i\)-summands in \eqref{eq:cartan-exponent} are exchanged.
Together with \eqref{eq:extension-datum-cartan-monomial}, this gives the same
Cartan factors as in the plus calculation.  This also interchanges
\(k_-,k_+\), exactly as in the
seed transposition of
\cite[Proposition~3.9]{ShenWeng2021}.  Hence the same calculation proves the
minus case.

Finally, pair minors disjoint from the block are unchanged, while existing
contraction coordinates are preserved by
\eqref{eq:component-existing-coordinate-preservation}.  This proves
\eqref{eq:one-step-coordinate-compatibility} at every vertex.
\end{proof}

\begin{corollary}
\label{cor:reduced-iterated-coordinate-term}
For a double weave \(\fW:\beta\to\delta\),
\[
 \mathbf A^\fW
 =\mu_\fW\circ\mathbf A_\beta^{\mathrm{SW}}:
 \Conf(p,q)_{u,v}
 \dashrightarrow\Tcal_{Q^{\mathrm{am}}(\fW)}.
\]
Moreover, if \(\mathbf p',\mathbf q'\) are reduced words for \(u,v\),
respectively, then
\begin{equation}
 \mathbf A_{\delta,\fW}^{\mathrm{am}}\circ\Split
 =\mu_\fW\circ\mathbf A_\beta^{\mathrm{SW}}:
 \Conf(p,q)_{u,v}
 \dashrightarrow\Tcal_{Q^{\mathrm{am}}(\fW)}
 \label{eq:iterated-component-coordinate-term}
\end{equation}
holds.
\end{corollary}

\begin{proof}
The first identity follows by induction, with the identity double weave as the
base case and a one-step extension as the induction step.  At the initial end,
use \eqref{eq:component-identity-coordinate-map}.  In the induction
step, use Theorem~\ref{thm:one-step-coordinate-term} and
\eqref{eq:term-one-step-composition}.  If the terminal words are
reduced, Proposition~\ref{prop:component-terminal-propagation} and
Theorem~\ref{thm:flag-product} give
\(\operatorname{prop}_{\fW}=\operatorname{red}\) and
\(\Split=(\operatorname{prop}_{\fW},\operatorname{braid})\).  Hence the
left-hand side of \eqref{eq:component-partial-coordinate-map} is
\(\mathbf A_{\delta,\fW}^{\mathrm{am}}\circ\Split\).
\end{proof}

\begin{proposition}
\label{prop:one-sided-coordinate-compatibility}
Let \(\fW:\beta\to\delta\) be a double weave.  If
\(\fW_-=\id_{\mathbf q}\) and \(\mathbf p'\) is a reduced word for \(u\),
then
\begin{equation}
 \mathbf A_{\delta,\fW}^{\mathrm{am}}\circ\Split_+
 =\mu_\fW\circ\mathbf A_\beta^{\mathrm{SW}}:
 \Conf(p,q)_{u,*}
 \dashrightarrow\Tcal_{Q^{\mathrm{am}}(\fW)}.
 \label{eq:plus-one-sided-coordinate-compatibility}
\end{equation}
Similarly, if \(\fW_+=\id_{\mathbf p}\) and \(\mathbf q'\) is a reduced word
for \(v\), then
\begin{equation}
 \mathbf A_{\delta,\fW}^{\mathrm{am}}\circ\Split_-
 =\mu_\fW\circ\mathbf A_\beta^{\mathrm{SW}}:
 \Conf(p,q)_{*,v}
 \dashrightarrow\Tcal_{Q^{\mathrm{am}}(\fW)}
 \label{eq:minus-one-sided-coordinate-compatibility}
\end{equation}
holds.
\end{proposition}

\begin{proof}
We prove the assertion on the plus side, and set
\(v:=\operatorname{dem}(q)\).  Since \(v\) is the Demazure product of \(q\),
\(\Conf(p,q)_{u,v}\) is a dense open subvariety of
\(\Conf(p,q)_{u,*}\).  Since \(\fW_-=\id_{\mathbf q}\),
\(\operatorname{prop}_{\fW}\) does not change the minus flag sequence.
Moreover, the terminal plus flag sequence is reduced, so
Lemma~\ref{lem:reduced-flag-chain} gives, on this dense open subvariety,
\[
 \operatorname{prop}_{\fW}=\operatorname{red}_{u,q}.
\]
Furthermore, \(\mathbf A_{\delta,\fW}^{\mathrm{am}}\) factors through the map
that forgets the second braid-variety factor, so the \(\mathbf A^\fW\) of
Definition~\ref{def:component-intermediate-functions} equals
\(\mathbf A_{\delta,\fW}^{\mathrm{am}}\circ\Split_+\).  The first identity in
Corollary~\ref{cor:reduced-iterated-coordinate-term} therefore gives
\eqref{eq:plus-one-sided-coordinate-compatibility} on the dense open
subvariety.  Thus its two sides define the same rational map on
\(\Conf(p,q)_{u,*}\).  The assertion on the minus side follows from the
transpose isomorphism in Definition~\ref{def:double-bs-transposition} and the
seed isomorphism induced by Shen--Weng transposition
\cite[Proposition~3.9]{ShenWeng2021}.
\end{proof}

\subsection{Cluster localization}
\label{subsec:endpoint-locus-cluster-variables}

Recall that \(\mathbf p'\) and \(\mathbf q'\) are the plus and minus parts of the 
terminal word of \(\fW\).
In this subsection, 
when \(\mathbf p'\) is reduced, we write it as
\(\mathbf u=u_1\cdots u_r\), and when \(\mathbf q'\) is reduced, we write it as
\(\mathbf v=v_1\cdots v_s\).

\subsubsection{Nonvanishing conditions for endpoint relative positions}

The seed \(\Sigma_{\mathrm{sweep}}(\fW)\) is realized on \(\Conf(p,q)\) by
\[
 \mathbf A_{\mathrm{sweep}}^\fW
 :=\mu_\fW\circ\mathbf A_\beta^{\mathrm{SW}}.
\]
Write its coordinates as
\begin{equation}
 A_z^{\mathrm{sweep}}
 :=\left(\mathbf A_{\mathrm{sweep}}^\fW\right)_z
 \qquad
 \bigl(z\in\Vcal(\delta)\sqcup V_3(\fW)\bigr).
 \label{eq:term-regular-functions}
\end{equation}
The result \cite[Theorem~3.45]{ShenWeng2021} identifies \(\mathcal O(\Conf(p,q))\) with
the corresponding upper cluster algebra, so these are regular functions on
\(\Conf(p,q)\).  Below, for a regular function \(f\), write its principal open
subvariety as
\[
 D(f):=\{x\mid f(x)\ne0\}.
\]

For each \(x\in\Conf(p,q)\), choose the unique representative satisfying
\(\flagA^0=\unipotentU_+\) and \(\flagB_0=\borelB_-\), using
\eqref{eq:component-opposite-pair-isomorphism}.  Following the
Shen--Weng construction
\cite[Definition~2.9 and Lemma~2.10]{ShenWeng2021}, decorate the upper flag
sequence successively starting from \(\flagA^0\), and write its terminal end
as \(\flagA_{\beta,m}^+(x)\in\Aplus\).

For each \(i\in I\), define a map
\(f_i:\Conf(p,q)\to\mathbb A^1\) by
\begin{equation}
 f_i(x)
 :=\Delta_i\left(
  \unipotentU_-\dot u^{-1},
   \flagA_{\beta,m}^+(x)
  \right).
 \label{eq:plus-endpoint-minor}
\end{equation}

\begin{lemma}
\label{lem:plus-endpoint-minors}
As open subschemes,
\begin{equation}
 \Conf(p,q)_{u,*}
 =D\left(\prod_{i\in I}f_i\right)
 \label{eq:plus-endpoint-minor-locus}
\end{equation}
holds.
\end{lemma}

\begin{proof}
Let \(x\in\Conf(p,q)\).  Set
\(w:=\pos_+(\flagB^0,\flagB^m)\).  Since
\((\flagB^0,\ldots,\flagB^m)\) is a flag chain of type \(\mathbf p\), we have
\(w\le\operatorname{dem}(\mathbf p)=u\).  If we write
\(\flagA_{\beta,m}^+(x)=\eta\unipotentU_+\), then
\(\eta\in\borelB_+w\borelB_+\) and
\[
 f_i(x)=\Delta_{\omega_i}^{\groupG}(\dot u^{-1}\eta).
\]
Up to a nonzero scalar, the minor above is the \(u\omega_i\)-weight component
of \(\eta v_{\omega_i}\).  Since \(\eta\in\borelB_+w\borelB_+\), its weights
belong to \(w\omega_i+\sum_{j\in I}\mathbb N\alpha_j\), and its
\(w\omega_i\)-weight component is nonzero.  On the other hand, if \(w\le u\),
then \(w\omega_i-u\omega_i\in\sum_{j\in I}\mathbb N\alpha_j\).  Thus
\(f_i(x)\ne0\) if and only if \(w\omega_i=u\omega_i\), and this equality for
all fundamental weights is equivalent to \(w=u\).  This proves
\eqref{eq:plus-endpoint-minor-locus}.
\end{proof}

Write the positions occupied by the plus subsequence of \(\delta\) as
\[
 t_1<\cdots<t_r.
\]
For \(\sigma\in V_3(\fW_+)\), set
\begin{equation}
 \nu_\sigma^+
 :=H_{L,+}^{\delta}\bigl(0;\partial\gamma_\sigma\bigr)
 \in X_*(\torusH).
 \label{eq:plus-terminal-coweight}
\end{equation}

\begin{lemma}
\label{lem:plus-endpoint-minor-monomial}
Suppose that \(\mathbf p'\) is reduced.  For every \(i\in I\),
\begin{equation}
 f_i
 =\prod_{\sigma\in V_3(\fW_+)}
   \left(A_\sigma^{\mathrm{sweep}}\right)^{
    \langle\omega_i,\nu_\sigma^+\rangle}
 \label{eq:plus-endpoint-minor-monomial}
\end{equation}
holds in \(\mathcal O(\Conf(p,q))\).  All the exponents are nonnegative
integers, and
\begin{equation}
 \nu_\sigma^+\ne0
 \quad\Longleftrightarrow\quad
 \sigma\in\Dcal(\fW_+)
 \label{eq:plus-terminal-coweight-support}
\end{equation}
holds.
\end{lemma}

\begin{proof}
We prove the assertion on a common dense open subvariety on which all the
coordinates appearing below are defined.  Let \(x\in\Conf(p,q)_{u,v}\), and,
following Definition~\ref{def:double-braid-factors}, write
\[
 \flagB^m=gn u\borelB_+,
 \qquad n\in\unipotentU_+(u)
\]
and set \(x_{\mathrm{br}}:=\operatorname{braid}(x)\).  Apply the construction
of cluster variables in
\cite[Subsection~4.1, Theorem~5.1, and Subsection~6.1]
{CasalsEtAlBraidVarieties} to \(\fW_+\).  Since the terminal word \(\mathbf u\)
is reduced, its standard flag sequence and
\eqref{eq:weave-terminal-framing-monomial} give
\begin{equation}
 \flagA_{\beta,m}^+(x)
 =gn\dot s_{u_1}\alpha_{u_1}^\vee(u_{t_1}^\fW(x_{\mathrm{br}}))
     \cdots
     \dot s_{u_r}\alpha_{u_r}^\vee(u_{t_r}^\fW(x_{\mathrm{br}}))
     \unipotentU_+
 =gn\dot u\,h_L^{\fW,+}(x_{\mathrm{br}})\unipotentU_+.
 \label{eq:terminal-decoration-cartan-factor}
\end{equation}
The last identity follows from
\eqref{eq:extension-datum-cartan-monomial-components} and the plus recursion
in \eqref{eq:cartan-exponent-recursion} at the positions \(t_1,\ldots,t_r\).

Since \(n\in\unipotentU_+(u)\), we have
\(\dot u^{-1}n\dot u\in\unipotentU_-\).  Hence
Lemma~\ref{lem:minor-cartan-semi-invariance} and
\eqref{eq:terminal-decoration-cartan-factor} give
\begin{equation}
 f_i(x)=\omega_i\bigl(h_L^{\fW,+}(x_{\mathrm{br}})\bigr).
 \label{eq:plus-endpoint-minor-cartan-evaluation}
\end{equation}
Apply \eqref{eq:extension-datum-cartan-monomial-components} to the right-hand
side.
By Proposition~\ref{prop:double-weave-plus-minus-decomposition}, a cycle in the
minus component has no support at a plus position, and therefore
\(H_{L,+}^\delta(0;\partial\gamma_\sigma)=0\).  Using also Corollary
\ref{cor:reduced-iterated-coordinate-term} and
\(A_\sigma^{\mathrm{am}}=A_\sigma^\fW\), we obtain
\eqref{eq:plus-endpoint-minor-monomial} on the common dense open
subvariety.

By Proposition~\ref{prop:double-weave-plus-minus-decomposition},
\(\partial\gamma_\sigma\) is supported only at plus positions.  The plus
recursion in \eqref{eq:cartan-exponent-recursion} gives
\[
 \nu_\sigma^+
 =\sum_{a=1}^r(\partial\gamma_\sigma)(t_a)
   s_{u_r}\cdots s_{u_{a+1}}(\alpha_{u_a}^\vee).
\]
Since \(\mathbf u\) is reduced, the coroots
\[
 s_{u_r}\cdots s_{u_{a+1}}(\alpha_{u_a}^\vee)
 \qquad(1\le a\le r)
\]
are positive coroots.  Moreover,
\((\partial\gamma_\sigma)(t_a)\in\mathbb N\), so the exponents in
\eqref{eq:plus-endpoint-minor-monomial} are nonnegative integers.
The displayed sum is zero if and only if all its coefficients are zero;
hence \(\nu_\sigma^+=0\) if and only if
\(\partial\gamma_\sigma=0\).  By
\eqref{eq:weave-frozen-set}, this gives
\eqref{eq:plus-terminal-coweight-support}.  Since all the exponents
are nonnegative, both sides of
\eqref{eq:plus-endpoint-minor-monomial} are regular on
\(\Conf(p,q)\).  Furthermore, since \(u,v\) are their respective Demazure
products, the identity obtained above on a dense open subvariety extends to
all of \(\Conf(p,q)\).
\end{proof}

\begin{theorem}
\label{thm:endpoint-locus-cluster-variables}
The following statements hold.
\begin{enumerate}
\item If \(\mathbf p'\) is reduced, then, as open subschemes,
\begin{equation}
 \Conf(p,q)_{u,*}
 =\bigcap_{\sigma\in\Dcal(\fW_+)}
   D\left(A_\sigma^{\mathrm{sweep}}\right).
 \label{eq:plus-endpoint-term-locus}
\end{equation}
\item If \(\mathbf q'\) is reduced, then, as open subschemes,
\begin{equation}
 \Conf(p,q)_{*,v}
 =\bigcap_{\sigma\in\Dcal(\fW_-^{\mathrm{op}})}
   D\left(A_\sigma^{\mathrm{sweep}}\right).
 \label{eq:minus-endpoint-term-locus}
\end{equation}
\item If both \(\mathbf p'\) and \(\mathbf q'\) are reduced, then, as open
subschemes,
\begin{equation}
 \Conf(p,q)_{u,v}
 =\bigcap_{\sigma\in\Dcal(\fW)}
   D\left(A_\sigma^{\mathrm{sweep}}\right).
 \label{eq:double-endpoint-term-locus}
\end{equation}
\end{enumerate}
Here we use the identification of vertices in
\eqref{eq:double-weave-quiver-decomposition}.
\end{theorem}

\begin{proof}
Assume first that \(\mathbf p'\) is reduced.
Set \(\rho:=\sum_{i\in I}\omega_i\).  Multiplying
\eqref{eq:plus-endpoint-minor-monomial} over \(i\in I\) gives
\[
 \prod_{i\in I}f_i
 =\prod_{\sigma\in V_3(\fW_+)}
   \left(A_\sigma^{\mathrm{sweep}}\right)^{
    \langle\rho,\nu_\sigma^+\rangle}.
\]
For every positive coroot \(\alpha^\vee\), we have
\(\langle\rho,\alpha^\vee\rangle>0\), so
\eqref{eq:plus-terminal-coweight-support} gives
\[
 \langle\rho,\nu_\sigma^+\rangle>0
 \quad\Longleftrightarrow\quad
 \sigma\in\Dcal(\fW_+).
\]
Lemma~\ref{lem:plus-endpoint-minors} and the properties of principal open
subvarieties now give \eqref{eq:plus-endpoint-term-locus}.

Assume next that \(\mathbf q'\) is reduced.  After applying the transpose in
Definition~\ref{def:double-bs-transposition},
the endpoint condition on the minus side becomes the endpoint condition for
\(v^{-1}\) on the plus side of \(\mathbf q^{\mathrm{op}}\).  The word
\((\mathbf q')^{\mathrm{op}}\) is a reduced word for \(v^{-1}\), and under the
seed isomorphism induced by Shen--Weng transposition
\cite[Proposition~3.9]{ShenWeng2021}, the corresponding weave is
\(\fW_-^{\mathrm{op}}\).  Thus, transposing the result on the plus side gives
\eqref{eq:minus-endpoint-term-locus}.  Finally, if both terminal words are
reduced,
Definition~\ref{def:endpoint-strata} and
Proposition~\ref{prop:double-weave-plus-minus-decomposition} give
\eqref{eq:double-endpoint-term-locus}.
\end{proof}

If both \(\mathbf p'\) and \(\mathbf q'\) are reduced, then
Theorem~\ref{thm:endpoint-locus-cluster-variables} gives
\begin{equation}
 \Conf(p,q)_{u,v}
 =D\left(\prod_{\sigma\in\Dcal(\fW)}A_\sigma^{\mathrm{sweep}}\right).
 \label{eq:endpoint-term-product}
\end{equation}
Under the same hypothesis, \(\Conf(p,q)\) is affine by the Shen--Weng theorem
\cite[Theorem~2.30]{ShenWeng2021}, and hence
\begin{equation}
 \mathcal O\bigl(\Conf(p,q)_{u,v}\bigr)
 =\mathcal O\bigl(\Conf(p,q)\bigr)
  \left[(A_\sigma^{\mathrm{sweep}})^{-1}
        \mid\sigma\in\Dcal(\fW)\right].
 \label{eq:endpoint-term-localization}
\end{equation}

\subsubsection{Frozen seeds and product seeds}

Set the Shen--Weng seed and the product seed to be
\[
 \Sigma_\delta:=\Sigma_{Q(\delta)},
 \qquad
 \Sigma_\times(\fW):=\Sigma_\delta\sqcup\Sigma_\fW.
\]
Also set
\begin{equation}
 \Sigma_{\mathrm{am}}^{\mathrm{fr}}(\fW)
 :=\operatorname{freeze}_{\Dcal(\fW)}
    \Sigma_{\mathrm{am}}(\fW).
 \label{eq:def-frozen-amalgamation-seed}
\end{equation}
If both \(\mathbf p'\) and \(\mathbf q'\) are reduced, then
\(\Sigma_\delta\) and \(\Sigma_\times(\fW)\) are realized on \(\Conf(u,v)\)
and \(\Conf(u,v)\times X(p)\times X(q^{\mathrm{op}})\) by
\(\mathbf A_\delta^{\mathrm{SW}}\) and
\(\mathbf A_\delta^{\mathrm{SW}}\times\mathbf A_\fW\), respectively.  In this
case, regard \(\Sigma_{\mathrm{sweep}}^{\mathrm{fr}}(\fW)\) as restricted to
\(\Conf(p,q)_{u,v}\).

If both \(\mathbf p'\) and \(\mathbf q'\) are reduced, then the identification
of vertices in Theorem~\ref{thm:quiver} and
\eqref{eq:iterated-component-coordinate-term} give
\begin{equation}
 A_z^{\mathrm{sweep}}=A_z^{\mathrm{am}}\circ\Split
 \qquad\bigl(z\in\Vcal(\delta)\sqcup V_3(\fW)\bigr).
 \label{eq:full-frozen-exact-seed}
\end{equation}
Thus, in this case, \(\Split\) gives an isomorphism between these two
realizations.

For a seed \(\Sigma\) and a mutable vertex \(j\), write its \(X\)-coordinate
as
\[
 X_j^\Sigma
 :=\prod_{i\in I(\Sigma)}A_i^{\eps^\Sigma_{ji}}.
\]

\begin{theorem}
\label{thm:frozen-product-comparison}
Consider the torus isomorphism defined by the following coordinate
transformation:
\[
 \id_{\times\to\mathrm{am}}:
 \Tcal_{\Sigma_\times(\fW)}
 \xrightarrow{\sim}
 \Tcal_{\Sigma_{\mathrm{am}}^{\mathrm{fr}}(\fW)}.
\]
The sets of mutable vertices of the two seeds are
\begin{equation}
 I_{\mathrm{uf}}\bigl(\Sigma_\times(\fW)\bigr)
 =I_{\mathrm{uf}}\bigl(\Sigma_{\mathrm{am}}^{\mathrm{fr}}(\fW)\bigr)
 =I_{\mathrm{int}}(\delta)
  \sqcup\bigl(V_3(\fW)\setminus\Dcal(\fW)\bigr).
 \label{eq:full-product-mutable-set}
\end{equation}
The pullback of this isomorphism is
\begin{equation}
 \begin{aligned}
 (\id_{\times\to\mathrm{am}})^*A_\ell^{\mathrm{am}}
 &=A_\ell^\delta
   \prod_{\tau\in\Dcal(\fW)}
   (A_\tau^\fW)^{c_\ell(\partial\gamma_\tau)}
 &&(\ell\in\Vcal(\delta)),\\
 (\id_{\times\to\mathrm{am}})^*A_\sigma^{\mathrm{am}}
 &=A_\sigma^\fW
 &&(\sigma\in V_3(\fW)),
 \end{aligned}
 \label{eq:full-coefficient-rescaling}
\end{equation}
and, for every mutable vertex \(j\),
\begin{equation}
 (\id_{\times\to\mathrm{am}})^*X_j^{\mathrm{am}}
 =X_j^\times
 \label{eq:full-x-preservation}
\end{equation}
holds.  In particular, \(\id_{\times\to\mathrm{am}}\) is a quasi-cluster
isomorphism.
\end{theorem}

\begin{proof}
By \eqref{eq:terminal-coordinate}
and~\eqref{eq:weave-terminal-framing-monomial},
\[
 A_\ell^{\mathrm{am}}
 =A_\ell^\delta
  \prod_{\tau\in\Dcal(\fW)}
  (A_\tau^\fW)^{c_\ell(\partial\gamma_\tau)}.
\]
This gives the first identity in
\eqref{eq:full-coefficient-rescaling}.  The second identity follows
from Definition~\ref{def:amalgamated-functions}.  The inverse map is given by
the same identities with the exponents replaced by
\(-c_\ell(\partial\gamma_\tau)\).

The set of frozen vertices of each seed is
\[
 \bigl(\Vcal(\delta)\setminus I_{\mathrm{int}}(\delta)\bigr)
 \sqcup\Dcal(\fW).
\]
If the frozen variables are ordered in this way, the exponent matrix of the
Laurent monomials in \((\id_{\times\to\mathrm{am}})^*\) is block triangular
with identity diagonal blocks.  Hence the pullback gives an isomorphism
between the groups of frozen Laurent monomials of the two seeds.

For \(j\in I_{\mathrm{int}}(\delta)\) and \(\sigma\in\Dcal(\fW)\), set
\[
 \eta_{j\sigma}
 :=\eps^{\mathrm{am}}_{j,\sigma}
  +\sum_{\ell\in\Vcal(\delta)}
        \eps^{Q(\delta)}_{j,\ell}
        c_\ell(\partial\gamma_\sigma).
\]
By Definition~\ref{def:extended} and
Lemma~\ref{lem:cartan-exponent-exchange-identity},
\[
 \eta_{j\sigma}
 =\bigl(\operatorname{inc}_\delta(\partial\gamma_\sigma)\bigr)_j
  -\bigl(\operatorname{inc}_\delta(\partial\gamma_\sigma)\bigr)_j
 =0.
\]
Therefore,
\[
 \begin{aligned}
 (\id_{\times\to\mathrm{am}})^*X_j^{\mathrm{am}}
 &=\prod_{\ell\in\Vcal(\delta)}
   (A_\ell^\delta)^{\eps^{Q(\delta)}_{j,\ell}}
   \prod_{\sigma\in\Dcal(\fW)}
   (A_\sigma^\fW)^{\eta_{j\sigma}}\\
 &=X_j^\times.
 \end{aligned}
\]
On the other hand, for
\(\sigma\in V_3(\fW)\setminus\Dcal(\fW)\), we have
\(\partial\gamma_\sigma=0\), so the exchange-matrix entries between
\(\sigma\) and \(\Vcal(\delta)\) vanish.  Moreover, by
\eqref{eq:double-am-contraction-contraction}, the \(\sigma\)-row on
\(V_3(\fW)\) equals the \(\sigma\)-row of \(Q(\fW)\).  Hence
\((\id_{\times\to\mathrm{am}})^*X_\sigma^{\mathrm{am}}=X_\sigma^\times\).
It follows that \(\id_{\times\to\mathrm{am}}\) is a quasi-cluster isomorphism
\cite[Definition~3.3 and the paragraph following Remark~3.4]
{GorskyKimScrogginSimental2025}.
\end{proof}

\begin{corollary}[Quasi-cluster property of the flag-configuration decomposition]
\label{cor:split-quasi-cluster}
If both \(\mathbf p'\) and \(\mathbf q'\) are reduced, then the
flag-configuration decomposition
\[
 \Split:
 \Conf(p,q)_{u,v}
 \xrightarrow{\sim}
 \Conf(u,v)\times X(p)\times X(q^{\mathrm{op}})
\]
is a quasi-cluster isomorphism with respect to
\(\Sigma_{\mathrm{sweep}}^{\mathrm{fr}}(\fW)\) and
\(\Sigma_\times(\fW)\).
\end{corollary}

\begin{proof}
This follows from Theorem~\ref{thm:flag-product},
\eqref{eq:full-frozen-exact-seed}, and
Theorem~\ref{thm:frozen-product-comparison}.
\end{proof}

\subsubsection{Cluster localization}

Below, let \(\mathcal A(\Sigma)\) and \(\mathcal U(\Sigma)\) denote the cluster
algebra and the upper cluster algebra associated with a seed \(\Sigma\).  We adopt
the convention that they contain the inverses of the frozen variables.  For
\(S\subset I_{\mathrm{uf}}(\Sigma)\), the identity map on the initial cluster
defines the canonical inclusion
\[
 \mathcal A\bigl(\operatorname{freeze}_S(\Sigma)\bigr)
 \lhook\joinrel\longrightarrow
 \mathcal A(\Sigma)[A_s^{-1}\mid s\in S].
\]
Following \cite[Definition~3.3]{Muller2013}, we call this freezing a
\emph{cluster localization} when this inclusion is an isomorphism.

\begin{theorem}[Cluster localization of an endpoint stratum]
\label{thm:endpoint-cluster-localization}
If both \(\mathbf p'\) and \(\mathbf q'\) are reduced, then the freezing from
\(\Sigma_{\mathrm{sweep}}(\fW)\) to
\(\Sigma_{\mathrm{sweep}}^{\mathrm{fr}}(\fW)\) is a cluster localization,
and
\begin{equation}
 \begin{aligned}
 \mathcal A\bigl(\Sigma_{\mathrm{sweep}}^{\mathrm{fr}}(\fW)\bigr)
 &=\mathcal U\bigl(\Sigma_{\mathrm{sweep}}^{\mathrm{fr}}(\fW)\bigr)\\
 &=\mathcal A\bigl(\Sigma_{\mathrm{sweep}}(\fW)\bigr)
   \left[(A_\sigma^{\mathrm{sweep}})^{-1}
         \mid\sigma\in\Dcal(\fW)\right]\\
 &=\mathcal O\bigl(\Conf(p,q)_{u,v}\bigr)
 \end{aligned}
 \label{eq:endpoint-cluster-localization}
\end{equation}
holds.
\end{theorem}

\begin{proof}
By the cluster structures on double Bott--Samelson cells
\cite[Theorems~3.45 and~4.13]{ShenWeng2021} and on braid varieties
\cite[Theorem~1.1]{CasalsEtAlBraidVarieties},
\[
 \begin{aligned}
 \mathcal O\bigl(\Conf(u,v)\bigr)
 &=\mathcal U(\Sigma_\delta)
  =\mathcal A(\Sigma_\delta),\\
 \mathcal O\bigl(X(p)\bigr)
 &=\mathcal U(\Sigma_{\fW_+})
  =\mathcal A(\Sigma_{\fW_+}),\\
 \mathcal O\bigl(X(q^{\mathrm{op}})\bigr)
 &=\mathcal U\bigl(\Sigma_{\fW_-^{\mathrm{op}}}\bigr)
  =\mathcal A\bigl(\Sigma_{\fW_-^{\mathrm{op}}}\bigr).
 \end{aligned}
\]
Taking tensor products gives
\[
 \mathcal O\bigl(
  \Conf(u,v)\times X(p)\times X(q^{\mathrm{op}})
 \bigr)
 =\mathcal U\bigl(\Sigma_\times(\fW)\bigr)
 =\mathcal A\bigl(\Sigma_\times(\fW)\bigr).
\]
 Transporting this identity by \(\Split\), using
 Corollary~\ref{cor:split-quasi-cluster}, gives
\[
 \mathcal A\bigl(\Sigma_{\mathrm{sweep}}^{\mathrm{fr}}(\fW)\bigr)
 =\mathcal U\bigl(\Sigma_{\mathrm{sweep}}^{\mathrm{fr}}(\fW)\bigr)
 =\mathcal O\bigl(\Conf(p,q)_{u,v}\bigr).
\]
On the other hand, \(\Sigma_{\mathrm{sweep}}(\fW)\) is mutation equivalent
to the Shen--Weng seed, so the same two Shen--Weng theorems give
\[
 \mathcal A\bigl(\Sigma_{\mathrm{sweep}}(\fW)\bigr)
 =\mathcal O\bigl(\Conf(p,q)\bigr).
\]
Combining this with \eqref{eq:endpoint-term-localization} proves
\eqref{eq:endpoint-cluster-localization}.
\end{proof}

\begin{corollary}
\label{cor:plus-endpoint-cluster-localization}
Suppose that a double weave \(\fW:\beta\longrightarrow\delta\) satisfies
\[
 \fW_+:\mathbf p\longrightarrow\mathbf u,
 \qquad
 \fW_-=\id_{\mathbf q}.
\]
Then \(\Split_+\) is a quasi-cluster isomorphism with respect to the freezing
of \(\Sigma_{\mathrm{sweep}}(\fW)\) at \(\Dcal(\fW)\) and
\(\Sigma_\times(\fW)\).  Furthermore, this freezing at \(\Dcal(\fW)\) is a
cluster localization, and
\[
 \mathcal A\bigl(
  \operatorname{freeze}_{\Dcal(\fW)}\Sigma_{\mathrm{sweep}}(\fW)
 \bigr)
 =\mathcal U\bigl(
  \operatorname{freeze}_{\Dcal(\fW)}\Sigma_{\mathrm{sweep}}(\fW)
 \bigr)
 =\mathcal O\bigl(\Conf(p,q)_{u,*}\bigr)
\]
holds.
\end{corollary}

\begin{proof}
Proposition~\ref{prop:one-sided-coordinate-compatibility} and
Theorem~\ref{thm:frozen-product-comparison} show that \(\Split_+\) is a
quasi-cluster isomorphism with respect to
\(\operatorname{freeze}_{\Dcal(\fW)}\Sigma_{\mathrm{sweep}}(\fW)\) and
\(\Sigma_\times(\fW)\). Applying the argument of
Theorem~\ref{thm:endpoint-cluster-localization} to
\(\Conf(u,q)\times X(p)\), with
\eqref{eq:plus-endpoint-term-locus} in place of
\eqref{eq:double-endpoint-term-locus}, proves the remaining assertions.
\end{proof}

\begin{corollary}
\label{cor:minus-endpoint-cluster-localization}
Suppose that a double weave \(\fW:\beta\longrightarrow\delta\) satisfies
\[
 \fW_+=\id_{\mathbf p},
 \qquad
 \fW_-:\mathbf q\longrightarrow\mathbf v.
\]
Then \(\Split_-\) is a quasi-cluster isomorphism with respect to the freezing
of \(\Sigma_{\mathrm{sweep}}(\fW)\) at \(\Dcal(\fW)\) and
\(\Sigma_\times(\fW)\).  Furthermore, this freezing at \(\Dcal(\fW)\) is a
cluster localization, and
\[
 \mathcal A\bigl(
  \operatorname{freeze}_{\Dcal(\fW)}\Sigma_{\mathrm{sweep}}(\fW)
 \bigr)
 =\mathcal U\bigl(
  \operatorname{freeze}_{\Dcal(\fW)}\Sigma_{\mathrm{sweep}}(\fW)
 \bigr)
 =\mathcal O\bigl(\Conf(p,q)_{*,v}\bigr)
\]
holds.
\end{corollary}

\begin{proof}
Using the seed isomorphism induced by Shen--Weng transposition
\cite[Proposition~3.9]{ShenWeng2021} and the transpose isomorphism in
Definition~\ref{def:double-bs-transposition}, it suffices to transpose
Corollary~\ref{cor:plus-endpoint-cluster-localization}.
\end{proof}

\section{Application to the splicing map conjecture}
\label{sec:splicing-specialization}

In this section, we specialize
Corollaries~\ref{cor:plus-endpoint-cluster-localization}
and~\ref{cor:minus-endpoint-cluster-localization} to half-decorated flag
configurations and deduce that the two splicing maps are quasi-cluster
isomorphisms.

Below, we use the inclusion
\(\Word(I)\hookrightarrow\Word(\widetilde I)\), \(i\mapsto i^+\), implicitly
and omit the plus tag.  In particular, \(Q(\mathbf p)\) and
\(\Vcal(\mathbf p)\) follow this convention.

\subsection{Half-decorated flag decompositions and splicing maps}

Below, let \(\Delta\) be a reduced positive braid for \(w_0\), and let
\(\boldsymbol\Delta\) be a reduced word for it.  Set \(N:=\len(\Delta)\).  When
\(\operatorname{dem}(p)=w_0\), applying
Corollary~\ref{cor:half-one-sided-flag-product} with
\(\mathbf q=\varnothing\) gives
\begin{align*}
 \Split_{\mathrm{half}}^+:
 \Confhalf^+(p,\varnothing)_{w_0,*}
 &\xrightarrow{\sim}
 \Confhalf^+(\Delta,\varnothing)\times X(p),\\
 \Split_{\mathrm{half}}^-:
 \Confhalf^-(\varnothing,p^{\mathrm{op}})_{*,w_0}
 &\xrightarrow{\sim}
 \Confhalf^-(\varnothing,\Delta^{\mathrm{op}})\times X(p).
\end{align*}
The second isomorphism is the transpose of the first under
\eqref{eq:half-conf-transposition}.

\begin{proposition}
\label{prop:delta-extension}
For any positive braid \(p\in\Br^+\), there are isomorphisms
\begin{equation}
 \iota_p^+:
 \Confhalf^+(p,\varnothing)\xrightarrow{\sim}X(p\Delta),
 \qquad
 \iota_p^-:
 \Confhalf^-(\varnothing,p^{\mathrm{op}})\xrightarrow{\sim}X(\Delta p).
 \label{eq:two-delta-extensions}
\end{equation}
\end{proposition}

\begin{proof}
\cite[Lemma~4.4]{GorskyKimScrogginSimental2025}.
\end{proof}

Below, assume \(\operatorname{dem}(p)=w_0\) and \(r=\len(p)\).  Using the
isomorphisms of Proposition~\ref{prop:delta-extension}, set
\begin{equation}
\begin{aligned}
 \Ucal_{r,w_0}(p\Delta)
 &:=\iota_p^+
    \bigl(\Confhalf^+(p,\varnothing)_{w_0,*}\bigr),\\
 \Ucal_{N,e}(\Delta p)
 &:=\iota_p^-
    \bigl(\Confhalf^-(\varnothing,p^{\mathrm{op}})_{*,w_0}\bigr).
\end{aligned}
\label{eq:two-splicing-opens}
\end{equation}

\begin{definition}[splicing maps]
Define isomorphisms
\begin{align*}
 \Psi_{r,w_0}:
 X(p)\times X(\Delta^2)
 &\xrightarrow{\sim}\Ucal_{r,w_0}(p\Delta),\\
 \Psi_{N,e}:
 X(\Delta^2)\times X(p)
 &\xrightarrow{\sim}\Ucal_{N,e}(\Delta p),
\end{align*}
by
\begin{align}
 \Psi_{r,w_0}
 &:=\iota_p^+\circ(\Split_{\mathrm{half}}^+)^{-1}\circ
  ((\iota_{\Delta}^+)^{-1}\times\id_{X(p)})\circ
  \operatorname{swap},
 \label{eq:psi-r-w0}\\
 \Psi_{N,e}
 &:=\iota_p^-\circ(\Split_{\mathrm{half}}^-)^{-1}\circ
  ((\iota_{\Delta}^-)^{-1}\times\id_{X(p)}).
 \label{eq:psi-N-e}
\end{align}
\end{definition}

\subsection{Half-decoration}

For a positive braid \(p\) and a point
\(c=[\flagA^0;\flagB^\bullet;\flagB_0]
\in\Confhalf^+(p,\varnothing)\), there is a unique \(g\in\groupG\) satisfying
\((\flagA^0,\flagB_0)=(g\unipotentU_+,\borelB_-g^{-1})\).  By
\eqref{eq:component-opposite-pair-isomorphism},
\[
 s_p^+:\Confhalf^+(p,\varnothing)\longrightarrow\Conf(p,\varnothing),
 \qquad
 c\longmapsto
 [\flagA^0;\flagB^\bullet;\flagB_0;\unipotentU_-g^{-1}]
\]
is a section of \(\operatorname{half}_+\).  For a word \(\mathbf p\) that
represents \(p\), the left-boundary coordinates are
\[
 (s_p^+)^*A_\ell^{\mathbf p}=1
 \qquad
 \bigl(\ell\in I_{\partial_L}(\mathbf p)\bigr).
\]
Every fiber of \(\operatorname{half}_+\) is a \(\torusH\)-torsor, and by
Lemma~\ref{lem:minor-cartan-semi-invariance}, the left-boundary coordinates
give its \(\torusH\)-coordinates.  Hence the remaining pullbacks form torus
coordinates.  If we set
\[
 Q_{\mathrm{half}}(\mathbf p)
 :=Q(\mathbf p)\setminus I_{\partial_L}(\mathbf p),
 \qquad
 \Sigma_{\mathrm{half}}(\mathbf p)
 :=\Sigma_{Q_{\mathrm{half}}(\mathbf p)},
\]
then
\[
 \mathbf A_{\mathrm{half}}^{\mathbf p}
 :=
 \bigl((s_p^+)^*A_\ell^{\mathbf p}\bigr)_{
  \ell\in\Vcal(\mathbf p)\setminus I_{\partial_L}(\mathbf p)}
 :
 \Confhalf^+(p,\varnothing)
 \dashrightarrow\Tcal_{\Sigma_{\mathrm{half}}(\mathbf p)}
\]
is a realization of \(\Sigma_{\mathrm{half}}(\mathbf p)\).

The sweep mutation sequence along a Demazure weave
\(\fW:\mathbf p\to\boldsymbol\Delta\) does not mutate the left-boundary
vertices, so it induces a sweep mutation sequence
\(\mu_\fW^{\mathrm{half}}\) with these vertices deleted.  Set
\[
 \begin{aligned}
 \Sigma_{\mathrm{sweep}}^{\mathrm{half}}(\fW)
 &:=
 \operatorname{freeze}_{\Dcal(\fW)}
 \bigl(\mu_\fW^{\mathrm{half}}\Sigma_{\mathrm{half}}(\mathbf p)\bigr),\\
 \Sigma_{\times}^{\mathrm{half}}(\fW)
 &:=
 \Sigma_{\mathrm{half}}(\boldsymbol\Delta)\sqcup\Sigma_\fW.
 \end{aligned}
\]
The superscript \({\mathsf t}\) denotes the pushforward by the transpose in
\eqref{eq:half-conf-transposition} for
\(\Sigma_{\mathrm{sweep}}^{\mathrm{half}}\), and by the product of the
transpose and \(\id_{X(p)}\) for \(\Sigma_\times^{\mathrm{half}}\).

\begin{proposition}
\label{prop:half-split-quasi-cluster}
For every Demazure weave \(\fW:\mathbf p\to\boldsymbol\Delta\), the maps
\(\Split_{\mathrm{half}}^+\) and \(\Split_{\mathrm{half}}^-\) are,
respectively, quasi-cluster isomorphisms with respect to
\[
 \begin{aligned}
 \left(
   \Sigma_{\mathrm{sweep}}^{\mathrm{half}}(\fW),
   \Sigma_{\times}^{\mathrm{half}}(\fW)
  \right),\qquad
 \left(
   \bigl(\Sigma_{\mathrm{sweep}}^{\mathrm{half}}(\fW)\bigr)^{\mathsf t},
   \bigl(\Sigma_{\times}^{\mathrm{half}}(\fW)\bigr)^{\mathsf t}
  \right)
 \end{aligned}.
\]
Moreover,
\[
 \begin{aligned}
 \mathcal A\bigl(\Sigma_{\mathrm{sweep}}^{\mathrm{half}}(\fW)\bigr)
 &=\mathcal U\bigl(\Sigma_{\mathrm{sweep}}^{\mathrm{half}}(\fW)\bigr)
  =\mathcal O\bigl(\Confhalf^+(p,\varnothing)_{w_0,*}\bigr),\\
 \mathcal A\bigl((\Sigma_{\mathrm{sweep}}^{\mathrm{half}}(\fW))^{\mathsf t}\bigr)
 &=\mathcal U\bigl((\Sigma_{\mathrm{sweep}}^{\mathrm{half}}(\fW))^{\mathsf t}\bigr)
  =\mathcal O\bigl(\Confhalf^-(\varnothing,p^{\mathrm{op}})_{*,w_0}\bigr).
 \end{aligned}
\]
\end{proposition}

\begin{proof}
Apply Corollary~\ref{cor:plus-endpoint-cluster-localization} with
\(\mathbf q=\varnothing\).  The flag-configuration decomposition commutes
with the normalization sections, and
\[
 \Split_+\circ s_p^+
 =
 (s_\Delta^+\times\id_{X(p)})\circ\Split_{\mathrm{half}}^+.
\]
The left-boundary vertices are frozen and are not mutated during the sweep
mutation sequence.  Furthermore, by
Lemma~\ref{lem:cartan-exponent-choice-independence}, the exponents in
\eqref{eq:full-coefficient-rescaling} satisfy
\[
 c_\ell(\partial\gamma_\sigma)=0
 \qquad
 \bigl(
  \sigma\in\Dcal(\fW),\
  \ell\in I_{\partial_L}(\boldsymbol\Delta)
 \bigr).
\]
Thus the quasi-cluster isomorphism in that corollary preserves the
specialization that sets the left-boundary coordinates equal to \(1\), and
induces \(\Split_{\mathrm{half}}^+\) on the remaining seeds.  The second
assertion follows from
Corollary~\ref{cor:minus-endpoint-cluster-localization}, or equivalently from
the seed isomorphism induced by Shen--Weng transposition
\cite[Proposition~3.9]{ShenWeng2021}.

Both factors of \(\Sigma_\times^{\mathrm{half}}(\fW)\) satisfy
\(\mathcal A=\mathcal U=\mathcal O\) by
\cite[Section~5.3 and Theorem~1.1]{CasalsEtAlBraidVarieties}.
The same tensor-product argument as in
Theorem~\ref{thm:endpoint-cluster-localization} proves the displayed equalities.
\end{proof}

\begin{remark}
The equalities \(\mathcal A=\mathcal U=\mathcal O\) in
Proposition~\ref{prop:half-split-quasi-cluster} also follow from the argument of
\cite[Proposition~5.9 and the proof of Lemma~5.7]
{GorskyKimScrogginSimental2025}.
\end{remark}

\subsection{Quasi-cluster property of the splicing maps}

Following \cite[Definition~4.5 and Section~6.4]{CasalsEtAlBraidVarieties},
choose right and left inductive weaves with terminal word
\(\boldsymbol\Delta\),
\[
 \rind{\mathbf p}:\mathbf p\longrightarrow\boldsymbol\Delta,
 \qquad
 \lind{\mathbf p}:\mathbf p\longrightarrow\boldsymbol\Delta.
\]
Set the following seeds:
\begin{align*}
 \widehat{\Sigma}_{r,w_0}(\rind{\mathbf p})
 &:=
 (\iota_p^+)_*
 \Sigma_{\mathrm{sweep}}^{\mathrm{half}}(\rind{\mathbf p}),\\
 \Sigma_{\times,w_0}(\rind{\mathbf p})
 &:=
 \bigl(\operatorname{swap}\circ
       (\iota_{\Delta}^+\times\id_{X(p)})\bigr)_*
 \Sigma_{\times}^{\mathrm{half}}(\rind{\mathbf p}),\\
 \widehat{\Sigma}_{N,e}(\lind{\mathbf p})
 &:=
 (\iota_p^-)_*
 \bigl(\Sigma_{\mathrm{sweep}}^{\mathrm{half}}
       (\lind{\mathbf p})\bigr)^{\mathsf t},\\
 \Sigma_{\times,e}(\lind{\mathbf p})
 &:=
 (\iota_{\Delta}^-\times\id_{X(p)})_*
 \bigl(\Sigma_{\times}^{\mathrm{half}}
       (\lind{\mathbf p})\bigr)^{\mathsf t}.
\end{align*}
Here \(\rind{\mathbf p}\) and \(\lind{\mathbf p}\) are, respectively, the
left-to-right and right-to-left inductive weaves of
\cite[Proposition~5.9]{GorskyKimScrogginSimental2025}.  By
\eqref{eq:weave-frozen-set} and that proposition, the two
\(\widehat{\Sigma}\) are the seeds obtained by freezing the vertices specified
there, and the two \(\Sigma_\times\) are the product
seeds of the corresponding braid varieties.

\begin{proof}[Proof of Corollary~\ref{cor:splicing-conjecture-four-endpoints}]
Applying \eqref{eq:psi-r-w0} and
Proposition~\ref{prop:half-split-quasi-cluster} with
\(\fW=\rind{\mathbf p}\), and transporting them by
\(\operatorname{swap}\circ(\iota_\Delta^+\times\id_{X(p)})\) and
\(\iota_p^+\), gives the assertion for \(\Psi_{r,w_0}\).  Similarly,
applying \eqref{eq:psi-N-e} and the transposed assertion of the same
proposition with \(\fW=\lind{\mathbf p}\) gives the assertion for
\(\Psi_{N,e}\).
\end{proof}

\section{Examples}

In this section, we let \(\groupG=\mathrm{SL}_3\) and compute the cluster
coordinates associated with the Demazure weave \(\fW\) used in
Example~\ref{subsec:a2-mutation-term-example}, by applying the sweep mutation
sequence along \(\fW\).
The string expression and
quiver for the initial word \(\beta=1122112\) are shown in
Figure~\ref{fig:SW-1122112-string-quiver}, and the five mutations along \(\fW\)
are shown in Figures~\ref{fig:a2-five-step-mutation-early}
and~\ref{fig:a2-five-step-mutation-late}.

Let \(z=(z_1,\ldots,z_7)\) be the root-subgroup variables corresponding to
the letters of \(\beta\), and set
\[
 g_a(z):=B_{\beta_1}(z_1)\cdots B_{\beta_a}(z_a)
 \qquad(0\le a\le7).
\]
Let \(h\in\torusH\), and write \(t_k:=\omega_k(h)\ (k=1,2)\).  The
braid-variety factor of the point
\[
 c
 :=[\unipotentU_+;
     g_0(z)\borelB_+,\ldots,g_7(z)\borelB_+;
     \borelB_-;\unipotentU_-h]
 \in\Conf(\beta,\varnothing)_{w_0,e}
\]
is
\[
 \operatorname{braid}_+(c)
 =[\unipotentU_+,g_7(z)\borelB_+;
   g_0(z)\borelB_+,\ldots,g_7(z)\borelB_+]
 \in X(\beta).
\]
The string coordinates at \(c\) are
\[
 A_{k,[r,s]}^\beta=t_k\Delta_{\omega_k}(g_r(z)).
\]
Below, we omit evaluation at \(c\) from the notation and write the initial
coordinates as
\begin{align*}
 (x_0,\ldots,x_4)
 &:=(A_{1,[0,1]}^\beta,A_{1,[1,2]}^\beta,
     A_{1,[2,5]}^\beta,A_{1,[5,6]}^\beta,
     A_{1,[6,\infty]}^\beta),\\
 (y_0,\ldots,y_3)
 &:=(A_{2,[0,3]}^\beta,A_{2,[3,4]}^\beta,
     A_{2,[4,7]}^\beta,A_{2,[7,\infty]}^\beta).
\end{align*}

Write \(A'_{2,[1,2]}\) for the coordinate of the string \((2,[1,2])\) after
the fourth step.  The mutation sequence shown in
Figures~\ref{fig:a2-five-step-mutation-early}
and~\ref{fig:a2-five-step-mutation-late}, together with the minor expression
above, gives
\begin{align*}
 A_{\sigma_1}
 &=\frac{x_0+x_2}{x_1}=z_2,
 &
 A_{\sigma_2}
 &=\frac{y_0+y_2}{y_1}=z_4,\\
 A_{\sigma_3}
 &=\frac{x_2+x_4}{x_3}=z_6,
 &
 A'_{2,[1,2]}
 &=\frac{A_{\sigma_1}x_4y_0+A_{\sigma_3}x_0y_2}{x_2}\\
 &&&=t_2(z_2z_5z_6-z_3z_4z_6-z_2+z_6).
\end{align*}
Finally, mutating at \((2,[3,4])\) gives
\begin{align*}
 A_{\sigma_4}
 &=\frac{A_{\sigma_1}A_{\sigma_2}y_3
          +A'_{2,[1,2]}}{y_2}\\
 &=z_2z_4z_7-z_2z_5z_6+z_2-z_6.
\end{align*}
In terms of the initial Shen--Weng coordinates alone, this becomes the
subtraction-free expression
\[
 A_{\sigma_4}
 =\frac{(x_0+x_2)(y_0+y_2)y_3}{x_1y_1y_2}
  +\frac{(x_0+x_2)x_4y_0}{x_1x_2y_2}
  +\frac{(x_2+x_4)x_0}{x_2x_3}.
\]

\printbibliography

@misc{AsplundEtAlDecompositions2025,
  author        = {Asplund, Johan and Capovilla-Searle, Orsola and
                   Hughes, James and Leverson, Caitlin and Li, Wenyuan and
                   Wu, Angela},
  title         = {Decompositions of augmentation varieties via weaves and
                   rulings},
  year          = {2025},
  eprint        = {2508.20226},
  archivePrefix = {arXiv},
  primaryClass  = {math.SG}
}

@article{BerensteinZelevinsky2001,
  author  = {Berenstein, Arkady and Zelevinsky, Andrei},
  title   = {Tensor product multiplicities, canonical bases and totally
             positive varieties},
  journal = {Invent. Math.},
  volume  = {143},
  number  = {1},
  pages   = {77--128},
  year    = {2001},
  doi     = {10.1007/s002220000102},
  eprint        = {math/9912012},
  archivePrefix = {arXiv}
}

@article{BerensteinFominZelevinsky2005,
  author  = {Berenstein, Arkady and Fomin, Sergey and Zelevinsky, Andrei},
  title   = {Cluster algebras. {III}: {U}pper bounds and double {B}ruhat cells},
  journal = {Duke Math. J.},
  volume  = {126},
  number  = {1},
  pages   = {1--52},
  year    = {2005},
  doi     = {10.1215/S0012-7094-04-12611-9},
  eprint        = {math/0305434},
  archivePrefix = {arXiv}
}

@article{CasalsGorskyGorskySimental2020,
  author  = {Casals, Roger and Gorsky, Eugene and Gorsky, Mikhail and
             Simental, Jos\'e},
  title   = {Algebraic weaves and braid varieties},
  journal = {Am. J. Math.},
  volume  = {146},
  number  = {6},
  pages   = {1469--1576},
  year    = {2024},
  doi     = {10.1353/ajm.2024.a944357},
  eprint        = {2012.06931},
  archivePrefix = {arXiv},
  primaryClass  = {math.RT}
}

@article{CasalsEtAlBraidVarieties,
  author  = {Casals, Roger and Gorsky, Eugene and Gorsky, Mikhail and
             Le, Ian and Shen, Linhui and Simental, Jos\'e},
  title   = {Cluster structures on braid varieties},
  journal = {J. Am. Math. Soc.},
  volume  = {38},
  number  = {2},
  pages   = {369--479},
  year    = {2025},
  doi     = {10.1090/jams/1048},
  eprint        = {2207.11607},
  archivePrefix = {arXiv},
  primaryClass  = {math.AG}
}

@misc{CasalsEtAlComparing2025,
  author        = {Casals, Roger and Galashin, Pavel and Gorsky, Mikhail and
                   Shen, Linhui and Sherman-Bennett, Melissa and
                   Simental, Jos\'e},
  title         = {Comparing cluster algebras on braid varieties},
  year          = {2025},
  eprint        = {2508.03816},
  archivePrefix = {arXiv},
  primaryClass  = {math.AG}
}

@article{CasalsWeng2024,
  author  = {Casals, Roger and Weng, Daping},
  title   = {Microlocal theory of {L}egendrian links and cluster algebras},
  journal = {Geom. Topol.},
  volume  = {28},
  number  = {2},
  pages   = {901--1000},
  year    = {2024},
  doi     = {10.2140/gt.2024.28.901},
  eprint        = {2204.13244},
  archivePrefix = {arXiv}
}

@article{CasalsZaslow2022,
  author  = {Casals, Roger and Zaslow, Eric},
  title   = {{L}egendrian weaves: {\(N\)}-graph calculus, flag moduli and
             applications},
  journal = {Geom. Topol.},
  volume  = {26},
  number  = {8},
  pages   = {3589--3745},
  year    = {2022},
  doi     = {10.2140/gt.2022.26.3589},
  eprint        = {2007.04943},
  archivePrefix = {arXiv}
}

@article{DimofteGabellaGoncharov2016,
  author  = {Dimofte, Tudor and Gabella, Maxime and Goncharov, Alexander B.},
  title   = {{\(K\)}-decompositions and 3d gauge theories},
  journal = {J. High Energy Phys.},
  volume  = {2016},
  number  = {11},
  eid     = {151},
  pagetotal = {147},
  year    = {2016},
  doi     = {10.1007/JHEP11(2016)151},
  eprint        = {1301.0192},
  archivePrefix = {arXiv}
}

@article{FockGoncharov2006,
  author  = {Fock, Vladimir and Goncharov, Alexander},
  title   = {Moduli spaces of local systems and higher {T}eichm\"uller theory},
  journal = {Publ. Math., Inst. Hautes \'Etud. Sci.},
  volume  = {103},
  pages   = {1--211},
  year    = {2006},
  doi     = {10.1007/s10240-006-0039-4},
  eprint        = {math/0311149},
  archivePrefix = {arXiv}
}

@article{GaroufalidisGoernerZickert2015,
  author  = {Garoufalidis, Stavros and Goerner, Matthias and
             Zickert, Christian},
  title   = {Gluing equations for {\(\mathrm{PGL}(n,\mathbb{C})\)}-representations
             of 3-manifolds},
  journal = {Algebr. Geom. Topol.},
  volume  = {15},
  number  = {1},
  pages   = {565--622},
  year    = {2015},
  doi     = {10.2140/agt.2015.15.565},
  eprint        = {1207.6711},
  archivePrefix = {arXiv}
}

@article{GaroufalidisThurstonZickert2015,
  author  = {Garoufalidis, Stavros and Thurston, Dylan P. and
             Zickert, Christian K.},
  title   = {The complex volume of {\(\mathrm{SL}(n,\mathbb{C})\)}-representations
             of 3-manifolds},
  journal = {Duke Math. J.},
  volume  = {164},
  number  = {11},
  pages   = {2099--2160},
  year    = {2015},
  doi     = {10.1215/00127094-3121185},
  eprint        = {1111.2828},
  archivePrefix = {arXiv}
}

@article{GaoShenWeng2024,
  author  = {Gao, Honghao and Shen, Linhui and Weng, Daping},
  title   = {Augmentations, fillings, and clusters},
  journal = {Geom. Funct. Anal.},
  volume  = {34},
  number  = {3},
  pages   = {798--867},
  year    = {2024},
  doi     = {10.1007/s00039-024-00673-y},
  eprint        = {2008.10793},
  archivePrefix = {arXiv}
}

@article{Zickert2020,
  author  = {Zickert, Christian K.},
  title   = {{F}ock--{G}oncharov coordinates for rank two {L}ie groups},
  journal = {Math. Z.},
  volume  = {294},
  number  = {1-2},
  pages   = {251--286},
  year    = {2020},
  doi     = {10.1007/s00209-019-02307-8},
  eprint        = {1605.08297},
  archivePrefix = {arXiv}
}

@misc{GalashinLamShermanBennettSpeyer2022,
  author        = {Galashin, Pavel and Lam, Thomas and
                   Sherman-Bennett, Melissa and Speyer, David},
  title         = {Braid variety cluster structures, {I}: 3{D} plabic graphs},
  year          = {2022},
  eprint        = {2210.04778},
  archivePrefix = {arXiv},
  primaryClass  = {math.CO},
  label         = {GLSBS}
}

@article{GalashinLamShermanBennett2023,
  author  = {Galashin, Pavel and Lam, Thomas and
             Sherman-Bennett, Melissa},
  title   = {Braid variety cluster structures. {II}: General type},
  journal = {Invent. Math.},
  volume  = {243},
  number  = {3},
  pages   = {1079--1127},
  year    = {2026},
  doi     = {10.1007/s00222-025-01390-5},
  eprint        = {2301.07268},
  archivePrefix = {arXiv},
  label   = {GLSB}
}

@article{GorskyKimScrogginSimental2025,
  author  = {Gorsky, Eugene and Kim, Soyeon and Scroggin, Tonie and
             Simental, Jos\'e},
  title   = {Splicing braid varieties},
  journal = {Can. Math. Commun.},
  volume  = {1},
  eid     = {e4},
  pagetotal = {48},
  year    = {2026},
  doi     = {10.4153/S2976859426100034},
  eprint        = {2505.08211},
  archivePrefix = {arXiv}
}

@incollection{Lusztig1994,
  author    = {Lusztig, George},
  title     = {Total positivity in reductive groups},
  booktitle = {Lie Theory and Geometry},
  editor    = {Brylinski, Jean-Luc and others},
  series    = {Prog. Math.},
  number    = {123},
  pages     = {531--568},
  publisher = {Birkh\"auser Boston},
  year      = {1994},
  doi       = {10.1007/978-1-4612-0261-5_20}
}

@article{Muller2013,
  author  = {Muller, Greg},
  title   = {Locally acyclic cluster algebras},
  journal = {Adv. Math.},
  volume  = {233},
  number  = {1},
  pages   = {207--247},
  year    = {2013},
  doi     = {10.1016/j.aim.2012.10.002},
  eprint        = {1111.4468},
  archivePrefix = {arXiv}
}

@article{ShenWeng2021,
  author  = {Shen, Linhui and Weng, Daping},
  title   = {Cluster structures on double {B}ott--{S}amelson cells},
  journal = {Forum Math. Sigma},
  volume  = {9},
  eid     = {e66},
  pagetotal = {89},
  year    = {2021},
  doi     = {10.1017/fms.2021.59},
  eprint        = {1904.07992},
  archivePrefix = {arXiv}
}

\bigskip

\noindent
\textsc{Yuma Mizuno, School of Mathematical Sciences, University College Cork, Western Road,
Cork, Ireland.}\par
\noindent
\textit{Email address}: \texttt{mizuno.y.aj@gmail.com}

\end{document}